\documentclass[a4paper]{amsart}
\pdfoutput=1
\usepackage[utf8]{inputenc}
\usepackage[T1]{fontenc}
\usepackage{lmodern}
\usepackage{amssymb,amsxtra}
\usepackage{graphicx}
\usepackage{mathabx}
\usepackage{fancybox}
\usepackage{nicefrac,mathtools}
\usepackage[lite]{amsrefs}
\usepackage[toc,page]{appendix}

\newcommand*{\arxiv}[1]{\href{http://www.arxiv.org/abs/#1}{arXiv: #1}}

\renewcommand{\PrintDOI}[1]{\href{http://dx.doi.org/\detokenize{#1}}{doi: \detokenize{#1}}}

\BibSpec{book}{%
  +{}  {\PrintPrimary}                {transition}
  +{,} { \textit}                     {title}
  +{.} { }                            {part}
  +{:} { \textit}                     {subtitle}
  +{,} { \PrintEdition}               {edition}
  +{}  { \PrintEditorsB}              {editor}
  +{,} { \PrintTranslatorsC}          {translator}
  +{,} { \PrintContributions}         {contribution}
  +{,} { }                            {series}
  +{,} { \voltext}                    {volume}
  +{,} { }                            {publisher}
  +{,} { }                            {organization}
  +{,} { }                            {address}
  +{,} { \PrintDateB}                 {date}
  +{,} { }                            {status}
  +{}  { \parenthesize}               {language}
  +{}  { \PrintTranslation}           {translation}
  +{;} { \PrintReprint}               {reprint}
  +{.} { }                            {note}
  +{.} {}                             {transition}
  +{} { \PrintDOI}                   {doi}
  +{} { available at \url}            {eprint}
  +{}  {\SentenceSpace \PrintReviews} {review}
}

\usepackage{enumitem} 
\setlist[enumerate,1]{label=\textup{(\arabic*)}}

\usepackage{tikz}
\usetikzlibrary{matrix,calc,arrows,decorations.pathmorphing}
\tikzset{node distance=2cm, auto}

\tikzset{cd/.style=matrix of math nodes,row sep=2em,column sep=2em, text height=1.5ex, text depth=0.5ex}
\tikzset{cdar/.style=->,auto}
\tikzset{mid/.style={anchor=mid}} 
\tikzset{narrowfill/.style={inner sep=1pt, fill=white}}
\tikzset{rndblock/.style={rounded corners,rectangle,draw,outer sep=0pt}}

\usepackage{tikz-cd}

\usepackage{microtype}
\usepackage[pdftitle={Essential crossed products for inverse
  semigroup actions: simplicity and pure infiniteness},
pdfauthor={Bartosz Kwa\'sniewski, Ralf Meyer},
pdfsubject={Mathematics}]{hyperref}

\theoremstyle{plain}
\newtheorem{theorem}{Theorem}
\newtheorem{lemma}[theorem]{Lemma}

\newtheorem{proposition}[theorem]{Proposition}
\newtheorem{corollary}[theorem]{Corollary}
\theoremstyle{definition}
\newtheorem{definition}[theorem]{Definition}
\theoremstyle{remark}
\newtheorem{remark}[theorem]{Remark}
\newtheorem{example}[theorem]{Example}

\numberwithin{equation}{section}

\DeclareMathOperator{\Aut}{Aut}
\DeclareMathOperator{\Prim}{Prim}
\DeclareMathOperator{\Bis}{Bis}
\DeclareMathOperator{\Fix}{Fix}
\DeclareMathOperator{\Int}{Int}
\newcommand{\twolinesubscript}[2]{\genfrac{}{}{0pt}{}{#1}{#2}}

\newcommand{\pHom}{\Pi}

\newcommand{\D}{X}
\newcommand{\haction}{h}

\newcommand*{\nb}{\nobreakdash}
\newcommand*{\Star}{\(^*\)\nobreakdash-}

\newcommand*{\C}{\mathbb C}
\newcommand*{\Z}{\mathbb Z}
\newcommand*{\Q}{\mathbb Q}
\newcommand*{\R}{\mathbb R}
\newcommand*{\N}{\mathbb N}
\newcommand*{\T}{\mathbb T}

\newcommand*{\Ideals}{\mathbb I}
\newcommand*{\Null}{\mathcal N}
\newcommand*{\lNull}{\mathcal{LN}}
\newcommand*{\rNull}{\mathcal{RN}}
\newcommand*{\Bound}{\mathbb B}
\newcommand*{\Comp}{\mathbb K}
\newcommand*{\Mat}{\mathbb M}
\newcommand{\Her}{\mathbb H}
\newcommand*{\red}{\mathrm r}
\newcommand*{\ess}{\mathrm{ess}}
\newcommand*{\sing}{\mathrm{sing}}
\newcommand*{\alg}{\mathrm{alg}} 

\newcommand*{\Cst}{\textup C^*}
\newcommand*{\Mult}{\mathcal M}
\newcommand*{\Locmult}{\mathcal{M}_\mathrm{loc}}

\newcommand*{\Cont}{\textup C}
\newcommand*{\Contb}{\textup C_\textup b}
\newcommand*{\Contc}{\Cont_\textup c} 

\newcommand{\idealin}{\mathrel{\triangleleft}} 
\newcommand*{\Slice}{\mathcal S}
\newcommand*{\Id}{\textup{Id}}

\newcommand*{\Hils}{\mathcal H}
\newcommand*{\Hilm}[1][E]{\mathcal #1}

\newcommand*{\B}{\mathcal B}
\newcommand*{\A}{\mathcal A}
\renewcommand*{\L}{\mathcal L}
\newcommand*{\defeq}{\mathrel{\vcentcolon=}}
\newcommand*{\eqdef}{\mathrel{=\vcentcolon}}
\newcommand*{\congto}{\xrightarrow\sim}

\DeclarePairedDelimiter{\abs}{\lvert}{\rvert}
\DeclarePairedDelimiter{\norm}{\lVert}{\rVert}
\DeclarePairedDelimiterX{\braket}[2]{\langle}{\rangle}{#1\,\delimsize\vert\,\mathopen{}#2}
\DeclarePairedDelimiterX{\BRAKET}[2]{\langle}{\rangle}{\!\delimsize\langle#1\,\delimsize\vert\,\mathopen{}#2\delimsize\rangle\!}
\DeclarePairedDelimiterX{\setgiven}[2]{\{}{\}}{#1\,{:}\,\mathopen{}#2}

\newcommand*{\dual}[1]{\widehat{#1}}

\newcommand*{\s}{s}
\newcommand*{\rg}{r}

\DeclareMathOperator*{\stlim}{s-lim}

\newcommand*{\into}{\rightarrowtail}
\newcommand*{\onto}{\twoheadrightarrow}

\begin{document}
\title[Essential crossed products for inverse semigroup actions]{Essential
  crossed products for inverse semigroup actions: simplicity and pure
  infiniteness}

\author{Bartosz Kosma Kwa\'sniewski}
\email{bartoszk@math.uwb.edu.pl}
 \address{Faculty of Mathematics\\
   University  of Bia\l ystok\\
   ul.\@ K.~Cio\l kowskiego 1M\\
   15-245 Bia\l ystok\\
   Poland}

\author{Ralf Meyer}
\email{rmeyer2@uni-goettingen.de}
\address{Mathematisches Institut\\
 Universit\"at G\"ottingen\\
 Bunsenstra\ss e 3--5\\
 37073 G\"ottingen\\
 Germany}

\begin{abstract}
  We study simplicity and pure infiniteness criteria for
  \(\Cst\)-algebras associated to inverse semigroup actions by
  Hilbert bimodules and to Fell bundles over \'etale not necessarily
  Hausdorff groupoids.  Inspired by recent work of Exel and Pitts,
  we introduce essential crossed
  products
  for which there are such criteria.  In our approach the major
  role is played by a generalised expectation with values in the
  local multiplier algebra.  We give a long list of equivalent
  conditions characterising when the essential and reduced
  \(\Cst\)-algebras coincide.  Our most general simplicity and pure
  infiniteness criteria apply to aperiodic \(\Cst\)-inclusions
  equipped with supportive generalised expectations.  We thoroughly
  discuss the relationship between aperiodicity, detection of
  ideals, purely outer inverse semigroup actions, and non-triviality
  conditions for dual groupoids.
\end{abstract}
\thanks{BKK was partially supported by the National Science Center
  (NCN), Poland, grant no.~2019/35/B/ST1/02684}
\maketitle
\setcounter{tocdepth}{1}
\tableofcontents

\section{Introduction}
\label{sec:introduction}

Much is known about the ideal structure of reduced crossed products
for group actions and of reduced groupoid \(\Cst\)\nb-algebras of étale,
Hausdorff, locally compact groupoids.  More precisely, for an
action~\(\alpha\) of a discrete group~\(G\) on a separable
\(\Cst\)\nb-algebra~\(A\), there is a  long list of equivalent
conditions due to Kishimoto and Olesen--Pedersen, which imply that the
coefficient algebra~\(A\) \emph{detects ideals} in the reduced crossed
product \(A\rtimes_{\alpha,\red} G\) in the sense that \(J\cap A=0\) for an
ideal~\(J\) in \(A\rtimes_{\alpha,\red} G\) implies \(J=0\).
This theory has recently been generalised
in~\cite{Kwasniewski-Meyer:Aperiodicity} to actions by Hilbert
bimodules or, equivalently, Fell bundles over groups.  The goal of
this article is to generalise the results
in~\cite{Kwasniewski-Meyer:Aperiodicity} to Fell bundles over
groupoids and inverse semigroups.  It turns out that many of these
results extend, without major problems, to actions of inverse
semigroups by Hilbert bimodules provided that the crossed product
carries a conditional expectation.  Such crossed products model, as
a special case, section \(\Cst\)-algebras associated to Fell bundles
over Hausdorff, étale, locally compact groupoids.  The existence of
a canonical conditional expectation is closely related to the
Hausdorffness of the underlying groupoid.  This is already quite
satisfactory, as it unifies the existing results for crossed
products for group actions and for (twisted) groupoid
\(\Cst\)\nb-algebras of Hausdorff étale groupoids and improves the
criteria for detection of ideals for groupoid crossed products by
Renault~\cite{Renault:Ideal_structure}.  Another intriguing source
of potential applications come from regular inclusions and the
noncommutative Cartan inclusions introduced by
Exel~\cite{Exel:noncomm.cartan}.  Namely,
Exel~\cite{Exel:noncomm.cartan} has found sufficient conditions for
a regular inclusion with a conditional expectation to be of the form
\(A\subseteq A\rtimes_\red S\) for some inverse semigroup action by
Hilbert bimodules.  However, a real challenge and an important
problem that has been open for decades now, is how to deal with
non-Hausdorff groupoids.  Eventually, the main point of this article
has become how to attack this problem and get results about
detection of ideals without a (genuine) conditional expectation.

Difficulties for non-Hausdorff
groupoids were already noticed in~\cite{Renault:Ideal_structure},
which includes an unexpected example by Skandalis of a minimal
foliation with a non-simple \(\Cst\)\nb-algebra.  Reduced crossed
products for non-Hausdorff groupoids were studied further in
\cites{Khoshkam-Skandalis:Regular,
  Khoshkam-Skandalis:Crossed_inverse_semigroup}, but without
progress on their ideal structure.  Until recently, all general
results about the ideal structure of groupoid \(\Cst\)\nb-algebras
were limited to the Hausdorff case.  There are several constructions
of étale groupoids from other data for which Hausdorffness is
unclear and not a natural assumption to make.  Important classes of
such groupoids are foliation \(\Cst\)\nb-algebras -- restricted to
complete transversals to make them étale -- and the
\(\Cst\)\nb-algebras of self-similar graphs defined by Exel and
Pardo~\cite{Exel-Pardo:Self-similar}.  Very recently, there has been
progress in the non-Hausdorff case in two directions.  First, the
simplicity of~\(\Cst_\red(H)\) for minimal, topologically principal,
second countable groupoids~\(H\) has been studied
in~\cite{Clark-Exel-Pardo-Sims-Starling:Simplicity_non-Hausdorff}.
 Secondly, Exel and
Pitts~\cite{Exel-Pitts:Weak_Cartan} have defined ``essential''
(twisted) groupoid \(\Cst\)\nb-algebras for topologically principal
groupoids in which \(\Cont_0(X)\) detects ideals.  The construction
in~\cite{Exel-Pitts:Weak_Cartan} appears to
be rather \emph{ad hoc}, however.  It only works well for
topologically principal groupoids.

Here we take the idea of Exel and Pitts much further.  We define
essential crossed products for all inverse semigroup actions on
\(\Cst\)\nb-algebras by Hilbert bimodules.  Our definition is
conceptual and analogous to the definition of reduced crossed
products in~\cite{Buss-Exel-Meyer:Reduced}.  Secondly, we derive a
very powerful abstract criterion for a \(\Cst\)\nb-subalgebra~\(A\)
in a \(\Cst\)\nb-algebra~\(B\) to detect ideals.  Even more, our
theory may show that~\(A\) supports~\(B\) in the sense that any
non-zero positive element in~\(B\) dominates some non-zero positive
element in~\(A\) with respect to the Cuntz preorder.  This allows to
prove that~\(B\) is purely infinite and simple under suitable
assumptions.  Previous criteria for pure infiniteness of crossed
products were restricted to groupoid \(\Cst\)\nb-algebras of
Hausdorff groupoids (see also
\cites{Jolissaint-Robertson:Simple_purely_infinite,
  Rordam-Sierakowski:Purely_infinite,
  Pasnicu-Phillips:Spectrally_free,
  Kirchberg-Sierakowski:Strong_pure,
  Giordano-Sierakowski:Purely_infinite,
  Kwasniewski-Meyer:Aperiodicity,
	Kwasniewski:Crossed_products,
  Kwasniewski-Szymanski:Pure_infinite} for more work on pure
infiniteness of crossed products).

The reduced crossed product \(A\rtimes_\red S\) for an inverse semigroup
action on~\(A\) is defined in~\cite{Buss-Exel-Meyer:Reduced} using a
weak conditional expectation \(E\colon A\rtimes S \to A''\),
where~\(A''\) is the bidual of~\(A\).  We call the action
\emph{closed} if this expectation takes values in
\(A\subseteq A''\).  This happens for inverse semigroup actions that
are derived from actions of Hausdorff étale groupoids.  In general,
however, we must enlarge~\(A\) to accommodate a conditional
expectation.  Let \(\Locmult(A)\) be the local multiplier algebra
of~\(A\), that is, the inductive limit of the multiplier algebras of
the essential ideals in~\(A\)
(see~\cite{Ara-Mathieu:Local_multipliers}).  If
\(A=\Cont_0(X)\), then \(\Locmult(A)\) is the inductive limit of
\(\Contb(U)\) for dense open subsets \(U\subseteq X\).  So it is spanned
by functions that are ``densely defined'' on~\(X\).
Our main idea is to  replace \(E\colon A\rtimes S \to A''\) by a
generalised conditional expectation
\(EL\colon A\rtimes S \to \Locmult(A)\), which we briefly call an
\emph{\(\Locmult\)-expectation}.
Then we define the essential crossed
product as the quotient of \(A\rtimes S\) by the largest
two-sided ideal on which~\(EL\) vanishes.
The main troublemakers in~\(\Cst_\red(H)\) for a non-Hausdorff
groupoid~\(H\) are elements \(x\in\Cst_\red(H)\) for which~\(E(x)\)
is supported on a nowhere dense subset of~\(X\).  The
expectation~\(EL\) kills such elements.  Hence they get killed in
the essential crossed product.

The main achievement in this article is a conceptual understanding of
the techniques used to prove that crossed products of various kinds are
simple or purely infinite.  We take this occasion to honour Emmy
Noether, who pioneered the conceptual approach to mathematics despite
strong resistance, and whose Habilitation in Göttingen was finally
granted only in 1919.  The basic concepts that make this paper work
are aperiodicity and \(\Locmult\)-expectations.

The concept of aperiodicity goes back to Kishimoto's proof that
reduced crossed products for outer group actions on simple
\(\Cst\)\nb-algebras are again simple
(see~\cite{Kishimoto:Outer_crossed}).  The crucial condition in
Kishimoto's proof was rewritten
in~\cite{Kwasniewski-Meyer:Aperiodicity}
in terms of normed \(A\)\nb-bimodules in order
to treat actions by Hilbert bimodules instead of by automorphisms.
The right generality to study aperiodicity is a
\(\Cst\)\nb-inclusion \(A\subseteq B\).  We call the inclusion
\emph{aperiodic} if~\(B/A\) equipped with the quotient norm and the
induced \(A\)\nb-bimodule structure satisfies Kishimoto's condition.
Given an aperiodic inclusion, we prove that there is a maximal
ideal~\(\Null\) in~\(B\) which is aperiodic as an \(A\)\nb-bimodule
and that~\(B/\Null\) is the unique quotient of~\(B\)
for which the map \(A\to B/\Null\) is injective and detects ideals.

For inverse semigroup actions by Hilbert bimodules, the aperiodicity
of~\((A\rtimes S)/A\) is equivalent to the aperiodicity of certain
Hilbert bimodules.  Using the results
of~\cite{Kwasniewski-Meyer:Aperiodicity} we show  the following.
Let~\(A\) be separable or of Type~I.  The inclusion
\(A\subseteq A\rtimes S\) is aperiodic if and only if the dual
groupoid \(\dual{A}\rtimes S\) is topologically free (see Theorem \ref{the:aperiodic_top_non-trivial}).
Here~\(\dual{A}\) is the space of isomorphism classes of irreducible
representations of~\(A\), and the action of~\(S\) on~\(A\) by
Hilbert bimodules induces an action on~\(\dual{A}\) by the Rieffel
correspondence.  We carefully discuss the concept of a
``topologically free'' étale groupoid in Section~\ref{sec:etale}
because several slightly different definitions are used in the
literature, and the dual groupoid \(\dual{A}\rtimes S\) is rather
badly non-Hausdorff.

Aperiodicity becomes powerful when combined with a conditional
expectation.  The following theorem is a special case of
Theorem~\ref{the:aperiodic_consequences}.

\begin{theorem}
  \label{the:aperiodic_intro}
  Let \(A\subseteq B\) be an aperiodic \(\Cst\)\nb-inclusion with a
  conditional expectation \(E\colon B\to A\).  Let~\(\Null_E\) be the
  largest two-sided ideal contained in \(\ker E\).  Let
  \(A^+\defeq \setgiven{a\in A}{a\ge0}\) and let~\(\precsim\) denote
  the Cuntz preorder on~\(B^+\).  Then
  \begin{enumerate}
  \item \label{enu:aperiodic_intro1}%
    for every \(b\in B^+\) with \(b\notin\Null_E\), there is
    \(a\in A^+\setminus\{0\}\) with \(a \precsim b\);
  \item \label{enu:aperiodic_intro2}%
    \(A\) supports~\(B/\Null_E\);
  \item \label{enu:aperiodic_intro4}%
    \(A\) detects ideals in~\(B/\Null_E\), and~\(B/\Null_E\) is the
    only quotient of~\(B\) with this property;
  \item \label{enu:aperiodic_intro5}%
    \(B\) is simple if and only if \(\Null_E=0\) and
    \(\overline{BIB}=B\) for all \(0\neq I\in \Ideals(A)\);
  \item \label{enu:aperiodic_intro6}%
    if~\(B\) is simple, then~\(B\) is purely infinite if and only if
    every element in \(A^+\setminus\{0\}\) is infinite in~\(B\).
  \end{enumerate}
\end{theorem}
\numberwithin{theorem}{section}

The first statement~\ref{enu:aperiodic_intro1} is the key step here.
It easily implies all the others.

In order to treat general inverse semigroup crossed products or
\(\Cst\)\nb-algebras of non-Hausdorff étale groupoids, it is
necessary to replace the conditional expectation in
Theorem~\ref{the:aperiodic_intro} by something weaker.  At first, we
tried a weak conditional expectation \(E\colon B\to A''\) as in the
definition of the reduced crossed product for inverse semigroup
actions in~\cite{Buss-Exel-Meyer:Reduced}.  An inspection of the
proof of Theorem~\ref{the:aperiodic_intro} led to the concept of a
supportive weak conditional expectation, which suffices to make the
proof of the theorem work.  Besides looking at weak conditional
expectations, we also looked at pseudo-expectations, which take
values in the injective hull of~\(A\).  Here we were motivated by
the theorem of Zarikian that a crossed product for a group action
has a unique pseudo-expectation if and only if the action is
aperiodic (see \cite{Zarikian:Unique_expectations}*{Theorem~3.5}).
The injective hull of a commutative \(\Cst\)\nb-algebra is equal to
its local multiplier algebra
(see~\cite{Frank:Injective_local_multiplier}).  And it turns out
that a generalised
conditional expectation with values in \(\Locmult(A)\) is always
supportive.  Thus Theorem~\ref{the:aperiodic_intro} still holds with
a generalised conditional expectation~\(E\) that takes values in
\(\Locmult(A)\).  And this then suggests our definition of the
essential crossed product.

As another test of our definition of the essential crossed product,
we carry over the main results of
Archbold--Spielberg~\cite{Archbold-Spielberg:Topologically_free} and
Kawamura--Tomiyama~\cite{Kawamura-Tomiyama:Properties_dynamical}.
Namely, if~\(S\) is an inverse semigroup acting on a
\(\Cst\)\nb-algebra~\(A\) by Hilbert bimodules and
\(\dual{A}\rtimes S\) is topologically free, then \(A\) detects
ideals in~\(A\rtimes_\ess S\)
(Theorem~\ref{the:top_free_uniqueness}).  And for an étale locally
compact groupoid~\(H\), \(\Cont_0(X)\) detects ideals
in~\(\Cst_\ess(H)\) if \emph{and only if}~\(H\) is topologically
free (Theorem~\ref{thm:detection_in_etale_groupoids}).  Another very
promising observation is that all derivations on~\(A\) become inner
in~\(\Locmult(A)\) (see~\cite{Ara-Mathieu:Local_multipliers}).  The
corresponding result for~\(A''\) instead of~\(\Locmult(A)\) plays a
key role in the work of Olesen--Pedersen.  The local multiplier
algebra always embeds into Hamana's injective hull
(see~\cite{Frank:Injective_local_multiplier}).  Hence every
\(\Locmult\)-expectation is a pseudo-expectation as well.  These
have been studied, for instance, in \cites{Pitts:Regular_I,
  Pitts-Zarikian:Unique_pseudoexpectation,
  Zarikian:Unique_expectations}.

It is, however, often necessary to work with the reduced crossed
product.  One reason is that the essential
crossed product is not functorial.  Thus we want to know when the
essential and the reduced crossed products are equal, meaning that
they are both quotients of~\(A\rtimes S\) by the same ideal.
We use that~\(E(A\rtimes S) \) embeds into the product
\(\prod_{\pi\in\dual{A}} \Bound(\Hils_\pi)\), consisting of
uniformly bounded families of operators on the Hilbert spaces on
which the irreducible representations of~\(A\) act.
The local multiplier algebra~\(\Locmult(A)\) embeds
into the quotient of \(\prod_{\pi\in\dual{A}} \Bound(\Hils_\pi)\)
that is defined by the \emph{essential} supremum of the pointwise
norms -- we disregard subsets that are meagre.
As a result, the following are equivalent (see
Corollary~\ref{cor:ess_vs_reduced}):
\begin{itemize}
\item \(A\rtimes_\red S = A\rtimes_\ess S\);
\item if \(x\in (A\rtimes S)^+\) and \(E(x)\neq0\), then the set of
  \(\pi\in\dual{A}\) with \(\norm{\pi''(E(x))}\neq0\) is not meagre;
\item if \(\varepsilon >0\) and \(x\in (A\rtimes S)^+\) satisfies
  \(E(x)\neq0\), then the set of \(\pi\in\dual{A}\) with
  \(\norm{\pi''(E(x))}> \varepsilon \) has non-empty interior.
\end{itemize}

We carry the theory above over to section \(\Cst\)\nb-algebras
\(\Cst(H,\A)\) for Fell bundles~\(\A\) over a locally compact, étale
groupoid~\(H\) with Hausdorff unit space~\(X\).  This includes
twisted groupoid \(\Cst\)\nb-algebras as a special case.  We define
an essential section \(\Cst\)\nb-algebra \(\Cst_\ess(H,\A)\) for all
Fell bundles~\(\A\).  For a Fell line bundle~\(\A\),  it coincides
with the essental twisted groupoid \(\Cst\)\nb-algebra defined by
Pitts and Exel~\cite{Exel-Pitts:Weak_Cartan} if and only if~\(H\) is
topologically principal.  We give many equivalent characterisations
for \(\Cst_\red(H,\A) = \Cst_\ess(H,\A)\).  One of them is related
to singular elements as defined
in~\cite{Clark-Exel-Pardo-Sims-Starling:Simplicity_non-Hausdorff}
for groupoid \(\Cst\)\nb-algebras.  We also discuss simplicity and
pure infiniteness criteria for \(\Cst_\ess(H,\A)\) that generalise
and improve various results of this sort.

Our more recent
article~\cite{Kwasniewski-Meyer:Aperiodicity_pseudo_expectations}
contains two important advances.  First, all topologically free
inverse semigroup actions are aperiodic.  Secondly,
pseudo-expectations are always supportive.  This generalises
Theorem~\ref{the:aperiodic_intro} to all aperiodic inclusions
\(A\subseteq B\).

Steinberg and Szak\'acs in~\cite{Steinberg-Szakacs:Simplicity} prove
a criterion when the Steinberg algebra of an \'etale groupoid with
totally disconnected object space is simple.  Our results imply an
analogous criterion for reduced groupoid \(\Cst\)\nb-algebras.

The paper is organised as follows.  Section~\ref{sec:etale}
introduces inverse semigroup actions on topological spaces and
\(\Cst\)\nb-algebras and the dual groupoid for an action on a
\(\Cst\)\nb-algebra, and it compares several concepts of topological
freeness for such actions and for non-Hausdorff étale groupoids.
Section~\ref{sec:prelim} discusses basic notation about generalised
conditional expectations and full and reduced crossed products for
inverse semigroup actions.  We show that the canonical weak
conditional expectation on the reduced crossed product is faithful.
Section~\ref{sec:essential_crossed} introduces the essential crossed
product and the \(\Locmult\)-expectation that
defines it.  We prove that this generalised expectation is faithful,
and we characterise when the reduced and essential crossed products
coincide.

Section~\ref{sec:hidden_ideal} contains our general results on
aperiodic \(\Cst\)\nb-inclusions \(A\subseteq B\).
We show that for any such inclusion there is a unique
quotient of~\(B\) in which~\(A\) detects ideals.  Then we say that
\(A\subseteq B\) has the generalised intersection property and call
the unique ideal~\(\Null\) such that~\(A\) detects ideals
in~\(B/\Null\) the hidden ideal.  We define supportive generalised
expectations and prove the generalisations of
Theorem~\ref{the:aperiodic_intro} discussed above.
In Section~\ref{sec:aperiodic_isg}, we specialise aperiodicity to
inverse semigroup actions.  We show that the inclusion
\(A\subseteq B\) into an exotic crossed product is aperiodic if and
only if the underlying action is aperiodic in a suitable sense, and
we reformulate aperiodicity in several equivalent ways.  We use
these equivalent characterisations of aperiodicity to rewrite our
main results with different assumptions.
In Section~\ref{sec:groupoids}, we describe section algebras of Fell
bundles over étale, locally compact groupoids through crossed
products for inverse semigroup actions.  Using  this we define
essential section algebras for Fell bundles over étale groupoids.
We
prove results about their ideal structure and when they are simple
and purely infinite.  And we compare the essential and reduced
section algebras.

\numberwithin{equation}{subsection}

\section{Preliminaries on inverse semigroup actions and étale groupoids}
\label{sec:etale}

First we define inverse semigroup actions on topological spaces and
compare them to étale groupoids.  We allow arbitrary topological
spaces, requiring neither Hausdorffness nor local compactness.  We
define actions of inverse semigroups on \(\Cst\)\nb-algebras by
Hilbert bimodules and relate them to regular inclusions and gradings
by inverse semigroups.  We define the dual groupoid
\(\dual{A}\rtimes S\) of such an action.  Then we discuss several
variants of the concept of topological freeness.

\subsection{Inverse semigroup actions on spaces and étale groupoids}
\label{sec:isg-action_vs_groupoid}

An \emph{inverse semigroup} is a semigroup~\(S\) with the property
that for each \(t\in S\) there is a unique element \(t^* \in S\)
such that \(t t^* t = t\) and \(t^* t t^* = t^*\).  Let
\[
  E(S)\defeq \setgiven{e\in S}{e^2=e}.
\]
If \(e,f\in E(S)\), then \(e = e^*\) and \(e f = f e\).  If
\(t\in S\), then \(t^* t, t t^* \in E(S)\).  We call \(E(S)\) the
\emph{idempotent semilattice} of~\(S\).  A partial order on~\(S\) is
defined by \(t \le u\) for \(t,u \in S\) if and only if
\(t = u t^* t\), if and only if there is \(e\in E(S)\) with
\(t = u e\).  By definition, \(E(S) = \setgiven{e\in S}{e \le 1}\).

Let~\(X\) be an arbitrary topological space.  A partial
homeomorphism of~\(X\) is a homeomorphism between two open subsets
of~\(X\).  Partial homeomorphisms with the composition of partial
maps form a unital inverse semigroup, which we denote by
\(\pHom(X)\).  Let~\(S\) be a unital inverse semigroup.  An
\emph{action of~\(S\) on~\(X\) by partial homeomorphisms} is a
unital semigroup homomorphism \(\haction\colon S\to\pHom(X)\).
So it consists of open subsets~\(\D_t\) of~\(X\) and homeomorphisms
\(\haction_t\colon \D_t\to \D_{t^*}\) for all \(t\in S\), such that
\(\haction_t \circ \haction_u = \haction_{t u}\) for all
\(t,u\in S\) and \(\haction_1= \Id_X\); then
\(\haction_{t^*} = \haction_t^* = \haction_t^{-1}\) for all
\(t\in S\). 
We denote  the domain of \(h_t\) by \(X_t\), rather than by \(X_{t^*}\) (which is a convention adopted in most of sources).
We do this to lighten the notation, as we will talk mostly about domains of \(h_t\)'s.

An inverse semigroup action~\(\haction\) on~\(X\) as above yields an
étale topological groupoid with object space~\(X\), namely, the
\emph{transformation groupoid} \(X\rtimes S\) (see \cite{Paterson:Groupoids}*{p. 140}
or 
\cite{Exel:Inverse_combinatorial}*{Section~4}). We avoid the name
``groupoid of germs'' used by Exel because some authors use that name
for another groupoid with a different germ relation.  The arrows of
\(X\rtimes S\) are equivalence classes of pairs~\((t,x)\) for
\(x\in \D_t\subseteq X\); two pairs \((t,x)\) and~\((t',x')\) are
equivalent if \(x=x'\) and there is \(v\in S\) with \(v\le t, t'\) and
\(x\in \D_v\).  The range and source maps
\(\rg,\s\colon X\rtimes S \rightrightarrows X\) and the multiplication
are defined by \(\rg([t,x])\defeq \haction_t(x)\),
\(\s([t,x])\defeq x\), and
\([t,\haction_u(x)] \cdot [u,x] = [t\cdot u,x]\).  We
give~\(X\rtimes S\) the unique topology for which \([t,x]\mapsto x\)
is a homeomorphism from an open subset of~\(X\rtimes S\) onto~\(\D_t\)
for each \(t\in S\).  Then the range and source maps are local
homeomorphisms \(X\rtimes S \rightrightarrows X\).  The multiplication
is continuous.  So~\(X\rtimes S\) is an étale topological groupoid.
The subsets \(U_t \defeq \setgiven{[t,x]}{x \in \D_t}\) are
bisections of \(X\rtimes S\), and they cover~\(X\rtimes S\).

Now let~\(H\) be an étale groupoid with object space~\(X\).  We are
going to write~\(H\) as a transformation groupoid.  A
\emph{bisection} of~\(H\) is an open subset \(U\subseteq H\) such
that \(\rg|_U\) and~\(\s|_U\) are injective.  If \(U,V\subseteq H\)
are bisections, then so are
\[
  U^{-1}\defeq \setgiven{\gamma^{-1}}{\gamma\in U},\qquad
  U\cdot V\defeq \setgiven{\gamma\cdot\eta}{\gamma\in U,\ \eta\in V}.
\]
The bisections of~\(H\) with these operations form a unital inverse
semigroup, which we denote by~\(\Bis(H)\).  The unit bisection is
the subset of all identity arrows in~\(H\).  The inverse
semigroup~\(\Bis(H)\) acts canonically on~\(X\) by
\(\haction_U\defeq \rg\circ (\s|_U)^{-1}\) for \(U\in \Bis(H)\).

\begin{definition}
  \label{def:wide_isg}
  An inverse subsemigroup \(S\subseteq \Bis(H)\) is called \emph{wide}
  if \(\bigcup S = H\) and \(U\cap V\) is a union of bisections
  in~\(S\) for all \(U,V\in S\).
\end{definition}

\begin{proposition}
  \label{prop:wide_semigroup_isomorphism}
  For any unital inverse subsemigroup \(S\subseteq \Bis(H)\), the map
  \[
    \Phi\colon X\rtimes S \to H,\qquad
    [U,x] \mapsto (\s|_U)^{-1}(x),
  \]
  is a well defined, continuous, open groupoid homomorphism.  It is an
  isomorphism if and only if \(S\subseteq \Bis(H)\) is wide.
\end{proposition}

\begin{proof}
  This is mostly proven in
  \cite{Exel:Inverse_combinatorial}*{Propositions 5.3 and~5.4}.
  Direct computations show that~\(\Phi\) is a well defined groupoid
  homomorphism.  If \(U\in S\), then
  \(U'\defeq \setgiven{[U,x]}{x \in \s(U)}\) is a bisection of  \(X\rtimes S\) with
  \(\Phi(U')= U\).  Since both \(H\) and \(X\rtimes S\) are \'etale
  and~\(\Phi\) is the identity map on objects, it follows
  that~\(\Phi\) restricts to a homeomorphism from~\(U'\) onto~\(U\).
  The bisections~\(U'\) for \(U\in S\) cover~\(X\rtimes S\).
  Therefore, the map~\(\Phi\) is both continuous and open.

  Clearly, the map~\(\Phi\) is surjective if and only if
  \(\bigcup S = H\).  And~\(\Phi\) is not injective if and only if
  there are \(U,V\in S\) and \(\gamma \in U\cap V\) with
  \([U,\s(\gamma)]\neq [V,\s(\gamma)]\) in \(X\rtimes S\).  By the
  definition of \(X\rtimes S\), the latter holds if and only if
  \(\gamma \notin W\) for all \(W\in S\) with
  \(W\subseteq U\cap V\).  Hence there is such a
  \(\gamma \in U\cap V\) if and only if
  \(\bigcup_{W\in S, W\subseteq U\cap V} W \neq U\cap V\).
  So~\(\Phi\) is bijective if and only if~\(S\) is wide.
\end{proof}

Proposition~\ref{prop:wide_semigroup_isomorphism} allows to
translate properties of groupoids into the language of inverse
semigroup actions, and vice versa.

\begin{definition}
  \label{def:invariant_subset_A}
  A subset \(Y\subseteq X\) is \emph{\(\haction\)\nb-invariant} for
  an action \(\haction\colon S\to\pHom(X)\) if
  \(\haction_t(Y\cap \D_t)\subseteq Y\) for all \(t\in S\).
  If~\(Y\) is invariant, then there is a restricted action
  \(\haction|_Y\colon S\to\pHom(Y)\), which is defined by
  \((\haction|_Y)_t \defeq \haction_t|_Y\colon Y\cap \D_t\to Y\cap
  \D_{t^*}\) for all \(t\in S\).  A subset \(Y\subseteq X\) is
  \(H\)\nb-invariant for a groupoid~\(H\) with object space \(X\) if and only if
  \(\s^{-1}(Y) = \rg^{-1}(Y)\) as subsets of~\(H\).
\end{definition}

\begin{remark}
  \label{rem:invariance_transformation_groupoid}
  A subset \(Y\subseteq X\) is \(X\rtimes_\haction S\)\nb-invariant if
  and only if it is \(\haction\)\nb-invariant (see
  \cite{Exel-Pardo:Tight_groupoid}*{Proposition~5.4}).
\end{remark}

\begin{definition}
  \label{def:closed_action_A}
  Let~\(X\) be a topological space and
  \(\haction\colon S\to\pHom(X)\) an inverse semigroup action.  For
  \(t\in S\), define
  \[
    \D_{1,t}\defeq\bigcup_{e\le t, e\in E(S)} \D_e.
  \]
  The action is called \emph{closed} if~\(\D_{1,t}\) is relatively
  closed in~\(\D_t\) for all \(t\in S\).
\end{definition}

\begin{lemma}
  \label{lem:closed_groupoid_vs_action}
  The action~\(\haction\) is closed if and only if the space of
  units~\(X\) is closed in~\(X\rtimes S\).
\end{lemma}

\begin{proof}
  This follows because the subset~\(\D_{1,t}\) is equal to the
  intersection of~\(\D_t\) with the unit bisection in
  \(X\rtimes S\) and \((\D_t)_{t\in S}\) is an open cover of~\(X\).
\end{proof}

\begin{remark}
  \label{rem:Hausdorff_vs_closed}
  Closed inverse semigroup actions are important because a
  topological groupoid is Hausdorff if and only if the object space
  is Hausdorff and the units form a closed subset of the arrows (see
  \cite{Buss-Exel-Meyer:Reduced}*{Lemma~5.2}).
\end{remark}

\subsection{Inverse semigroup actions on
  \texorpdfstring{$\Cst$}{C*}-algebras and regular inclusions}
\label{sec:isg_crossed}

\begin{definition}[\cite{Buss-Meyer:Actions_groupoids}]
  \label{def:S_action_Cstar}
  An \emph{action} of a unital inverse semigroup~\(S\) on a
  \(\Cst\)\nb-algebra~\(A\) (by Hilbert bimodules) consists of
  Hilbert \(A\)\nb-bimodules~\(\Hilm_t\) for \(t\in S\) and Hilbert
  bimodule isomorphisms
  \(\mu_{t,u}\colon \Hilm_t\otimes_A \Hilm_u\congto \Hilm_{tu}\) for
  \(t,u\in S\), such that
  \begin{enumerate}[label=\textup{(A\arabic*)}]
  \item \label{enum:AHB3}%
    for all \(t,u,v\in S\), the following diagram commutes
    (associativity):
    \[
    \begin{tikzpicture}[baseline=(current bounding box.west)]
      \node (1) at (0,1) {\((\Hilm_t\otimes_A \Hilm_u) \otimes_A \Hilm_v\)};
      \node (1a) at (0,0) {\(\Hilm_t\otimes_A (\Hilm_u \otimes_A \Hilm_v)\)};
      \node (2) at (5,1) {\(\Hilm_{tu} \otimes_A \Hilm_v\)};
      \node (3) at (5,0) {\(\Hilm_t\otimes_A \Hilm_{uv}\)};
      \node (4) at (7,.5) {\(\Hilm_{tuv}\)};
      \draw[<->] (1) -- node[swap] {ass}    (1a);
      \draw[cdar] (1) -- node {\(\mu_{t,u}\otimes_A \Id_{\Hilm_v}\)} (2);
      \draw[cdar] (1a) -- node[swap] {\(\Id_{\Hilm_t}\otimes_A\mu_{u,v}\)} (3);
      \draw[cdar] (3.east) -- node[swap] {\(\mu_{t,uv}\)} (4);
      \draw[cdar] (2.east) -- node {\(\mu_{tu,v}\)} (4);
    \end{tikzpicture}
    \]
  \item \label{enum:AHB1}%
    \(\Hilm_1\) is the identity Hilbert \(A,A\)-bimodule~\(A\);
  \item \label{enum:AHB2}%
    \(\mu_{t,1}\colon \Hilm_t\otimes_A A\congto \Hilm_t\) and
    \(\mu_{1,t}\colon A\otimes_A \Hilm_t\congto \Hilm_t\) for
    \(t\in S\) are the maps defined by
    \(\mu_{1,t}(a\otimes\xi)=a\cdot\xi\) and
    \(\mu_{t,1}(\xi\otimes a) = \xi\cdot a\) for \(a\in A\),
    \(\xi\in\Hilm_t\).
  \end{enumerate}
\end{definition}

We shall not use actions by partial automorphisms in this article.
\emph{So all actions of inverse semigroups on \(\Cst\)\nb-algebras
  are understood to be by Hilbert bimodules.}  Any \(S\)\nb-action
by Hilbert bimodules comes with canonical involutions
\(\Hilm_t^*\to \Hilm_{t^*}\), \(x\mapsto x^*\), and inclusion maps
\(j_{u,t}\colon \Hilm_t\to\Hilm_u\) for \(t \le u\) that satisfy the
conditions required for a saturated Fell bundle
in~\cite{Exel:noncomm.cartan} (see
\cite{Buss-Meyer:Actions_groupoids}*{Theorem~4.8}).  Thus
\(S\)\nb-actions by Hilbert bimodules are equivalent to saturated
Fell bundles over~\(S\).

\begin{definition}[\cite{Kwasniewski-Meyer:Stone_duality}*{Definition~6.15}]
  \label{def:isg_grading}
  An \emph{\(S\)\nb-graded \(\Cst\)\nb-algebra} is a
  \(\Cst\)\nb-algebra~\(B\) with closed subspaces \(B_t\subseteq B\)
  for \(t\in S\) such that \(\sum_{t\in S} B_t\) is dense in~\(B\),
  \(B_t B_u \subseteq B_{t u}\) and \(B_t^* = B_{t^*}\) for all
  \(t,u\in S\), and \(B_t \subseteq B_u\) if \(t\le u\) in~\(S\).
  The grading is \emph{saturated} if \(B_t\cdot B_u = B_{t u}\) for
  all \(t,u\in S\).  We call \(A\defeq B_1\subseteq B\) the
  \emph{unit fibre} of the grading.
\end{definition}

\begin{remark} 
  For group gradings, it is customary to require the fibres to be
  linearly independent.  We have no use for such a condition.  We
  will, however, restrict to ``topological gradings'' for several
  important results.  In the group case, our notion of a topological
  grading specialises to the usual one, and that implies immediately
  that the fibres are linearly independent.
\end{remark}

A (saturated) \(S\)\nb-grading~\((B_t)_{t\in S}\) on~\(B\) defines a
(saturated) Fell bundle over~\(S\) using the operations in~\(B\).
Conversely, the crossed product construction allows to realise any
(saturated) Fell bundle through a grading on a suitable
\(\Cst\)\nb-algebra.

One important source of inverse semigroup actions are Fell bundles
over étale groupoids (see Section~\ref{sec:groupoid_Fell}).  Another
important source are regular inclusions:

\begin{definition}
  \label{def:normaliser}
  Let \(A\subseteq B\) be a \(\Cst\)\nb-subalgebra.  We call the
  elements of
  \[
  N(A,B)\defeq \setgiven{b\in B}{b A b^*\subseteq A,\ b^* A b\subseteq A}
  \]
  \emph{normalisers} of~\(A\) in~\(B\)
  (see~\cite{Kumjian:Diagonals}).  We call the inclusion
  \(A\subseteq B\) \emph{regular} if it is non-degenerate
  and~\(N(A,B)\) generates~\(B\) as a \(\Cst\)\nb-algebra
  (see~\cite{Renault:Cartan.Subalgebras}).
\end{definition}

\begin{proposition}
  \label{prop:regular_vs_inverse_semigroups}
  The following are equivalent for a \(\Cst\)\nb-inclusion
  \(A\subseteq B\):
  \begin{enumerate}
  \item \label{enu:non_commutative_cartans1}%
    \(A\) is a regular subalgebra of~\(B\);
  \item \label{enu:non_commutative_cartans2}%
    \(A\) is the unit fibre for some inverse semigroup grading
    on~\(B\);
  \item \label{enu:non_commutative_cartans3}%
    \(A\) is the unit fibre for some saturated inverse semigroup
    grading on~\(B\).
  \end{enumerate}
  If these equivalent conditions hold, then
  \[
  \Slice(A,B)\defeq \setgiven{M \subseteq N(A,B)}
  {M \text{ is a closed linear subspace, } AM\subseteq M,\ MA\subseteq M}
  \]
  with the operations
  \(M\cdot N\defeq \operatorname{\overline{span}} {}\setgiven{m n}{m
    \in M,\ n\in N}\) and \(M^*\defeq \setgiven{m^*}{m \in M}\) is an
  inverse semigroup.  And the subspaces \(M\in \Slice(A,B)\) form a
  saturated \(\Slice(A,B)\)-grading on~\(B\).
\end{proposition}

\begin{proof}
  This goes back to~\cite{Exel:noncomm.cartan}.  See also
  \cite{Kwasniewski-Meyer:Stone_duality}*{Lemma~6.25 and
    Proposition~6.26}.
\end{proof}

\begin{lemma}
  \label{lem:commutative_regular}
  Any non-degenerate \(\Cst\)\nb-inclusion \(A\subseteq B\) with
  commutative~\(B\) is regular.
\end{lemma}

\begin{proof}
  Let~\(B^+\) be the minimal unitisation of~\(B\).  If \(u\in B^+\)
  is unitary, then \(A\cdot u\subseteq N(A,B)\) because~\(B\) is
  commutative.  The unital \(\Cst\)\nb-algebra~\(B^+\) is spanned by
  the unitaries it contains.  Hence~\(B\) is the linear span of
  \(A\cdot u\) for unitaries \(u\in B^+\).
\end{proof}

\subsection{Dual groupoids}
\label{sec:dual_groupoids}

Let~\(A\) be a \(\Cst\)\nb-algebra with an action~\(\Hilm\) of a
unital inverse semigroup~\(S\).  Let \(\dual{A}\) and
\(\widecheck{A}=\Prim(A)\) be the space of irreducible
representations and the primitive ideal space of~\(A\),
respectively.  Open subsets in~\(\dual{A}\) and in~\(\widecheck{A}\)
are in natural bijection with ideals in~\(A\).  The action of~\(S\)
on~\(A\) induces an action \((\widecheck{\Hilm}_t)_{t\in S}\)
on~\(\widecheck{A}\) by partial homeomorphisms by
\cite{Buss-Meyer:Actions_groupoids}*{Lemma~6.12}.  We explain how
this action lifts to an action~\(\bigl(\dual{\Hilm_t}\bigr)_{t\in S}\)
of~\(S\) on~\(\dual{A}\).

Let \(t\in S\).  Then \(\dual{\s(\Hilm_t)}\) is an open subset
of~\(\dual{A}\), consisting of all \(\pi\in\dual{A}\) that are
non-degenerate on \(\s(\Hilm_t) = \braket{\Hilm_t}{\Hilm_t}\).  Let
\(\pi\in\dual{A}\).  The tensor product
\(\Hilm_t \otimes_A \Hils_\pi\) is non-zero if and only if~\(\pi\)
belongs to \(\dual{\s(\Hilm_t)}\).  And then the left multiplication
action of~\(A\) on \(\Hilm_t \otimes_A \Hils_\pi\) is another
irreducible representation of~\(A\), which we call
\(\dual{\Hilm_t}(\pi)\in\dual{A}\).  This defines a homeomorphism
\(\dual{\Hilm_t}\colon \dual{\s(\Hilm_t)} \congto
\dual{\rg(\Hilm_t)}\) with inverse~\(\dual{\Hilm_{t^*}}\).  The
family \(\dual{\Hilm}\defeq (\dual{\Hilm}_t)_{t\in S}\) forms an
action of~\(S\) on~\(\dual{A}\) by partial homeomorphisms.

\begin{definition}
  \label{def:dual_groupoid}
  We call \(\dual{\Hilm}\) the \emph{dual action} to the
  action~\(\Hilm\).  The transformation
  groupoid~\(\dual{A}\rtimes S\) is called the \emph{dual groupoid}
  of~\(\Hilm\).
\end{definition}

Let \(\Ideals(A)\) for a \(\Cst\)\nb-algebra~\(A\) denote the
lattice of (closed, two-sided) ideals in~\(A\).

\begin{definition}
  \label{def:invariant_ideal}
  We call \(I\in\Ideals(A)\) \emph{\(\Hilm[F]\)\nb-invariant} for a
  Hilbert \(A\)\nb-bimodule~\(\Hilm[F]\) if
  \(I\cdot \Hilm[F] = \Hilm[F]\cdot I\).  We call \(I\in\Ideals(A)\)
  \emph{\(\Hilm\)\nb-invariant} for an action~\(\Hilm\) of an
  inverse semigroup~\(S\) if~\(I\) is \(\Hilm_t\)\nb-invariant for
  all \(t\in S\).  Let \(\Ideals^{\Hilm}(A)\) denote the set of all
  \(\Hilm\)\nb-invariant ideals in~\(A\).  We call~\(\Hilm\)
  \emph{minimal} if \(\Ideals^{\Hilm}(A) = \{0,A\}\), that is, the
  only \(\Hilm\)\nb-invariant ideals in~\(A\) are \(0\) and~\(A\).
\end{definition}

\begin{remark}
  Let~\(B\) be an \(S\)\nb-graded \(\Cst\)\nb-algebra with grading
  \(\Hilm = (\Hilm_t)_{t\in S}\).  Then
  \(\Ideals^{\Hilm}(A)=\setgiven{J\cap A}{J\in \Ideals(B)}\) (see
  \cite{Kwasniewski-Meyer:Stone_duality}*{Proposition~6.19}).
\end{remark}

\begin{lemma}
  \label{lem:invariance_vs_duals}
  An ideal~\(I\) in~\(A\) is \(\Hilm\)\nb-invariant if and only if
  the corresponding open subset \(\dual{I} \subseteq \dual{A}\) is
  invariant for the dual groupoid \(\dual{A}\rtimes S\).
\end{lemma}

\begin{proof}
  By Remark~\ref{rem:invariance_transformation_groupoid},
  \(\dual{A}\rtimes S\)-invariance is the same as invariance under
  the dual \(S\)\nb-action \((\dual{\Hilm}_t)_{t\in S}\)
  on~\(\dual{A}\).  Therefore, we need to show that~\(I\) is
  \(\Hilm\)\nb-invariant if and only if
  \(\dual{\Hilm}_t(\dual{I}\cap \dual{D}_t)= \dual{I} \cap
  \dual{D}_{t^*}\) for all \(t\in S\).  This follows from a known
  fact for any Hilbert \(A\)\nb-bimodule~\(\Hilm[F]\): an
  ideal~\(I\) in~\(A\) is \(\Hilm[F]\)\nb-invariant if and only
  if~\(\dual{I}\) is invariant under the partial
  homeomorphism~\(\dual{\Hilm[F]}\)
  (see~\cite{Kwasniewski:Topological_freeness} or the proof of
  \cite{Abadie-Abadie:Ideals}*{Proposition~3.10}).
\end{proof}

\begin{definition}
  \label{def:Hausdorffness}
  An inverse semigroup action~\(\Hilm\) on a
  \(\Cst\)\nb-algebra~\(A\) is called \emph{closed} if the dual
  \(S\)\nb-action on~\(\dual{A}\) is closed.
\end{definition}

\begin{remark}
  Let~\(H\) be an étale locally compact groupoid with a Hausdorff
  unit space \(X\) and let~\(S\) be a wide inverse semigroup of
  bisections of~\(H\).  If we equip \(A=\Cont_0(X)\) with the
  canonical action of~\(S\), then the dual action on
  \(\dual{A} \cong X\) is the canonical \(S\)\nb-action and
  \(\dual{A}\rtimes S \cong H\).  Hence the \(S\)\nb-action on~\(A\)
  is closed if and only if~\(H\) is Hausdorff (see
  Remark~\ref{rem:Hausdorff_vs_closed}).
\end{remark}

\subsection{Non-Triviality conditions for étale groupoids}

We carefully distinguish several versions of the concept of
topological freeness.  They are all equivalent for groupoids that
are second countable, locally compact and Hausdorff.  We will,
however, meet groupoids where the object space is the spectrum of a
\(\Cst\)\nb-algebra, which is often badly non-Hausdorff, and the
unit space is not closed in the groupoid.

\begin{definition}
  \label{def:non-triviality_groupoids}
  Let~\(H\) be an étale groupoid and \(X\subseteq H\) its unit space.
  The \emph{isotropy group} of a point \(x\in X\) is \(H(x)\defeq
  \s^{-1}(x)\cap \rg^{-1}(x)\subseteq H\).  We call~\(H\)
  \begin{enumerate}
  \item \emph{effective} if any open subset \(U\subseteq H\) with
    \(\rg|_U = \s|_U\) is contained in~\(X\);
  \item \label{enu:non-triviality_groupoids2}%
    \emph{topologically free} if, for every bisection
    \(U\subseteq H\setminus X\), the set
    \(\setgiven{x\in X}{H(x)\cap U\neq \emptyset}\) has empty
    interior;
  \item\label{enu:A-S}%
    \emph{AS topologically free} if, for finitely many bisections
    \(U_1,\dotsc,U_n\subseteq H\setminus X\), the union
    \(\bigcup_{i=1}^n {}\setgiven{x\in X}{H(x)\cap U_i\neq
      \emptyset}\) has empty interior;
  \item \emph{topologically principal} if the set of points of~\(X\)
    with trivial isotropy is dense or, equivalently, the set
    \(\setgiven{x\in X}{H(x)\setminus X\neq \emptyset}\) has empty
    interior.
  \end{enumerate}
  An action of an inverse semigroup on a topological space~\(X\) or
  on a \(\Cst\)\nb-algebra~\(A\) is called \emph{topologically free}
  if the transformation groupoid \(X\rtimes S\) or the dual groupoid
  \(\dual{A}\rtimes S\) is topologically free, and similarly for
  effective, AS topologically free and topologically principal
  actions.
\end{definition}

\begin{remark}
  Let \(\alpha\colon G\to \Aut(A)\) be an action of a discrete group
  on a \(\Cst\)\nb-algebra~\(A\).  The transformation groupoid
  \(\dual{A}\rtimes G\) for the dual action is AS topologically free
  if and only if, for any \(t_1,\dotsc, t_n\in G\setminus\{1\}\), the set
  \(\bigcup_{i=1}^n \setgiven{x\in
    \dual{A}}{\dual{\alpha}_{t_i}(x)=x}\) has empty interior
  in~\(\dual{A}\).  This definition is due to Archbold and Spielberg
  (see \cite{Archbold-Spielberg:Topologically_free}*{Definition~1}),
  and this is what ``AS'' in
  Definition~\ref{def:non-triviality_groupoids}\ref{enu:A-S} stands
  for. 
\end{remark}

\begin{remark}
  Topological freeness for groupoids as defined in
  Definition~\ref{def:non-triviality_groupoids}.\ref{enu:non-triviality_groupoids2}
  has not yet received as much attention as it deserves.  This
  condition appears, for instance, in
  \cite{Brown-Clark-Farthing-Sims:Simplicity}*{Lemma~3.1.(3)} and in
  \cite{Exel-Pitts:Weak_Cartan}*{14.15(ii)}, where it is related to
  the conditions that we call ``effective'' and ``topologically
  principal'', respectively.
\end{remark}

We are going to describe the above properties in terms of an inverse
semigroup action \(\haction\colon S\to\pHom(X)\).  For \(t\in S\),
define
\[
  \Fix(\haction_t)\defeq
  \setgiven*{x\in \D_t\setminus \D_{1,t}}{\haction_t(x)=x},
  \qquad
  \underline{\Fix}(\haction_t)\defeq
  \setgiven*{x\in \D_t\setminus
    \overline{\D_{1,t}}}{\haction_t(x)=x}.
\]
By definition,
\(\underline{\Fix}(\haction_t) = \Fix(\haction_t)=\emptyset\) if
\(t\in E(S)\).  If~\(S\) is a group and \(t\in S\setminus \{1\}\),
then
\(\Fix(\haction_t)=\underline{\Fix}(\haction_t) = \setgiven{x\in
  X}{\haction_t(x)=x}\).

\begin{lemma}
  \label{lem:groupoid_vs_action}
  Let \(\haction\colon S\to\pHom(X)\) be an inverse semigroup action
  and \(H\defeq X\rtimes S\).
  \begin{enumerate}
  \item \label{en:groupoid_vs_action1}%
    If~\(H\) is effective, then~\(H\) is topologically free.  The
    converse holds if~\(X\) is closed in~\(H\).
  \item \label{en:groupoid_vs_action1.5}%
    \(H\) is topologically free if and only if \(\Fix(\haction_t)\)
    has empty interior for any \(t\in S\), if and only if
    \(\underline{\Fix}(\haction_t)\) has empty interior for
    any \(t\in S\).
  \item \label{en:groupoid_vs_action2}%
    \(H\) is AS topologically free if and only if\/
    \(\bigcup_{i=1}^n \underline{\Fix}(\haction_{t_i})\) has empty
    interior for any \(t_1,\dotsc, t_n\in S\).
  \item \label{en:groupoid_vs_action3}%
    \(H\) is topologically principal if and only if\/
    \(\bigcup_{t\in S} \Fix(\haction_t)\) has empty interior.
  \end{enumerate}
\end{lemma}

\begin{proof}
  The construction of \(H=X\rtimes S\) implies
  \[
    \setgiven{x\in X}{H(x)\setminus X\neq \emptyset} =
    \bigcup_{t\in S} \Fix(\haction_t).
  \]
  This readily implies~\ref{en:groupoid_vs_action3}.

  For a subset \(A\subseteq X\), let \(\Int(A)\) denote its interior
  in~\(X\).  If \(t\in S\), then
  \(\Int(\D_t\setminus \D_{1,t})=\D_t\setminus \overline{\D_{1,t}}\) and
  hence
  \[
    \Int\left(\Fix(\haction_t)\right)
    = \Int\left(\underline{\Fix}(\haction_t)\right).
  \]
  Thus the second equivalence in~\ref{en:groupoid_vs_action1.5} is
  valid.  The subset
  \(V_t\defeq \setgiven*{[t,x]}{x\in \D_t\setminus
    \overline{\D_{1,t}}}\) is a bisection contained in
  \(H\setminus X\), and
  \begin{equation}
    \label{eq:fixed_points_description}
    \setgiven{x\in X}{H(x)\cap V_t\neq \emptyset}
    = \underline{\Fix}(\haction_t).
  \end{equation}
  Hence, if~\(H\) is effective or topologically free, then
  \(\Int\left(\Fix(\haction_t)\right) =
  \Int\left(\underline{\Fix}(\haction_t)\right)=\emptyset\).
  Conversely, if~\(H\) is not topologically free, then there is a
  bisection \(U\subseteq H\setminus X\) such that the interior~\(V\)
  of \(\setgiven{x\in X}{H(x)\cap U\neq \emptyset}\) is non-empty.
  Let \(x\in V\).  Then \(x\in \s(U)\) and there is a unique
  \(\gamma\in U\) with \(\s(\gamma) = x\) because~\(U\) is a
  bisection.  And \(\rg(\gamma) = x\) because
  \(H(x)\cap U\neq \emptyset\).  There is \(t\in S\) with
  \(\gamma=[t,\s(\gamma)]\).  Since~\(U\) is open, it contains an
  open neighbourhood~\(U_2\) of~\(\gamma\), which we may take of the
  form \(U_2 = \setgiven{[t,x]}{x\in V_2}\) for an open subset
  \(V_2 \subseteq \D_t\).  Then \(V\cap V_2\) is an open
  neighbourhood of~\(\s(\gamma)\) such that \([t,x]\in U\) and
  \(H(x) \cap U\neq\emptyset\) for all \(x\in V\cap V_2\).  Thus
  \(\rg([t,x]) = \s([t,x])\) for all \(x\in V \cap V_2\).  So the
  interior of \(\Fix(\haction_t)\) is non-empty.  This finishes the
  proof of~\ref{en:groupoid_vs_action1.5}.

  As we noticed above, if~\(H\) is effective, then
  \(\Int\left(\Fix(\haction_t)\right) =\emptyset\) for all
  \(t\in S\), and hence \(H\) is topologically free by~\ref{en:groupoid_vs_action1.5}.  Suppose now that~\(H\) is not
  effective.
  So there is a bisection \(U\subseteq H\) with
  \(\rg|_U = \s|_U\) that is not contained in~\(X\).  If~\(X\) is
  closed in~\(H\), then \(V\defeq U\setminus X\) is a
  bisection contained in \(H\setminus X\), and it is still
  non-empty.  Since
  \(\s(V)=\setgiven{x\in X}{H(x)\cap V\neq \emptyset}\), the
  groupoid~\(H\) is not topologically free. Thus
  \ref{en:groupoid_vs_action1} is proved.

  The `only if' part in~\ref{en:groupoid_vs_action2} follows
  from~\eqref{eq:fixed_points_description}.  For the `if' part,
  suppose that there are bisections \(U_1,\dotsc,U_n\)
  contained in \(H\setminus X\) and a non-empty open subset
  \begin{equation}
    \label{eq:W_inclusion}
    W \subseteq
    \bigcup_{i=1}^n {}\setgiven{x\in X}{H(x)\cap U_i \neq \emptyset}.
  \end{equation}
  We may assume without loss of generality that
  \(W\cap \s(U_i)\neq \emptyset\) for \(i=1,\dotsc,n\).
  Since~\(U_1\) is open and contained in \(H\setminus X\) there are
  open sets \(W_t\subseteq \D_t\setminus \D_{1,t}\), \(t\in S\),
  with
  \(U_1 = \bigcup_{t\in S} \setgiven{[t,x]}{x\in W_t,\ t\in S}\).
  If \(t_1\in S\) is such that
  \(W_{t_1}\cap W\neq \emptyset\), then replacing~\(W\) by
  \(W\cap W_{t_1}\) and~\(U_1\) by
  \(\setgiven{[t_1,x]}{x\in W_{t_1}}\), the
  inclusion~\eqref{eq:W_inclusion} remains valid.  Proceeding in
  this way, we may arrange for the sets~\(U_i\) to be of the form
  \(U_i=\setgiven{[t_i,x]}{x\in W_i}\) for some \(t_i\in S\) and
  some open subsets~\(W_i\) of \(\D_{t_i}\setminus \D_{1,t_i}\) for
  \(i=1,\dotsc,n\).  Being open, these subsets are even contained in
  \(\D_t\setminus \overline{\D_{1,t}}\).  And
  \[
    W\subseteq \bigcup_{i=1}^n {}\setgiven{x}{H(x)\cap U_i\neq \emptyset}
    = \bigcup_{i=1}^n {}\setgiven{x\in W_{t_i}}{\haction_{t_i}(x)=x}
    \subseteq \bigcup_{i=1}^n \underline{\Fix}(\haction_{t_i}).
  \]
  This finishes the proof of~\ref{en:groupoid_vs_action2}.
\end{proof}

Lemma~\ref{lem:groupoid_vs_action} implies the following relations
between the properties in
Definition~\ref{def:non-triviality_groupoids}:
\[
  (\text{effective})\Rightarrow
  (\text{top.\@ free}) \Leftarrow
  (\text{AS top.\@ free}) \Leftarrow
  (\text{top.\@ principal}).
\]
And topological freeness and effectiveness are equivalent if the
unit space is closed.
Example~\ref{exa:aperiodic_inclusion_not_aperiodic_action} below
exhibits a topologically principal action of an inverse
semigroup~\(S\) on~\([0,1]\) that is not effective (see also
\cite{Clark-Exel-Pardo-Sims-Starling:Simplicity_non-Hausdorff}*{Example~5.1}
for such an example).  There are situations when topological
freeness implies topological principality.  The following result of
this nature is essentially due to Renault (see
\cite{Renault:Cartan.Subalgebras}*{Proposition~3.6.(ii)}).

\begin{proposition}
  \label{prop:groupoid_non-triviality_conditions}
  Let~\(X\) be a Hausdorff space.  An action~\(\haction\) of an
  inverse semigroup~\(S\) on~\(X\) is topologically free if and only
  if it is AS topologically free.  If, in addition, \(X\) is a Baire
  space and~\(S\) is countable, then~\(\haction\) is topologically
  free if and only if it is topologically principal.
\end{proposition}

\begin{proof}
  Assume~\(\haction\) to be topologically free.  We are going to
  prove that~\(\haction\) is AS topologically free.  This implies
  the first statement.  Let \(t\in S\).
  Lemma~\ref{lem:groupoid_vs_action} shows that \(\Fix(\haction_t)\)
  has empty interior in~\(X\).  Equivalently, the set of
  \(x\in X_t\) with \(\haction_t(x)\neq x\) is dense in~\(X_t\).
  This subset is open because~\(X\) is Hausdorff.  Therefore, the
  open subset~\(Y_t\) of all \(x\in X\) with either
  \(x\notin \overline{X_t}\) or \(x\in X_t\) and
  \(\haction_t(x)\neq x\) is dense in~\(X\).  No point in~\(Y_t\)
  can belong to the closure of~\(\Fix(\haction_t)\).  Therefore, the
  interior of \(\overline{\Fix(\haction_t)}\) is empty or,
  equivalently, \(\Fix(\haction_t)\) is nowhere dense in~\(X\).
  This is inherited by the subset \(\underline{\Fix}(\haction_t)\)
  of \(\Fix(\haction_t)\).  A finite union of nowhere dense subsets
  is again nowhere dense.  Hence every finite union
  \(\bigcup_{i=1}^n \underline{\Fix}(\haction_{t_i})\) for
  \(t_1,\dotsc,t_n\in S\) is nowhere dense and hence has empty
  interior.  Thus~\(\haction\) is AS topologically free by
  Lemma~\ref{lem:groupoid_vs_action}.  If, in addition, \(S\) is
  countable, and~\(X\) Baire, then the countable union
  \(\bigcup_{t\in S} \Fix(\haction_t)\) is still nowhere dense
  in~\(X\).  Hence Lemma~\ref{lem:groupoid_vs_action} implies
  that~\(\haction\) is topologically principal.
\end{proof}

In Theorem~\ref{the:aperiodic_top_non-trivial}, we will prove an
analogue of
Proposition~\ref{prop:groupoid_non-triviality_conditions} for the
dual groupoid \(H = \dual{A}\rtimes S\) of an inverse semigroup
action on a separable \(\Cst\)\nb-algebra~\(A\), so~\(\dual{A}\)
need not be Hausdorff.  Some further assumption besides~\(X\) being
Baire and~\(S\) countable is needed for this converse implication.
For instance, the action of the permutation group~\(S_3\) on the
three-element set~\(X\) with the chaotic
topology~\(\{\emptyset,X\}\) is topologically free, but not AS
topologically free.

The following lemma allows to relax the assumptions in
Proposition~\ref{prop:groupoid_non-triviality_conditions} slightly:

\begin{lemma}
  \label{lem:open_dense_subgroupoid}
  Let~\(H\) be an étale groupoid and~\(X\) its space of units.  Let
  \(X'\subseteq X\) be an open dense subset of~\(X\).  Then
  \(H'\defeq \s^{-1}(X')\cap \rg^{-1}(X')\) is an open dense
  subgroupoid of~\(H\), and
  \begin{enumerate}
  \item\label{lem:open_dense_subgroupoid1}%
    \(H\) is topologically free if and only if~\(H'\) is
    topologically free.
  \item\label{lem:open_dense_subgroupoid2}%
    \(H\) is AS topologically free if and only if~\(H'\) is AS
    topologically free.
  \item\label{lem:open_dense_subgroupoid3}%
    \(H\) is topologically principal if and only if~\(H'\) is
    topologically principal.
  \end{enumerate}
\end{lemma}

\begin{proof}
  Clearly, \(H'\) is an open subgroupoid of~\(H\).  It is dense
  in~\(H\) because for every bisection \(U\subseteq H\) the
  intersection \(U\cap H'\) is open and dense in~\(U\).  For any
  \(x\in X'\) the isotropy groups in \(H\) and~\(H'\) are the same.
  This immediately gives~\ref{lem:open_dense_subgroupoid3}.
  If~\(U\) is a bisection in~\(H\), then \(U\subseteq H\setminus X\)
  if and only if \(U\cap H' \subseteq H'\setminus X'\).  This
  readily implies~\ref{lem:open_dense_subgroupoid1}.
  Concerning~\ref{lem:open_dense_subgroupoid2}, it is easy to see
  that~\(H'\) is AS topologically free if~\(H\) is.  Conversely,
  let~\(H'\) be AS topologically free and let
  \(U_1,\dotsc,U_n\subseteq H\setminus X\) be bisections of~\(H\).
  Let \(U_j' \defeq U_j \cap H'\) for \(j=1,\dotsc,n\).  Then
  \begin{align*}
    \bigcup_{i=1}^n {}\setgiven{x\in X}{H(x)\cap U_i\neq \emptyset} \cap X'
      &=\bigcup_{i=1}^n {}\setgiven{x\in X'}{H'(x)\cap U_i'\neq \emptyset}.
  \end{align*}
  The set on the right has empty interior because~\(H'\) is AS
  topologically free.  Since~\(X'\) is dense in~\(X\), it follows
  that
  \(\bigcup_{i=1}^n {}\setgiven{x\in X}{H(x)\cap U_i\neq
    \emptyset}\) has empty interior.  Thus~\(H\) is AS topologically
  free.
\end{proof}

\begin{corollary}
  \label{cor:non-triviality_conditions_equiv}
  Suppose that the space of units~\(X\) of an étale groupoid~\(H\)
  contains an open dense subset \(X'\) which is Hausdorff.  Then~\(H\) is
  topologically free if and only if~\(H\) is AS topologically free.
  If~\(X'\) is Baire and~\(H\) has a countable cover by
  bisections, then~\(H\) is topologically free if and only if it is
  topologically principal.
\end{corollary}

\begin{proof}
  Combine Proposition~\ref{prop:groupoid_non-triviality_conditions}
  and Lemma~\ref{lem:open_dense_subgroupoid}.
\end{proof}

\section{Full and reduced crossed products for inverse semigroup
  actions}
\label{sec:prelim}

We first establish some basic notation about generalised conditional
expectations.  Then we construct the full and reduced crossed
products for inverse semigroup actions and prove that the canonical
weak expectation on the reduced crossed product is faithful.

\subsection{Generalised expectations}
\label{sec:gen_expectation}

Conditional expectations are crucial tools in the study of crossed
products for group actions.  For an inverse semigroup action, the
reduced crossed product is defined in~\cite{Buss-Exel-Meyer:Reduced}
using a ``weak conditional expectation''
\(E\colon A\rtimes S\to A''\), which takes values in the bidual von
Neumann algebra~\(A''\).  Pseudo-expectations, which take values in
the injective hull of~\(A\), have been studied in
\cites{Pitts:Regular_I, Pitts-Zarikian:Unique_pseudoexpectation,
  Zarikian:Unique_expectations}.  To define the essential crossed
product, we will use expectations with values in the local
multiplier algebra (see Definition~\ref{defn:essential_expectation}
below).  The following definition covers all these cases:

\begin{definition}
  \label{def:gen_expectation}
  Let \(A\subseteq B\) be a \(\Cst\)\nb-inclusion.  A
  \emph{generalised expectation} consists of another
  \(\Cst\)\nb-inclusion \(A\subseteq \tilde{A}\) and a completely
  positive, contractive map \(B \to \tilde{A}\) that restricts to
  the identity map on~\(A\).  If \(\tilde{A} = A\),
  \(\tilde{A} = A''\), or \(\tilde{A}\) is the injective envelope
  of~\(A\), then we speak of a \emph{conditional expectation}, a
  \emph{weak conditional expectation}, or a
  \emph{pseudo-expectation}, respectively.
\end{definition}

\begin{lemma}
  Any generalised expectation is an \(A\)\nb-bimodule map.
\end{lemma}

\begin{proof}
  The unique unital, linear extension
  \(E^+\colon B^+ \to \tilde{A}^+\) is still completely positive and
  contractive (see, for instance,
  \cite{Brown-Ozawa:Approximations}*{Subsection~2.2}).  And it is
  the identity map on~\(A^+\).  Then~\(E^+\) is \(A^+\)\nb-bilinear
  by Choi's Theorem, \cite{Choi:Schwarz}*{Theorem~3.1}.
\end{proof}

Any idempotent linear contraction \(E\colon B\to A\) is a completely
positive bimodule map and thus a conditional expectation.  This
result is due to Tomiyama~\cite{Tomiyama:Projection_norm_one}.

\begin{example}
 \label{exa:trivial_gen_expectation}
 For any \(\Cst\)\nb-inclusion \(A\subseteq B\), the identity map
 on~\(B\) is a generalised expectation with values in~\(B\).
\end{example}

\begin{example}[\cite{Pitts:Regular_I}]
  \label{exa:pseudo-expectation}
  The identity map on~\(A\) extends to a completely positive map from~\(B\)
  to the injective hull of~\(A\).  Thus, any \(\Cst\)\nb-inclusion
  \(A\subseteq B\) has a pseudo-expectation.
\end{example}

We will see more examples of generalised expectations in the
definitions of the reduced and essential crossed products for
inverse semigroup actions.

\begin{definition}
  \label{def:reduced_expectation}
  Let \(E\colon B\to \tilde{A} \supseteq A\) be a generalised
  expectation.  Let~\(\Null_E\) be the closed linear span of all
  \(J\in\Ideals(B)\) with \(J\subseteq \ker E\).  This is the
  largest two-sided ideal in~\(B\) that is contained in \(\ker E\).
  Let \(B_\red \defeq B/\Null_E\) and let
  \(\Lambda\colon B\to B_\red\) be the quotient map.  Since
  \(E|_{\Null_E}=0\), the expectation~\(E\) descends to a map
  \(E_\red \colon B_\red \to \tilde{A} \supseteq A\), called the
  \emph{reduced generalised expectation} of~\(E\).
\end{definition}

Since~\(E|_A=\Id_A\) and \(E|_{\Null_E} = 0\), it follows that
\(A\cap \Null_E = 0\).  Hence the composite map
\(A \to B \to B/\Null_E\) is injective.  The map~\(E_\red\) is a
generalised expectation for the inclusion
\(A \hookrightarrow B_\red\).

\begin{proposition}
  \label{pro:GNS_kernels}
  Let \(E\colon B\to \tilde{A} \supseteq A\) be a generalised
  expectation.  Let
  \begin{align*}
    \lNull_E&\defeq\setgiven{b\in B}{E(b^*b)=0},\\
    \rNull_E&\defeq\setgiven{b\in B}{E(bb^*)=0}.
  \end{align*}
  Then \(\lNull_E\) and~\(\rNull_E\) are the largest left and right
  ideals in~\(B\) contained in~\(\ker E\), respectively.  Hence
  \((\lNull_E)^*=\rNull_E\) and
   \(\Null_E \subseteq \lNull_E \cap \rNull_E\).  And
  \begin{align*}
    \Null_E
    &=\setgiven{b\in B}{E((b c)^*b c)=0 \text{ for all }c\in B}
    =\setgiven{b\in B}{b\cdot B \subseteq \lNull_E}\\
    &=\setgiven{b\in B}{E(x b y)=0 \text{ for all }x,y\in B}.
  \end{align*}
\end{proposition}

The third description of~\(\Null_E\) was pointed out to us by Ruy
Exel.

\begin{proof}
  If necessary, adjoin units to \(B\) and~\(\tilde{A}\) to make them
  unital.  The unique unital extension of~\(E\) is unital and
  completely positive.  It is already observed in
  \cite{Choi:Schwarz}*{Remark~3.4} that~\(\lNull_E\) is the largest
  left ideal contained in \(\ker E\); this is a general feature of
  \(2\)\nb-positive maps.  The main point is the Schwarz inequality
  \(E(b^*)E(b)\le E(b^*b)\) for all \(b\in B\) (see
  \cite{Choi:Schwarz}*{Corollary~2.8}).  Since
  \(\rNull_E=(\lNull_E)^*\), it follows that~\(\rNull_E\) is the
  largest right ideal contained in~\(\ker E\).  Thus~\(\Null_E\) is
  contained in \(\lNull_E\) and~\(\rNull_E\).  If \(b\in \Null_E\),
  then \(b\cdot c \in \Null_E \subseteq \lNull_E\) for all \(c\in B\).
  Thus \(E((b c)^* b c)=0\) for all \(c\in B\).  Conversely, if
  \(E((b c)^* b c)=0\) for all \(c\in B\), then
  \(\overline{b\cdot B} \subseteq \lNull_E\); and then
  \(\overline{B\cdot b\cdot B} \subseteq \lNull_E\)
  because~\(\lNull_E\) is a left ideal.  This implies \(b\in \Null_E\)
  because \(\overline{B\cdot b\cdot B}\) is a two-sided ideal that
  contains~\(b\).  Thus \(b\in \Null_E\) if and only if
  \(E((b c)^* b c)=0\) for all \(c\in B\), if and only if
  \(b\cdot B\subseteq \lNull_E\).  And \(E(x b y)=0\) for all
  \(x,y\in B\) if and only if \(B \cdot b \cdot B \subseteq \ker
  E\), if and only if the closed two-sided ideal generated by~\(b\)
  is contained in \(\ker E\), if and only if \(b\in \Null_E\).
\end{proof}

\begin{definition}
  \label{def:faithful_expectation}
  A generalised expectation \(E\colon B\to \tilde{A} \supseteq A\)
  is \emph{faithful} if \(E(b^* b)=0\) for some \(b\in B\) implies
  \(b=0\).  It is \emph{almost faithful} if \(E((b c)^* b c)=0\) for
  all \(c\in B\) and some \(b\in B\) implies \(b=0\).  It is
  \emph{symmetric} if \(E(b^* b)=0\) for some \(b\in B\) implies
  \(E(b b^*)=0\).
\end{definition}

The concept of an almost faithful positive map plays an important
role in the theory of Exel's crossed products (see
\cite{Brownlowe-Raeburn-Vittadello:Exel}*{Definition~4.1} and
\cite{Kwasniewski:Exel_crossed}*{Subsection~2.1}).

\begin{corollary}
  \label{cor:GNS_kernels}
  Let \(E\colon B\to \tilde{A} \supseteq A\) be a generalised
  expectation.
  \begin{enumerate}
  \item \label{enu:GNS_kernels1}%
    \(E\) is symmetric if and only if
    \(\lNull_E = \rNull_E=\Null_E\);
  \item \label{enu:GNS_kernels2}%
    \(E\) is faithful if and only if
    \(\lNull_E = \rNull_E=\Null_E=0\);
  \item \label{enu:GNS_kernels3}%
    \(E\) is almost faithful if and only if \(\Null_E=0\), if and
    only if there are no non-zero ideals in~\(B\) contained
    in~\(\ker E\);
  \item \label{enu:GNS_kernels4}%
    \(E\) is faithful if and only if~\(E\) is almost faithful and
    symmetric.
  \end{enumerate}
\end{corollary}

\begin{proof}
  This readily follows from Proposition~\ref{pro:GNS_kernels}.
\end{proof}

\begin{corollary}
  \label{cor:detects_implies_almost_faithfulness}
  If~\(A\) detects ideals in~\(B\), then~\(E\) is almost faithful.
\end{corollary}

\begin{proof}
  The intersection \(\Null_E \cap A\) is~\(0\) because~\(E\) is the
  identity map on~\(A\).  Therefore, if~\(A\) detects ideals
  in~\(B\), then \(\Null_E = \{0\}\), that is, \(E\) is almost
  faithful.
\end{proof}

\begin{lemma}
  \label{lem:reduced_almost_faithful}
  The reduced generalised expectation~\(E_\red\) is almost faithful.
  It is faithful if and only if~\(E\) is symmetric.
\end{lemma}

\begin{proof}
  Proposition~\ref{pro:GNS_kernels} implies
  \(\lNull_{E_\red} =\Lambda(\lNull_E)\),
  \(\rNull_{E_\red} = \Lambda(\rNull_E)\) and
  \(\Null_{E_\red} = \Lambda(\Null_E) = 0\).  This together with
  Corollary~\ref{cor:GNS_kernels} implies all the statements.
\end{proof}

\begin{lemma}
  \label{lem:abstract_uniqueness}
  Let \(E\colon B\to \tilde{A}\supseteq A\) be a generalised
  expectation, and \(\pi\colon B\to C\) a \Star{}homomorphism.  The
  following are equivalent:
  \begin{enumerate}
  \item \label{lem:abstract_uniqueness1}%
    \(\ker \pi \subseteq \Null_E\);
  \item \label{lem:abstract_uniqueness2}%
    there is a \Star{}homomorphism
    \(\varphi\colon \pi(B)\to B/\Null_E\) with
    \(\varphi\circ \pi = \Lambda\colon B \to B/\Null_E\);
  \item \label{lem:abstract_uniqueness3}%
    there is a map \(E_\pi\colon \pi(B)\to \tilde{A}\) with
    \(E_\pi\circ \pi = E\).
  \end{enumerate}
  If these conditions hold, then \(\pi|_A\) is injective,
  \(\varphi\) and~\(E_\pi\) are unique and, identifying~\(A\)
  with~\(\pi(A)\), \( E_\pi\) is a generalised expectation for the
  inclusion \(\pi(A)\subseteq \pi(B)\).  The
  \Star{}homomorphism~\(\varphi\) in~\ref{lem:abstract_uniqueness2}
  is faithful if and only if~\(E_\pi\) is almost faithful.
\end{lemma}

\begin{proof}
  Since \(\Null_E=\ker \Lambda\), the
  \Star{}homomorphism~\(\Lambda\) descends to a \Star{}homomorphism
  from \(\pi(B) \cong B/\ker\pi\) to \(B/\Null_E\) if and
  only if \(\ker\pi\subseteq \Null_E\).
  Thus~\ref{lem:abstract_uniqueness1} is equivalent
  to~\ref{lem:abstract_uniqueness2}.  A map~\(E_\pi\) as
  in~\ref{lem:abstract_uniqueness3} exists if and only if
  \(E(\ker \pi) = 0\) or, equivalently,
  \(\ker \pi \subseteq \ker E\).  Since~\(\ker\pi\) is an ideal,
  this is equivalent to \(\ker \pi \subseteq \Null_E\).
  Thus~\ref{lem:abstract_uniqueness3} is equivalent
  to~\ref{lem:abstract_uniqueness1}.  Since \(\pi\colon B\to
  \pi(B)\) is surjective,
  the maps \(\varphi\) and~\(E_\pi\) are unique if they exist.
  Then~\(E_\pi\) is automatically a completely positive contraction.
  And~\(E_\pi\) is almost faithful if and only if
  \(\ker \pi = \Null_E\), if and only if~\(\varphi\) is injective.
\end{proof}

Lemma~\ref{lem:abstract_uniqueness} is a key point in the proof of
gauge-equivariant uniqueness theorems.  An action of a compact group
such as~\(\T\) defines a conditional expectation onto the fixed-point
algebra by averaging, and \(\T\)\nb-equivariant maps intertwine these
expectations.  Hence a conditional expectation as in
\ref{lem:abstract_uniqueness}.\ref{lem:abstract_uniqueness3} exists in
the situation of gauge-equivariant uniqueness theorems.

\subsection{Full and reduced crossed products for inverse semigroup
  actions}
\label{sec:crossed_isg}

We fix an action
\(\Hilm=\bigl((\Hilm_t)_{t\in S}, (\mu_{t,u})_{t,u\in S}\bigr)\) of
a unital inverse semigroup~\(S\) on a \(\Cst\)\nb-algebra~\(A\) as
in Definition~\ref{def:S_action_Cstar}.  We recall how the algebraic
crossed product \(A\rtimes_\alg S\), the full crossed product
\(A\rtimes S\) and the reduced crossed product \(A\rtimes_\red S\)
of~\(\Hilm\) are defined in~\cite{Buss-Exel-Meyer:Reduced}.  The
algebraic crossed products defined in~\cite{Buss-Exel-Meyer:Reduced}
and in~\cite{Exel:noncomm.cartan} differ.  The definition
in~\cite{Buss-Exel-Meyer:Reduced} has the merit that the canonical
maps from \(A\rtimes_\alg S\) to \(A\rtimes S\) and
\(A\rtimes_\red S\) are injective.

For any \(t\in S\), let \(\rg(\Hilm_t)\) and~\(\s(\Hilm_t)\) be the
ideals in~\(A\) generated by the left and right inner products of
elements in~\(\Hilm_t\), respectively.  Thus~\(\Hilm_t\) is an
\(\rg(\Hilm_t)\)-\(\s(\Hilm_t)\)-imprimitivity bimodule.  These
ideals satisfy
\[
  \s(\Hilm_t) = \s(\Hilm_{t^*t})
  = \rg(\Hilm_{t^*t})
  = \rg(\Hilm_{t^*}).
\]
If \(v\le t\), then the inclusion map~\(j_{t,v}\) restricts to a
Hilbert bimodule isomorphism from~\(\Hilm_v\) onto
\(\rg(\Hilm_v) \cdot \Hilm_t = \Hilm_t\cdot \s(\Hilm_v)\).  For
\(t,u\in S\) and \(v\le t,u\), this gives Hilbert bimodule
isomorphisms
\[
  \vartheta^v_{u,t}\colon
  \Hilm_t\cdot \s(\Hilm_v) \xrightarrow[\cong]{j^{-1}_{t,v}}
  \Hilm_v \xrightarrow[\cong]{j_{u,v}}
  \Hilm_u\cdot \s(\Hilm_v).
\]
Let
\begin{equation}
  \label{eq:Itu}
  I_{t,u} \defeq \overline{\sum_{v \le t,u} \s(\Hilm_v)}
\end{equation}
be the closed ideal generated by~\(\s(\Hilm_v)\) for \(v\le t,u\).
This is contained in \(\s(\Hilm_t)\cap\s(\Hilm_u)\), and the
inclusion may be strict.  There is a unique Hilbert bimodule
isomorphism
\begin{equation}
  \label{eq:Def-thetas}
  \vartheta_{u,t}\colon \Hilm_t\cdot I_{t,u}
  \congto \Hilm_u\cdot I_{t,u}
\end{equation}
that restricts to~\(\vartheta_{u,t}^v\) on
\(\Hilm_t\cdot \s(\Hilm_v)\) for all \(v\le t,u\) by
\cite{Buss-Exel-Meyer:Reduced}*{Lemma~2.4}.

Let \(A\rtimes_\alg S\) be the quotient vector space of
\(\bigoplus_{t\in S} \Hilm_t\) by the linear span of
\(\vartheta_{u,t}(\xi)\delta_u-\xi\delta_t\) for all \(t,u\in S\)
and \(\xi\in\Hilm_t\cdot I_{t,u}\).  The multiplication
maps~\(\mu_{t,u}\) and the involutions \(\Hilm_t^* \to \Hilm_{t^*}\)
turn \(A\rtimes_\alg S\) into a \Star{}algebra.

\begin{definition}[\cite{Buss-Exel-Meyer:Reduced}]
  The \emph{\textup{(}full\textup{)} crossed product} \(A\rtimes S\)
  of the action~\(\Hilm\) is the maximal \(\Cst\)\nb-completion of
  the \Star{}algebra \(A\rtimes_\alg S\) described above.
\end{definition}

The \(\Cst\)\nb-algebra \(A\rtimes S\) is canonically isomorphic to
the full section \(\Cst\)\nb-algebra of the Fell bundle over~\(S\)
corresponding to~\(\Hilm\) introduced in~\cite{Exel:noncomm.cartan}.
It is also characterised by a universal property.

\begin{definition}
  \label{def:representation}
  A \emph{representation} of~\(\Hilm\) in a
  \(\Cst\)\nb-algebra~\(B\) is a family of linear maps
  \(\pi_t\colon \Hilm_t\to B\) for \(t\in S\) such that
  \(\pi_{tu}(\mu_{t,u}(\xi\otimes\eta)) = \pi_t(\xi)\pi_u(\eta)\),
  \(\pi_t(\xi_1)^*\pi_t(\xi_2) = \pi_1(\braket{\xi_1}{\xi_2})\) and
  \(\pi_t(\xi_1)\pi_t(\xi_2)^* = \pi_1(\BRAKET{\xi_1}{\xi_2})\) for
  all \(t,u\in S\), \(\xi, \xi_1,\xi_2 \in\Hilm_t\),
  \(\eta\in\Hilm_u\).
  Here \(\braket{\cdot}{\cdot}\) and \(\BRAKET{\cdot}{\cdot}\)
  denote the right and left inner products, respectively.
\end{definition}

\begin{remark}
  \label{rem:representation_vs_grading}
  A representation~\(\pi\) of~\(\Hilm\) in~\(B\) induces a
  \Star{}homomorphism \(A\rtimes S\to B\) and, conversely, every
  \Star{}homomorphism \(A\rtimes S\to B\) is of this form for a
  unique representation~\(\pi\) (compare
  \cite{Buss-Exel-Meyer:Reduced}*{Proposition~2.9}).  This universal
  property determines \(A\rtimes S\) uniquely up to isomorphism.
  Let~\(B\) be a \(\Cst\)\nb-algebra with an \(S\)\nb-grading
  \((B_t)_{t\in S}\).  Then~\((B_t)_{t\in S}\) is a Fell bundle
  over~\(S\), and the inclusion maps \(B_t\to B\) are a
  representation of this Fell bundle.  The induced
  \Star{}homomorphism \(A\rtimes S \to B\) is surjective because
  \(\sum B_t\) is dense in~\(B\).  And its restriction to
  \(A\subseteq A\rtimes S\) is injective by construction.
  Conversely, any surjective \Star{}homomorphism
  \(A\rtimes S \to B\) that restricts to an injective map on
  \(A\subseteq A\rtimes S\) comes from a unique \(S\)\nb-grading
  on~\(B\) because \(A\rtimes  S\) is \(S\)\nb-graded by the images
  of the Hilbert bimodules~\(\Hilm_t\) for \(t\in S\).
\end{remark}

The reduced section \(\Cst\)\nb-algebra of a Fell bundle over~\(S\)
is introduced in~\cite{Exel:noncomm.cartan}.  An equivalent
definition appears in~\cite{Buss-Exel-Meyer:Reduced}, where it is
called the reduced crossed product \(A\rtimes_\red S\) of the
action~\(\Hilm\).  The main ingredient in the construction
in~\cite{Buss-Exel-Meyer:Reduced} is a canonical conditional
expectation \(E\colon A''\rtimes_\alg S \to A''\), involving the
unique normal extension of the \(S\)\nb-action~\(\Hilm\) to~\(A''\).
It is defined in \cite{Buss-Exel-Meyer:Reduced}*{Lemma~4.5} through
the formula
\begin{equation}
  \label{eq:formula-cond.exp}
  E(\xi\delta_t)=\stlim_i\vartheta_{1,t}(\xi\cdot u_i)
\end{equation}
for \(\xi\in \Hilm_t\) and \(t\in S\), where~\((u_i)\) is an
approximate unit for~\(I_{1,t}\) and \(\stlim\) denotes the limit in
the strict topology on \(\Mult(I_{1,t})\subseteq A''\).   In fact,
this net also converges in the strong topology on~\(A''\).  Let
\(x\in A\rtimes_\alg S\).  Then \(E(x^* x) \ge 0\), and
\(E(x^* x)=0\) implies \(x=0\) (see
\cite{Buss-Exel-Meyer:Reduced}*{Proposition~3.6}).  Thus
\(A''\rtimes_\alg S\) may be completed to a Hilbert
\(A''\)\nb-module \(\ell^2(S,A'')\) using the inner product
\(\braket{x}{y} \defeq E(x^* y)\).  The action of
\(A\rtimes_\alg S\) on \(A''\rtimes_\alg S\) by left multiplication
extends to a non-degenerate representation of \(A\rtimes S\) on
\(\ell^2(S,A'')\).

\begin{definition}[\cite{Buss-Exel-Meyer:Reduced}]
  \label{def:reduced_crossed}
  The \emph{reduced crossed product} \(A\rtimes_\red S\) of the
  action~\(\Hilm\) is the image of \(A\rtimes S\) in
  \(\Bound(\ell^2(S,A''))\).  Let
  \(\Lambda\colon A\rtimes S \to A\rtimes_\red S\) be the canonical
  map.
\end{definition}

This definition agrees with the one
in~\cite{Buss-Exel-Meyer:Reduced} by
\cite{Buss-Exel-Meyer:Reduced}*{Remark~4.2}.

\begin{proposition}
  \label{prop:reduced_through_E}
  The map \(E\colon A\rtimes_\alg S \to A''\) extends to a weak
  conditional expectation \(A\rtimes S \to A''\), which we also denote
  by~\(E\).  Moreover, \(\Null_E = \ker\Lambda\).  Thus
  \[
    A\rtimes_\red S \cong (A\rtimes S)/\Null_E
  \]
  is the reduced quotient of~\(A\rtimes S\) for the
  generalised expectation~\(E\), and \(E\) descends to an almost
  faithful weak conditional expectation
  \(E_\red\colon A\rtimes_\red S \to A''\).
\end{proposition}

\begin{proof}
  The unit element of~\(A''\) gives an element~\(1\) of
  \(A''\rtimes_\alg S \subseteq \Bound(\ell^2(S,A''))\), and
  \(\braket{1}{b\cdot 1} = E(b)\) for all \(b\in A\rtimes S\).  This
  provides the unique extension of~\(E\) to a completely positive,
  contractive map \(A\rtimes S \to A''\), which we also denote
  by~\(E\).  By construction, an element \(b\in A\rtimes S\)
  satisfies \(\Lambda(b)=0\) if and only if \(b\cdot c=0\) for all
  \(c\in \ell^2(S,A'')\).  This is equivalent to
  \(E(c^* b^* b c) = \braket{b\cdot c}{b\cdot c} = 0\) for all
  \(c\in \ell^2(S,A'')\).  This follows once it holds for all~\(c\)
  in the dense subspace \(A'' \rtimes_\alg S\).  Since~\(E\) is
  normal, we may further reduce to the weakly dense subspace
  \(A \rtimes_\alg S\).  This is norm dense in \(A \rtimes S\).  So
  \(\Lambda(b)=0\) if and only if \(E(c^* b^* b c) = 0\) for all
  \(c\in A \rtimes S\).  Then Proposition~\ref{pro:GNS_kernels} implies
  \(\ker \Lambda = \Null_E\).  Hence the weak conditional
  expectation on \(A \rtimes S\) descends to one on
  \(A \rtimes_\red S\), which is almost faithful.
\end{proof}

\begin{remark}
  \label{rem:cross_product_graded}
  The canonical map from \(A\rtimes_\alg S\) to
  \(A\rtimes_\red S\) is injective by
  \cite{Buss-Exel-Meyer:Reduced}*{Proposition~4.3}.
\end{remark}

\begin{remark}
  \label{rem:make_saturated}
  Let \((\A_t)_{t\in S}\) be a non-saturated Fell bundle over a
  unital inverse semigroup~\(S\).  It is turned into a saturated
  Fell bundle \((\Hilm_t)_{t\in \tilde{S}}\) over another inverse
  semigroup~\(\tilde{S}\)
  in~\cite{BussExel:InverseSemigroupExpansions}.  This construction
  does not change the full and reduced section \(\Cst\)\nb-algebras
  by \cite{BussExel:InverseSemigroupExpansions}*{Theorem~7.2}.
  Therefore, we usually restrict attention to saturated Fell
  bundles, which we replace by inverse semigroup actions as in
  Definition~\ref{def:S_action_Cstar}.
\end{remark}

We recall some known conditions for \(E\colon A\rtimes S \to A''\) to
be \(A\)\nb-valued, that is, a genuine conditional expectation.

\begin{proposition}[\cite{Buss-Exel-Meyer:Reduced}*{Proposition~6.3}]
  \label{pro:A_tilde_A}
  The map~\(E\) is \(A\)\nb-valued if and only if the
  ideal~\(I_{1,t}\) defined in~\eqref{eq:Itu} is complemented in the
  larger ideal~\(\s(\Hilm_t)\) for each \(t\in S\).
\end{proposition}
 There is a largest ideal
\(J\in\Ideals(A)\) with \(J\cap I_{1,t}=\{0\}\), namely,
\begin{equation}
  \label{eq:I1t_orthogonal}
  I_{1,t}^\bot \defeq \setgiven{x\in A}{x\cdot I_{1,t}=0}.
\end{equation}
Therefore, \(I_{1,t}\) is complemented in~\(\s(\Hilm_t)\) if and
only if
\(\s(\Hilm_t) = I_{1,t} \oplus (\s(\Hilm_t) \cap I_{1,t}^\bot)\).
Then
\begin{equation}
  \label{eq:decompose_Hilm_closed}
  \Hilm_t = (\Hilm_t \cdot I_{1,t}) \oplus (\Hilm_t \cdot I_{1,t}^\bot)
  \xrightarrow[\cong]{\vartheta_{1,t} \oplus \Id}
  I_{1,t} \oplus (\Hilm_t \cdot I_{1,t}^\bot),
\end{equation}
and \(E|_{\Hilm_t}\) is the orthogonal projection onto the
summand~\(I_{1,t}\) by~\eqref{eq:formula-cond.exp}.

\begin{proposition}
  \label{prop:conditional_expectation}
  Let~\(\Hilm\) be an action of~\(S\) on~\(A\).  The following are
  equivalent:
  \begin{enumerate}
  \item \label{prop:conditional_expectation1}%
    the weak conditional expectation \(E\colon A\rtimes S\to A''\)
    is \(A\)\nb-valued;
  \item \label{prop:conditional_expectation2}%
    the subset of units~\(\widecheck{A}\) is closed in
    \(\widecheck{A}\rtimes S\);
  \item \label{prop:conditional_expectation3}%
    the action is closed, that is, the dual \(S\)\nb-action
    on~\(\dual{A}\) is closed.
  \end{enumerate}
\end{proposition}

\begin{proof}
  The equivalence of the first two statements is
  \cite{Buss-Exel-Meyer:Reduced}*{Theorem 6.5}.  The map
  \(\dual{A}\ni [\pi]\mapsto \ker[\pi]\in \widecheck{A}\)
  intertwines the inverse semigroup actions
  \((\dual{\Hilm}_t)_{t\in S}\) and
  \((\widecheck{\Hilm}_t)_{t\in S}\).  It is continuous, open and
  closed because the spaces \(\widecheck{A}\) and~\(\dual{A}\) both
  have~\(\Ideals(A)\) as their lattice of open subsets.  So the map
  above extends to a continuous, open and closed groupoid
  homomorphism
  \[
  \kappa\colon \dual{A}\rtimes S \to \widecheck{A}\rtimes S,\qquad
  [t,[\pi]]\mapsto [t,\ker \pi].
  \]
  Since \(\kappa^{-1}(\widecheck{A})=\dual{A}\), the space of
  units~\(\dual{A}\) is closed in \(\dual{A}\rtimes S\) if and only
  if~\(\widecheck{A}\) is closed in \(\widecheck{A}\rtimes S\).  Now
  \ref{prop:conditional_expectation2}
  and~\ref{prop:conditional_expectation3} are equivalent by
  Lemma~\ref{lem:closed_groupoid_vs_action}.
\end{proof}

\subsection{The reduced weak conditional expectation is faithful}
\label{sec:faithful_weak_expectation}

It is claimed in~\cite{Buss-Exel-Meyer:Reduced} that the weak
conditional expectation~\(E\) becomes faithful on \(A\rtimes_\red S\).
But in \cite{Buss-Exel-Meyer:Reduced}*{Proposition~3.6}, this is shown
only for its restriction to \(A\rtimes_\alg S\).  Here we fill this
gap.

\begin{lemma}
  \label{lem:atomic_weak_expectation}
  Let
  \(\varrho\colon A'' \to \prod_{\pi\in\dual{A}} \Bound(\Hils_\pi)\)
  be the projection to the direct sum of all irreducible
  representations in the universal representation of~\(A''\).  Then
  \(\varrho\) is isometric on the \(\Cst\)\nb-algebra generated by
  the range of~\(E\) in~\(A''\).  Hence
  \(\norm{E(x)} = \norm{\varrho\circ E(x)}\) for all
  \(x\in A\rtimes S\) and, in particular,
  \(\ker E = \ker (\varrho\circ E)\).
\end{lemma}

\begin{proof}
  A family of representations \((\pi_i)_{i\in I}\) of~\(A\) is called
  \emph{\(E\)\nb-faithful} in
  \cite{Buss-Exel-Meyer:Reduced}*{Definition~4.9} if the extension of
  the representation \(\bigoplus_{i\in I} \pi_i\) to~\(A''\) restricts
  to a faithful representation on the \(\Cst\)\nb-subalgebra that is
  generated by the image of~\(E\).  The family of all irreducible
  representations of~\(A\) is \(E\)\nb-faithful by
  \cite{Buss-Exel-Meyer:Reduced}*{Theorem~7.4}.
\end{proof}

\begin{theorem}
  \label{the:crossed_expectation_faithful}
  The weak conditional expectation \(E\colon A\rtimes S \to A''\) is
  symmetric.  Equivalently, \(E_\red\colon A\rtimes_\red S \to A''\) is
  faithful.
\end{theorem}

\begin{proof}
  By Proposition~\ref{pro:GNS_kernels} and
  Corollary~\ref{cor:GNS_kernels}, it suffices to show that if
  \(b\in A\rtimes S\) satisfies \(E(b^* b)=0\), then
  \(E(c^* b^* b c)=0\) for all \(c\in A\rtimes S\).  Since
  \(c\in A\rtimes S\) satisfies \(E(c^* b^* b c)=0\) if and only if
  \(b\cdot c \in \lNull_E\), the space of~\(c\) with this property
  is a closed linear subspace.  Hence it suffices to prove
  \(E(c^* b^* b c)=0\) for \(c\in \Hilm_t\) for some \(t\in S\).  Here we
  embed the Hilbert \(A\)\nb-bimodules~\(\Hilm_t\) of the \(S\)\nb-action
  on~\(A\) into \(A\rtimes_\alg S\) in the usual way.  So we may fix
  \(t\in S\) and \(c\in\Hilm_t\).  We must show that
  \(E(b^* b)=0\) implies \(E(c^* b^* b c)=0\).

  Let \(\pi\in\dual{A}\).  The tensor product
  \(\Hilm_t \otimes_A \Hils_\pi\) is non-zero if and only if~\(\pi\)
  belongs to \(\dual{\s(\Hilm_t)}\), and then the left
  multiplication action of~\(A\) on \(\Hilm_t \otimes_A \Hils_\pi\)
  is the irreducible representation
  \(\omega\defeq\dual{\Hilm_t}(\pi)\).  The element \(c\in\Hilm_t\)
  induces an adjointable linear map
  \[
    T_c\colon \Hils_\pi \to \Hilm_t \otimes_A \Hils_\pi,\qquad
    \xi\mapsto c\otimes \xi.
  \]
  Its adjoint is given by \(T_c^*(e\otimes \xi)=\pi(c^*e)\xi\) for
  \(e\in \Hilm_t\), \( \xi\in\Hils_\pi\).  Let \(\pi''\)
  and~\(\omega''\) be the weakly continuous extensions of \(\pi\)
  and~\(\omega\) to~\(A''\).  We claim that
  \begin{equation}
    \label{eq:prepare_faithful}
    \pi''(E(c^* d c)) =
    \begin{cases}
      T_c^* \omega''(E(d)) T_c&\text{if }\pi\in\dual{\s(\Hilm_t)},\\
      0&\text{otherwise,}
    \end{cases}
  \end{equation}
  for all \(d\in A\rtimes S\).  We are interested, of course, in
  \(d = b^* b\).  If \(\pi\notin\dual{\s(\Hilm_t)}\), then
  \(\pi''(E(c^* d c))=0\).  So it suffices to consider the case when
  \(\pi\in\dual{\s(\Hilm_t)}\).  The set of \(d\in A\rtimes S\) for
  which~\eqref{eq:prepare_faithful} holds is a closed linear
  subspace because the two sides of the equality are bounded linear
  operators of~\(d\).  Therefore, it suffices to
  prove~\eqref{eq:prepare_faithful} for \(d\in\Hilm_u\) with some
  \(u\in S\).  It is more convenient to work in the bidual~\(A''\)
  for this computation and to allow \(d\in\Hilm_u''\).  Let~\([I]\)
  for an ideal~\(I\) in~\(A\) denote the support projection
  of~\(I\), that is, the weak limit in~\(A''\) of an approximate
  unit in~\(I\).  There is a canonical Hilbert bimodule isomorphism
  \[
    \Theta_u\colon \Hilm_u''\cdot [I_{1,u}] \congto I_{1,u}'',
  \]
  and \(E(d) = \Theta_u(d\cdot [I_{1,u}]) = E(d\cdot [I_{1,u}])\).
  Similarly,
  \(E(c^* d c) = \Theta_{t^* u t}(c^* d c \cdot [I_{1, t^* u t}])\)
  because \(c^* d c \in \Hilm_{t^* u t}''\).  The Rieffel
  correspondence for~\(\Hilm_t''\) implies that
  \(\Hilm_t''\cdot [I] = (\Hilm_t\cdot I)'' = (J\cdot\Hilm_t)'' =
  [J]\cdot \Hilm_t''\), where
  \(\dual{J} = \dual{\Hilm_t}(\dual{I})\).  Therefore,
  \(c\cdot [I_{1,t^* u t}] = [J]\cdot c\), where
  \(\dual{J} = \dual{\Hilm_t}(\dual{I_{1,t^* u t}})\).  Since
  \[
    \dual{\Hilm_t}(\dual{I_{1,t^* u t}}) = \dual{I_{t t^*,u}}
    = \dual{I_{1,u} \cap \rg(\Hilm_t)},
  \]
  we get \([J]\cdot c = [I_{1,u}]\cdot c\).  So
  \[
    E(c^* d c)
    = \Theta_{t^* u t}(c^* d\cdot ([I_{1,u}]\cdot c))
    = E(c^* (d\cdot [I_{1,u}]) c).
  \]
  Therefore, the two sides in~\eqref{eq:prepare_faithful} do not
  change when we replace~\(d\) by \(d\cdot [I_{1,u}]\).  But
  \(\Theta_u(d\cdot [I_{1,u}]) \in I_{1,u}'' \subseteq A''\) and
  \(d \in\Hilm_u''\cdot [I_{1,u}] \subseteq \Hilm_u''\) are
  identified in~\(A''\rtimes S\).  Since the two sides
  in~\eqref{eq:prepare_faithful} only depend on the image of~\(d\)
  in~\(A''\rtimes S\), we are reduced to the case \(d\in A''\).
  Then \(c^* d c \in A''\) as well, and so \(E(d) = d\) and
  \(E(c^* d c) = c^* d c\).  And~\eqref{eq:prepare_faithful} follows
  because \(T_c^* \omega''(d) T_c=\pi''(c^*dc) \).  This finishes
  the proof of~\eqref{eq:prepare_faithful}.

  Now \(E(b^* b)=0\) is equivalent to \(\omega''(E(b^* b))=0\) for all
  \(\omega\in\dual{A}\) by Lemma~\ref{lem:atomic_weak_expectation}.
  This implies \(\pi(E(c^* b^* b c))= T_c^* \omega''(E(b^* b)) T_c = 0\)
  for all \(\pi\in\dual{A}\) by~\eqref{eq:prepare_faithful}.  This is
  equivalent to \(E(c^* b^* b c)=0\) by
  Lemma~\ref{lem:atomic_weak_expectation}.  This proves the claim
  about~\(E\) being symmetric.  And this is equivalent to~\(E_\red\)
  being faithful by Lemma~\ref{lem:reduced_almost_faithful}.
\end{proof}

\section{The essential crossed product}
\label{sec:essential_crossed}

We are going to define a variant of the reduced crossed product that
is based on a conditional expectation with values in the local
multiplier algebra.

\subsection{The definition of the essential crossed product}
\label{sec:def_essential_crossed}

\begin{definition}[\cite{Ara-Mathieu:Local_multipliers}]
  \label{def:local_multiplier_algebra}
  Let~\(A\) be a \(\Cst\)\nb-algebra.  The essential ideals in~\(A\)
  form a directed set for the relation~\(\supseteq\) because
  \(I\cap J\in\Ideals(A)\) is essential if \(I,J\in\Ideals(A)\) are
  essential.  If \(I\subseteq J \subseteq A\) are essential ideals,
  then a multiplier of~\(J\) restricts to a multiplier of~\(I\).  This
  defines a canonical, unital \Star{}homomorphism
  \(\Mult(J) \to \Mult(I)\), which is injective because~\(I\) is
  essential in~\(J\).  The \emph{local multiplier
    algebra}~\(\Locmult(A)\) of~\(A\) is the inductive limit
  \(\Cst\)\nb-algebra of this inductive system.
\end{definition}

The canonical map \(A \to \Locmult(A)\) is injective because
\(A \to \Mult(I)\) is injective for any essential ideal~\(I\)
in~\(A\).  We identify~\(A\) with its image in~\(\Locmult(A)\).

\begin{definition}
  \label{defn:essential_expectation}
  A conditional expectation with values in \(\Locmult(A)\) is called
  an \emph{\(\Locmult\)-expectation}.
\end{definition}


Now let us fix an action~\(\Hilm\) of a unital inverse
semigroup~\(S\) on a \(\Cst\)\nb-algebra~\(A\).  We modify the
construction of the reduced crossed product by replacing the weak
conditional expectation \(E\colon A\rtimes S \to A''\) by an
\(\Locmult\)-expectation
\(EL\colon A\rtimes S \to \Locmult(A)\).  We first define~\(EL\) on
the dense subalgebra \(A\rtimes_\alg S \subseteq A\rtimes S\).  This
is spanned by~\(\Hilm_t\) for \(t\in S\).  Let \(t\in S\) and
define~\(I_{1,t}\) as in~\eqref{eq:Itu}.  Recall that the
ideal~\(I_{1,t}^\bot\) defined in~\eqref{eq:I1t_orthogonal} is the
largest ideal \(J\in\Ideals(A)\) with \(J\cap I_{1,t}=\{0\}\).  The
ideal \(J_t \defeq I_{1,t} \oplus I_{1,t}^\bot\subseteq A\) is
essential by construction.  And
\[
  \Hilm_t \cdot J_t
  \cong  (\Hilm_t  \cdot  I_{1,t})\oplus (\Hilm_t\cdot  I_{1,t}^\bot)
  \xrightarrow[\cong]{\vartheta_{1,t} \oplus \Id}
  I_{1,t} \oplus  (\Hilm_t\cdot  I_{1,t}^\bot)
\]
as in~\eqref{eq:decompose_Hilm_closed}.  Any \(\xi\in\Hilm_t\)
defines a map \(J_t \to \Hilm_t \cdot J_t\),
\(a\mapsto \xi\cdot a\).  Composing with the orthogonal projection
to~\(I_{1,t}\) defines an adjointable operator
\(J_t \to I_{1,t} \subseteq J_t\) or, equivalently, a multiplier
of~\(J_t\).  We let \(EL(\xi) \in \Mult(J_t) \subseteq \Locmult(A)\)
be this multiplier of~\(J_t\).  

\begin{proposition}
  \label{pro:EL_exists}
  The map \(\bigoplus_{t\in S} \Hilm_t\ni  \xi \mapsto EL(\xi)  \in \Locmult(A)\)
  factors through to a map
  \(A\rtimes_\alg S \to \Locmult(A)\), which extends to an
  \(\Locmult\)-expectation
  \(EL\colon A\rtimes S \to \Locmult(A)\) that satisfies
  \(\norm{EL(x)} \le \norm{E(x)}\) for all \(x\in A\rtimes S\).
\end{proposition}

\begin{proof}
  If \(\xi \in \Hilm_t\cdot J_t\), then the above definition gives
  \(EL(\xi) \in I_{1,t} \subseteq A\).  In this case,
  \(EL(\xi) = E(\xi)\) because the strict limit
  in~\eqref{eq:formula-cond.exp} is equal to a norm limit.  More
  generally, let \(\xi\in A\rtimes_\alg S\).  Then there are a
  finite subset \(F\subseteq S\) and \(\xi_t \in \Hilm_t\) for
  \(t\in F\) with \(\xi= \sum_{t\in F} \xi_t\in A\rtimes_\alg S\).
  The ideal \(J \defeq \bigcap_{t\in F} J_t\) in~\(A\) is essential
  as a finite intersection of essential ideals.  The arguments above
  imply
  \[
    E(\xi)a
    = \sum_{t\in F} E(\xi_ta)
    = \sum_{t\in F} EL(\xi_ta)=\left(\sum_{t\in F} EL(\xi_t)\right) a\in J
  \]
  for every \(a\in J\).  Hence
  \(\sum_{t\in F} EL(\xi_t)\in \Mult(J)\subseteq \Locmult(A)\)
  coincides with the multiplier given by multiplication
  by~\(E(\xi)\).  This implies that
  \(EL(\xi)\defeq \sum_{t\in F} EL(\xi_t)\) is well defined on
  \(A\rtimes_\alg S\) (because~\(E\) is).  The norm of~\(EL(\xi)\)
  in \(\Locmult(A)\) is equal to its norm in~\(\Mult(J)\).  If
  \(a\in J\) satisfies \(\norm{a}\le1\), then
  \(\norm{EL(\xi)\cdot a} = \norm{E(\xi)\cdot a} \le
  \norm{E(\xi)}\).  Hence \(\norm{EL(\xi)} \le \norm{E(\xi)}\).
  Thus~\(EL\) extends to a contractive linear map
  \(EL\colon A\rtimes S \to \Locmult(A)\) because
  \(E\colon A\rtimes S \to A''\) is contractive.  The above
  considerations also imply \(EL|_A = \Id_A\).

  To see that~\(EL\) is completely positive, let
  \(x\in \Mat_n(A\rtimes_\alg S)\).  There are a finite subset
  \(F\subseteq S\) and \(x_t \in \Hilm_t\otimes \Mat_n\) for \(t\in
  F\) such that \(x = \sum_{t\in F} x_t\).  Then
  \(J \defeq \bigcap_{s,t\in F} J_{s^*t}\) is an essential ideal
  in~\(A\).  One checks as above that multiplication by
  \(E(x^*x)\) on \(\Mat_n(J)\) coincides with
  \(EL(x^*x)\in \Mult(\Mat_n(J))=\Mat_n(\Mult(J))\subseteq
  \Mat_n(\Locmult(A))\).  Therefore, if \(y\in \Mat_n(J)\), then
  \[
    y^*\cdot EL(x^* x)\cdot y = y^*\cdot E(x^* x)\cdot y = E((xy)^* xy) \ge 0
  \]
  because~\(E\) is completely positive.  Therefore, \(EL(x^* x)\) is
  positive in \(\Mat_n(\Locmult(A))\).  So~\(EL\) is completely
  positive on \(A\rtimes_\alg S\).  It follows that~\(EL\) is
  completely positive on~\(A\rtimes S\).  So~\(EL\) is an
  \(\Locmult\)-expectation.
\end{proof}

\begin{definition}
  \label{def:essential_crossed}
  The \emph{essential crossed product} \(A\rtimes_\ess S\) is the
  reduced quotient \((A\rtimes S)/\Null_{EL}\) for the generalised
  expectation~\(EL\).
\end{definition}

\begin{remark}
  Since \(\norm{EL(x)} \le \norm{E(x)}\) for all
  \(x\in A\rtimes S\), it follows that \(\ker EL \supseteq \ker E\).
  Thus \(A\rtimes_\ess S\) is also a quotient of
  \(A\rtimes_\red S\).  By construction, \(EL\) descends to an
  almost faithful \(\Locmult\)-expectation
  \(A\rtimes_\ess S\to \Locmult(A)\) for the inclusion
  \(A \hookrightarrow A\rtimes_\ess S\).
\end{remark}

\begin{remark}
  If the action is closed, then \(EL=E\) and
  \(A\rtimes_\ess S=A\rtimes_\red S\) by the construction
  and~\eqref{eq:decompose_Hilm_closed}.  Thus essential crossed
  products give something new only for non-closed actions.
\end{remark}

\begin{example}
  \label{exa:group_bundle}
  We show by an example that~\(A\rtimes_\ess S\) may differ from
  \(A\rtimes_\red S\).  Let~\(G\) be an amenable discrete group.
  View~\(G\) as an inverse semigroup and adjoin a zero element to
  it.  This gives the inverse semigroup \(S \defeq G \cup \{0\}\)
  with \(g\cdot 0 = 0\cdot g = 0\) for all \(g\in S\) and such that
  \(g\cdot h\) for \(g,h\in G\) is the usual product in~\(G\).
  Let~\(G\) act on \(A\defeq \Cont[0,1]\) by \(\Hilm_g \defeq
  \Cont[0,1]\) for
  \(g\in G\) and \(\Hilm_0 \defeq \Cont_0(0,1]\), equipped with the usual
  involution and multiplication maps.  If \(g\in G\setminus\{1\}\),
  then \(I_{g,1} = \Cont_0(0,1]\) because \(v\in S\) satisfies
  \(v\le 1,g\) if and only if \(v=0\).  The full crossed product for
  this action is a \(\Cont[0,1]\)-\(\Cst\)-algebra with fibres
  \(\Cst(G)\) at~\(0\) and \(\C\) at \(x\in (0,1]\).  The reduced
  crossed product is equal to the full one here because~\(G\) is
  amenable.  (If~\(G\) is non-amenable, then the fibre of the
  reduced crossed product at~\(0\) is \(\C\oplus \Cst_\red(G)\) and
  not \(\Cst_\red(G)\).  We restrict to amenable groups to avoid
  discussing this issue further.)  The essential crossed product in
  this case reduces to \(A=\Cont[0,1]\).  Indeed, let
  \(\xi_g\in \Hilm_g = \Cont[0,1]\).  Let \(\xi_1\in \Hilm_1\) be
  the same element of \(\Cont[0,1]\).  Then
  \(EL(\xi_g) = EL(\xi_1)\) because the ideal \(\Cont_0(0,1]\)
  in~\(A\) is essential.  Thus \(\ker EL\) contains the kernel of
  the \Star{}homomorphism \(A\rtimes_\alg S \to A\) that
  applies the trivial representation of the group in each fibre.  It
  follows that \(A\rtimes_\ess S\cong A\).

  In this example, \(EL\) takes values in~\(A\) although the action
  is not closed.  And the map
  \(A\rtimes_\alg S \to A\rtimes_\ess S\) is not injective.  In
  contrast, the map \(A\rtimes_\alg S \to A\rtimes_\red S\) is
  always injective.
\end{example}

\begin{remark}
  \label{rem:essential_crossed_not_functorial}
  Unlike the reduced and full crossed products, the essential
  crossed product is not functorial for \Star{}homomorphisms.  This
  is because the local multiplier algebra is not functorial.  The
  problem is very visible for quotient maps.  If \(I\subseteq A\) is
  an essential ideal, then a local multiplier of~\(A\) does not
  descend to anything on~\(A/I\).  For instance, consider the
  restriction map \(\Cont[0,1] \to \C\), \(f\mapsto f(0)\), in the
  situation of Example~\ref{exa:group_bundle}.  This is
  \(S\)\nb-equivariant, where~\(S\) acts on~\(\C\) by the trivial
  action, with~\(0\) acting by the ideal~\(\{0\}\).  Now
  \(\C\rtimes S \cong \Cst(G)\), and the weak conditional
  expectation \(\Cst(G) \to \C\) is the usual trace, which has
  values in \(\C=\C''=\Locmult(\C)\).  Hence the essential and
  reduced crossed products are the same for the action on~\(\C\).
  But the fibre restriction map \(A\rtimes_\red S \to \Cst_\red(G)\)
  does not factor through \(A\rtimes_\ess S = A\).
\end{remark}

In order to relate \(EL\colon A\rtimes_{\ess} S\to \Locmult(A)\) to
\(E\colon A\rtimes S\to A''\), we use
Lemma~\ref{lem:atomic_weak_expectation}.  It says that the
\Star{}homomorphism
\(\varrho\colon A'' \to \prod_{\pi\in\dual{A}} \Bound(\Hils_\pi)\)
is isometric on \(E(A\rtimes S)\).  We shall embed~\(\Locmult(A)\)
into a quotient of \(\prod_{\pi\in\dual{A}} \Bound(\Hils_\pi)\).

\begin{definition}
  \label{def:essential_supremum}
  A subset of~\(\dual{A}\) is \emph{nowhere dense} if its closure
  has empty interior.  It is \emph{meagre} if it may be written as a
  countable union of nowhere dense subsets.  It is \emph{comeagre}
  if its complement is meagre.  For
  \((f_\pi)_{\pi\in\dual{A}}\in \prod_{\pi\in\dual{A}}
  \Bound(\Hils_\pi)\), define
  \[
    \norm*{(f_\pi)_{\pi\in\dual{A}}}_\ess \defeq
    \inf
    {}\setgiven*{ \sup_{\pi\in R} {}\norm{f_\pi}}{
      R\subseteq\dual{A} \textup{ comeagre}}.
  \]
\end{definition}

By definition, a subset is meagre if and only if it is contained in
a union of countably many closed subsets with empty interior.  Thus
a subset is comeagre if and only if it contains an intersection of
countably many dense open subsets.
Any countable intersection of comeagre subsets is again comeagre,
and a subset that contains a comeagre subset is also comeagre.  It
follows that the infimum in Definition~\ref{def:essential_supremum}
is a minimum.  And the set of comeagre subsets of~\(\dual{A}\) is
directed by~\(\supseteq\).  The
values~\(\sup_{\pi\in R} {}\norm{f_\pi}\) for comeagre subsets
\(R\subseteq\dual{A}\) form a monotone net indexed by this directed
set.  Thus
\[
  \norm{(f_\pi)}_\ess
  = \min_{\twolinesubscript{R\subseteq\dual{A}}{\mathrm{comeagre}}}
  \sup_{\pi\in R} {}\norm{f_\pi}= \lim_{\twolinesubscript{R\subseteq\dual{A}}{\mathrm{comeagre}}}
  \sup_{\pi\in R} {}\norm{f_\pi}.
\]

\begin{proposition}
  \label{pro:Locmult_as_operator_families}
  The function~\(\norm{\cdot}_\ess\) on
  \(\prod_{\pi\in\dual{A}} \Bound(\Hils_\pi)\) is a
  \(\Cst\)\nb-seminorm.  Let~\(D\) be the quotient of\/
  \(\prod_{\pi\in\dual{A}} \Bound(\Hils_\pi)\) by the null space
  of~\(\norm{\cdot}_\ess\).  The faithful representation
  \(A \hookrightarrow \prod_{\pi\in\dual{A}} \Bound(\Hils_\pi)\)
  factors through to \(A \hookrightarrow D\), and the latter extends
  to an isometric \Star{}homomorphism
  \(\iota\colon \Locmult(A) \hookrightarrow D\) such that the
  following diagram commutes:
  \[
    \begin{tikzcd}
      A\rtimes S \arrow[r, "E"] \arrow[d, "EL"] &
      A'' \arrow[r, "\varrho"] &
      \prod_{\pi\in\dual{A}} \Bound(\Hils_\pi) \arrow[d, ->>, "q"] \\
      \Locmult(A) \arrow[rr, hookrightarrow, "\iota"] && D
    \end{tikzcd}
  \]
  If \(b\in A\rtimes S\), then \(EL(b)= 0\) if and only if
  \(\setgiven{\pi \in \dual{A}}{\pi''(E(b))=0}\) is comeagre
  in~\(\dual{A}\).
\end{proposition}

\begin{proof}
  If \(R\subseteq \dual{A}\) is a comeagre subset, then the supremum
  of~\(\norm{f_\pi}\) for \(\pi\in R\) is a \(\Cst\)\nb-seminorm on
  \(\prod_{\pi\in\dual{A}} \Bound(\Hils_\pi)\).  As the pointwise
  limit of these \(\Cst\)\nb-seminorms, \(\norm{\cdot}_\ess\) is a
  \(\Cst\)\nb-seminorm as well.  Since
  \(\prod_{\pi\in\dual{A}} \Bound(\Hils_\pi)\) is already a
  \(\Cst\)\nb-algebra, the quotient~\(D\) by the null space of this
  \(\Cst\)\nb-seminorm is complete, hence a \(\Cst\)\nb-algebra.

  We define~\(\iota\).  Let \(J\subseteq A\) be an essential ideal.
  Then \(\dual{J} \subseteq \dual{A}\) is a dense open subset, hence
  comeagre.  If \(\pi\in\dual{J}\), then~\(\pi\) extends uniquely to
  a representation~\(\bar\pi\) of~\(\Mult(J)\), and so
  \(\bar\pi(a)\in\Bound(\Hils_\pi)\) is defined for all
  \(a\in\Mult(J)\).  This defines a unital \Star{}homomorphism
  \(\iota_J\colon \Mult(J) \to D\).  If \(K\subseteq J\), then the
  composite of~\(\iota_K\) with the restriction map
  \(\Mult(J) \to \Mult(K)\) is equal to~\(\iota_J\).  Hence the
  maps~\(\iota_J\) for all essential ideals \(J\subseteq A\) combine
  to a unital \Star{}homomorphism \(\iota\colon \Locmult(A) \to D\).
  To see that \(\iota\) is isometric, it suffices to prove that
  \(\norm{a} \le \norm{\iota_J(a)}\) holds for all essential ideals
  \(J\subseteq A\) and all \(a\in\Mult(J)\).  Let \(\varepsilon>0\)
  and \(C\defeq \norm{a}_{\Mult(J)}\).  The function
  \(\dual{\Mult(J)}\ni \overline{\pi} \to \norm{\overline{\pi}(a)}\)
  has the supremum~\(\norm{a}\) and is lower semicontinuous (see
  \cite{Dixmier:Cstar-algebras}*{Proposition~3.3.2}),
  and~\(\dual{J}\) is an open dense subset in \(\dual{\Mult(J)}\).
  Therefore,
  \(U\defeq \setgiven{\pi \in \dual{J} }{ \norm{\bar\pi(a)}>
    C-\varepsilon}\) is a non-empty open subset of
  \(\dual{J}\subseteq \dual{A}\).
  Since~\(\dual{A}\) is a Baire space, every comeagre subset
  \(R\subseteq \dual{A}\) is also dense in~\(\dual{A}\). Thus
  \(R\cap U \neq \emptyset\), and the supremum of
  \(\norm{\bar\pi(a)}\) for \(\pi \in R\) is at least
  \(C-\varepsilon\).  Then \(\norm{\iota_J(a)} > C-\varepsilon\)
  follows.  Since \(\varepsilon>0\) is arbitrary, this implies
  \(\norm{\iota_J(a)} \ge C\).  Hence
  \(\iota\colon \Locmult(A) \to D\) is isometric.

  Now we show that the diagram in the proposition commutes.  Let
  \(b\in A\rtimes_\alg S\).
  Then \(b = \sum_{t\in F} \xi_t \delta_t\) for a finite subset
  \(F\subseteq S\) and \(\xi_t\in\Hilm_t\).  Let
  \(J_t \defeq I_{1,t} \oplus I_{1,t}^\bot\) and
  \(J = \bigcap_{t\in F} J_t\).  These are essential ideals
  in~\(A\).  We have already seen that \(E(b)\cdot a \in J\) for all
  \(a\in J\), and \(EL(b)\) is the resulting multiplier of~\(J\).
  If \(\pi\in \dual{J} \subseteq \dual{A}\), then
  \(\pi''(E(b)) = \bar\pi(EL(b))\).  Since~\(\dual{J}\) is comeagre,
  \(q\) maps \(\varrho\circ E(b)\) to \(\iota(EL(b))\).  Hence the
  diagram commutes on all elements of \(A\rtimes_\alg S\).  Since
  \(A\rtimes_\alg S\) is dense in \(A\rtimes S\) and all maps
  involved are contractive, it commutes on all of \(A\rtimes S\).
\end{proof}

\begin{theorem}
  \label{thm:essential_crossed_expectation_faithful}
  The \(\Locmult\)-expectation
  \(EL\colon A\rtimes S \to \Locmult(A)\) is symmetric and thus
  descends to a faithful \(\Locmult\)-expectation
  \(A\rtimes_\ess S \to \Locmult(A)\).
\end{theorem}

\begin{proof}
  The two statements in the assertion are equivalent by
  Lemma~\ref{lem:reduced_almost_faithful}.  Let~\(\Hilm_t\) be the
  Hilbert \(A\)\nb-bimodules that define the \(S\)\nb-action
  on~\(A\).  We embed these into \(A\rtimes_\ess S\) as usual.
  Arguing as in the beginning of the proof of
  Theorem~\ref{the:crossed_expectation_faithful}, it suffices to
  show that if \(b\in A\rtimes_{\ess} S\) satisfies \(EL(b^* b)=0\),
  then \(EL(c^* b^* b c)=0\) for all \(c\in \Hilm_t\), \(t\in S\).
  Thus let \(b\in A\rtimes_{\ess} S\) be such that \(EL(b^* b)=0\).
  By Proposition~\ref{pro:Locmult_as_operator_families}, the set of
  \(\omega\in\dual{A}\) with \(\omega''(E(b^* b))\neq 0\) is meagre.
  Equation~\eqref{eq:prepare_faithful} shows that
  \(\pi''(E(c^* b^* b c))\neq 0\) can only happen if
  \(\pi\in\s\bigl(\dual{\Hilm_t}\bigr)\) and
  \(\omega(E(b^* b))\neq0\) for \(\omega = \dual{\Hilm_t}(\pi)\).
  Therefore, the set
  \[
    \setgiven*{\pi \in \dual{A}}{\pi''(E(c^* b^* b c)))\neq 0}\subseteq  \dual{\Hilm_t}^{-1}\left(\setgiven*{\omega \in  \dual{\rg(\Hilm_t)} }{\omega''(E(b^* b))\neq 0}\right)
  \]
  is meagre as a subset of a meagre set.  Thus \(EL(c^* b^* b c)=0\)
  by Proposition~\ref{pro:Locmult_as_operator_families}.
\end{proof}

\begin{corollary}
  \label{cor:kernel_essential_expectation}
  \(\Null_{EL}= \setgiven*{b \in A\rtimes S}{ \setgiven{\pi \in \dual{A}}{\pi''(E(b^*b))=0}\text{ is comeagre}}\).
\end{corollary}

\subsection{Topological gradings}
\label{sec:top_grading}

The following
definition is an inverse semigroup analogue of
\cite{Exel:Amenability}*{Definition~3.4}.  We choose to use the
essential instead of the reduced crossed product here.

\begin{definition}
  \label{def:topological_grading}
  Let~\(B\) be a \(\Cst\)\nb-algebra with an \(S\)\nb-grading
  \((B_t)_{t\in S}\).  Let \(\pi\colon A\rtimes S \to B\) be the canonical \Star{}epimorphism
  as in Remark~\ref{rem:representation_vs_grading}.  The grading is called
  \emph{topological} and~\(B\) is called \emph{topologically
    \(S\)\nb-graded} if \(\ker \pi\) is contained in the
  kernel~\(\Null_{EL}\) of the quotient map
  \(A\rtimes S \to A\rtimes_\ess S\).
\end{definition}

By definition, a grading is topological if and only if there is a
(surjective) \Star{}homomorphism
\[
\varphi\colon B \to A\rtimes_\ess S
\]
with \(\varphi\circ \pi = \Lambda\).  So~\(B\) lies between the full
and the essential crossed products and may be called an \emph{exotic
  crossed product}.  Another equivalent characterisation when a
grading is topological is \(\norm{\pi(x)} \ge \norm{x}_\ess\) for
all \(x\in A\rtimes S\) or all \(x\in A\rtimes_\alg S\),
where~\(\norm{x}_\ess\) denotes the norm in the essential crossed
product \(A\rtimes_\ess S\).  Equivalently, \(B\) is isomorphic to a
completion of the \Star{}algebra \(A\rtimes_\alg S\) for
a \(\Cst\)\nb-norm that lies between the maximal and the essential
\(\Cst\)\nb-norm:

\begin{proposition}
  \label{pro:topological_grading_through_expectation}
  Let~\(B\) be an \(S\)\nb-graded \(\Cst\)\nb-algebra.  View the
  grading as an \(S\)\nb-action on its unit fibre~\(A\), and let
  \(\pi\colon A\rtimes S \to B\) be the induced surjective
  \Star{}homomorphism as in
  Remark~\textup{\ref{rem:representation_vs_grading}}.  Let
  \(EL\colon A\rtimes S \to \Locmult(A)\) be the canonical
  \(\Locmult\)-expectation.  The following are
  equivalent:
  \begin{enumerate}
  \item \label{enum:topological_grading_through_expectation_1}%
    the \(S\)\nb-grading on~\(B\) is topological;
  \item \label{enum:topological_grading_through_expectation_2}%
    there is an \(\Locmult\)-expectation
    \(P\colon B\to \Locmult(A)\) with \(P\circ \pi = EL\);
  \item \label{enum:topological_grading_through_expectation_3}%
    \(\norm{EL(x)}\le \norm{\pi(x)}\) for all
    \(x\in A\rtimes_\alg S\).
  \end{enumerate}
  If~\(P\) is an \(\Locmult\)-expectation as
  in~\ref{enum:topological_grading_through_expectation_2}, then the
  canonical surjective \Star{}homomorphism \(B\to A\rtimes_\ess S\)
  is an isomorphism if and only if~\(P\) is almost faithful, if and
  only if~\(P\) is faithful.
\end{proposition}

\begin{proof}
  The assertions mostly follow from
  Lemma~\ref{lem:abstract_uniqueness}.  That~\(P\) is almost
  faithful if and only if it is faithful follows from
  Theorem~\ref{thm:essential_crossed_expectation_faithful}.
\end{proof}

\subsection{Coincidence of the reduced and essential crossed products}
\label{sec:compare}

The essential crossed product has the advantage of having the
``right'' ideal structure, even for non-Hausdorff groupoids and
similar situations.  A serious disadvantage is that it is not
functorial (see Remark~\ref{rem:essential_crossed_not_functorial}).
This suggests that it is not the right object for K\nb-theory
computations.  A question like exactness does not even make sense
for it.  So it is desirable to know when the reduced and
essential crossed products are the same, meaning that the canonical
quotient map \(A\rtimes_\red S \to A\rtimes_\ess S\) is an
isomorphism.

Let \(E\colon A\rtimes S \to A''\) be the canonical weak conditional
expectation and let \(EL\colon A\rtimes S \to \Locmult(A)\) be the
canonical \(\Locmult\)-expectation.  By
definition, \(A\rtimes_\red S = A\rtimes_\ess S\) if and only if
\(\Null_E = \Null_{EL}\).  A case when this is particularly clear is
when~\(E\) is a genuine conditional expectation, taking values in
\(A\subseteq A''\).  Then~\(EL\) also takes values in
\(A\subseteq \Locmult(A)\) and \(EL=E\).  If~\(E\) is not
\(A\)\nb-valued, then \(\Null_{E}\subseteq \Null_{EL}\).  The
difference between \(A\rtimes_\red S\) and \(A\rtimes_{\ess} S\)
is measured by the ideal
\[
  J_\sing\defeq \Lambda(\Null_{EL})
\]
in \(A\rtimes_\red S\) because
\[
(A\rtimes_\red S)/J_\sing\cong A\rtimes_{\ess} S.
\]
We call elements of~\(J_\sing\) \emph{singular}.
Proposition~\ref{pro:Locmult_as_operator_families} describes the
kernels of \(E\) and~\(EL\) using the supremum and the essential
supremum norm of functions on~\(\dual{A}\), where the essential
supremum is defined as the supremum over comeagre subsets.  We use
it to describe~\(J_\sing\).  If \(a\in A''\), then we define
\[
  \nu_{a}\colon \dual{A} \to [0,\infty),\qquad
  \pi\mapsto \norm{\pi''(a)},
\]
where \(\pi''\colon A'' \to \Bound(\Hils_\pi)\) denotes the unique
extension of~\(\pi\) to~\(A''\).  If \(b\in A\rtimes S\), then
\[
  \norm{E(b)} = \norm{\varrho\circ E(b)}
  = \sup_{\pi\in\dual{A}}\nu_{E(b)}(\pi)=\norm{\nu_{E(b)}}_\infty
\]
by Lemma~\ref{lem:atomic_weak_expectation}.  Similarly,
Proposition~\ref{pro:Locmult_as_operator_families} gives
\[
  \norm{EL(b)} = \norm{\iota\circ EL(b)}
  = \min_{\twolinesubscript{R\subseteq\dual{A}}{\mathrm{comeagre}}}
  \sup_{\pi\in R}\nu_{E(b)}(\pi)
  \eqdef \norm{\nu_{E(b)}}_\ess.
\]
In particular, \(EL(b)=0\) if and only if
\(\setgiven{\pi\in\dual{A}}{\nu_{E(b)}(\pi)\neq 0}\) is meagre.

\begin{proposition}
  \label{pro:essential_norm}
  Let \(b\in A\rtimes S\).  The function
  \(\nu_{E(b)}\colon \dual{A} \to [0,\infty)\) is lower semicontinuous
  on a comeagre subset.  The following are equivalent:
  \begin{enumerate}
  \item \label{enu:essential_norm1}%
    \(\setgiven{\pi\in\dual{A}}{\nu_{E(b)}(\pi)\neq0}\) is meagre;
  \item \label{enu:essential_norm2}%
    \(\setgiven{\pi\in\dual{A}}{\nu_{E(b)}(\pi)\neq0}\) has empty
    interior;
  \item \label{enu:essential_norm3}%
    the subsets
    \(\setgiven{\pi\in\dual{A}}{\nu_{E(b)}(\pi)> \varepsilon}\) have
    empty interior for all \(\varepsilon>0\).
  \end{enumerate}
\end{proposition}

\begin{proof}
  Assume first that
  \(b= \sum_{t\in F} \xi_t \delta_t\in A\rtimes_\alg S\).  In the
  last part of the proof of
  Proposition~\ref{pro:Locmult_as_operator_families}, we have
  defined an essential ideal~\(J\) in~\(A\) such that
  \(EL(b) \in\Mult(J)\) and \(\pi''(E(b)) = \bar\pi(EL(b))\) for all
  \(\pi\in \dual{J} \subseteq \dual{A}\), where~\(\bar\pi\) denotes
  the unique extension of~\(\pi\) to~\(\Mult(J)\).  The map
  \(\pi\mapsto \bar\pi\) identifies~\(\dual{J}\) with an open subset
  in \(\dual{\Mult(J)}\), and the function
  \(\omega \to \norm{\omega(EL(b))}\) is lower semicontinuous
  on~\(\dual{\Mult(J)}\) by
  \cite{Dixmier:Cstar-algebras}*{Proposition~3.3.2}.
  Hence~\(\nu_{E(b)}\) is lower semicontinuous on~\(\dual{J}\),
  which is a dense open subset in~\(\dual{A}\).  Now let
  \(b\in A\rtimes S\).  Then there is a sequence \((b_n)_{n\in\N}\)
  in \(A\rtimes_\alg S\) with \(\lim b_n = b\).  Since~\(E\) is
  contractive, the sequence of functions~\(\nu_{E(b_n)}\) converges
  uniformly towards~\(\nu_{E(b)}\).  For each \(n\in\N\), there is
  an essential ideal~\(J_n\) as above such that~\(\nu_{b_n}\) is
  lower semicontinuous on~\(\dual{J_n}\).  The intersection
  \(Y\defeq \bigcap \dual{J_n}\) is
  comeagre.
  If \(\pi\in Y\), then the functions~\(\nu_{E(b_n)}\) are lower
  semicontinuous at~\(\pi\).  Since they converge uniformly
  to~\(\nu_{E(b)}\), the latter is also lower semicontinuous
  at~\(\pi\).

  The statement~\ref{enu:essential_norm1}
  implies~\ref{enu:essential_norm2} because~\(\dual{A}\) is a Baire
  space (see \cite{Dixmier:Cstar-algebras}*{Proposition~3.4.13}).
  And~\ref{enu:essential_norm2} clearly
  implies~\ref{enu:essential_norm3}.  To complete the proof, we show
  that not~\ref{enu:essential_norm1} implies
  not~\ref{enu:essential_norm3}.  Let~\(Y\) be the comeagre subset
  of~\(\dual{A}\) defined above.  If
  \(\setgiven{\pi\in\dual{A}}{\nu_{E(b)}(\pi)\neq0}\) is not meagre,
  then it cannot be contained in the complement of~\(Y\).  So there
  is \(\pi\in Y\) with \(\nu_{E(b)}(\pi)\neq0\).
  Since~\(\nu_{E(b)}\) is lower semicontinuous at~\(\pi\), it
  follows that \(\nu_{E(b)}(\omega)>\nu_{E(b)}(\pi)/2>0\) for
  all~\(\omega\) in some neighbourhood of~\(\pi\).
\end{proof}

\begin{corollary}
  \label{cor:singular_ideal_description}
  An element \(b\in A\rtimes_\red S\) belongs to~\(J_\sing\) if and
  only if the subset
  \(\setgiven{\pi\in\dual{A}}{\nu_{E_\red(b^*b)}(\pi)\neq0}\) is
  meagre in~\(\dual{A}\), if and only if this subset has empty
  interior, if and only if
  \(\setgiven{\pi\in\dual{A}}{\nu_{E_\red(b^*b)}(\pi)>\varepsilon}\)
  has empty interior in~\(\dual{A}\) for all \(\varepsilon >0\).
\end{corollary}

\begin{proof}
  Combine Corollary~\ref{cor:kernel_essential_expectation} and
  Proposition~\ref{pro:essential_norm}.
\end{proof}

\begin{corollary}
  \label{cor:ess_vs_reduced}
  \(A\rtimes_\ess S = A\rtimes_\red S\) if and only if for every
  \(b\in (A\rtimes_\red S)^+\setminus \{0\}\) there is
  \(\varepsilon>0\) such that
  \(\setgiven{\pi\in\dual{A}}{\nu_{E_r(b)}(\pi)> \varepsilon}\) has
  non-empty interior, if and only if for every
  \(b\in (A\rtimes_\red S)^+\setminus \{0\}\) the set
  \(\setgiven{\pi\in\dual{A}}{\nu_{E_r(b)}(\pi)\neq0}\) is not
  meagre.
\end{corollary}

\begin{corollary}
  \label{cor:sufficient_for_no_singular}
  If there is \(\varepsilon>0\) with
  \(\norm{\nu_{E(b)}}_\ess \ge \varepsilon\cdot
  \norm{\nu_{E(b)}}_\infty\) for all positive elements
  \(b\in (A\rtimes_\alg S)^+\), then
  \(A\rtimes_\ess S = A\rtimes_\red S\).
\end{corollary}

\begin{proof}
  By Proposition~\ref{pro:Locmult_as_operator_families}, the
  assumption implies
  \(\norm{EL(b^* b)} \ge \varepsilon \norm{E(b^* b)}\) for all
  \(b\in A\rtimes_\alg S\).  This inequality extends by continuity
  to all \(b\in A\rtimes S\) and then implies
  \(\lNull_{EL} = \lNull_E\).  Since both \(E\) and \(EL\) are
  symmetric, the latter is equivalent to \(\Null_{EL} = \Null_E\)
  and then to \(A\rtimes_\ess S = A\rtimes_\red S\).
\end{proof}

The condition in Corollary~\ref{cor:sufficient_for_no_singular} has
the advantage to involve only elements of~\(A\rtimes_\alg S\),
making it more checkable.  It is unclear, however, whether it is
necessary.

\begin{remark}
  The reader may readily rephrase the above results to characterise
  when \(\ker EL = \ker E\).  Namely, this is equivalent to
  conditions as above for all elements in \(A\rtimes_\red S\), not
  only in the positive cone \((A\rtimes_\red S)^+\).
\end{remark}

\subsection{The local multiplier algebra and the injective hull}
\label{sec:pseudo-expectations}

Let~\(I(A)\) be the injective hull of~\(A\)
(see~\cite{Hamana:Injective-Envelope-Cstar}).  The canonical
embedding \(A\hookrightarrow I(A)\) factors through the inclusion
\(A\subseteq \Locmult(A)\) and an embedding
\(\iota\colon \Locmult(A) \hookrightarrow I(A)\) by
\cite{Frank:Injective_local_multiplier}*{Theorem~1}.  And this is an
isomorphism if~\(A\) is commutative.  Hence every
\(\Locmult\)-expectation is a pseudo-expectation, and both the kernel and
the ideal~\(\Null_E\) do not change when we view an
\(\Locmult\)-expectation as a pseudo-expectation.  The two notions
are exactly the same if~\(A\) is commutative.

Let \(A=\Cont_0(X)\) be commutative.  Let \(\mathfrak{B}(X)\supseteq A\) be the
\(\Cst\)\nb-algebra of all bounded Borel functions on~\(X\).  The
subset
\[
  \mathfrak{M}(X)\defeq \setgiven{f\in \mathfrak{B}(X)}{ f \text{ vanishes on a comeagre set}}
\]
is an ideal in \(\mathfrak{B}(X)\) with \(\Cont_0(X)\cap \mathfrak{M}(X)=0\).  So
\(\Cont_0(X)\) embeds into \(\mathfrak{B}(X)/\mathfrak{M}(X)\). 
 Gonshor has identified
\(I(\Cont_0(X))\) with the algebra \(\mathfrak{B}(X)/\mathfrak{M}(X)\) (see
\cite{Gonshor:Injective_hulls_II}*{Theorem~1}).  In fact, this
follows from a much earlier result of
Dixmier~\cite{Dixmier:Sur_espaces_Stone}.
So
\[
  \Locmult(\Cont_0(X)) \cong \mathfrak{B}(X)/\mathfrak{M}(X) \cong I(\Cont_0(X))
\]
for any locally compact Hausdorff space~\(X\).  For Gonshor,
injectivity means that if \(A\to B\) is an injective
\Star{}homomorphism to another commutative \(\Cst\)\nb-algebra~\(B\)
such that~\(A\) detects ideals in~\(B\), then there is an injective
\Star{}homomorphism \(B \hookrightarrow \Locmult(A)\) that respects
the inclusions of~\(A\).  Roughly speaking, \(\Locmult(A)\) is the
largest commutative \(\Cst\)\nb-algebra in which~\(A\) detects
ideals.

If~\(A\) is simple, then \(\Locmult(A) = \Mult(A)\).
In particular, if~\(A\) is simple and unital, then
\(\Locmult(A) = A\).  A simple unital \(\Cst\)\nb-algebra need not
be injective; an example is the Calkin algebra,
see~\cite{Hamana:Injective-Envelope-Cstar}.  So~\(I(A)\) differs
from \(\Locmult(A)\) in general.

\section{Aperiodic inclusions and  generalised expectations}
\label{sec:hidden_ideal}

In this section, we fix a general \(\Cst\)\nb-inclusion
\(A \subseteq B\).  More generally, we often treat an injective
\Star{}homomorphism \(A\hookrightarrow B\) as if it were an
inclusion.  Recall that \(\Ideals(B)\) denotes the lattice of
(closed, two-sided) ideals in~\(B\).  We say that~\emph{\(A\)
  detects ideals in~\(B\)} if \(J\cap A\neq 0\) for any ideal
\(J\in\Ideals(B)\) with \(J\neq0\).  This is a fundamental property
in the study of the ideal structure of reduced crossed products (see
\cites{Archbold-Spielberg:Topologically_free,
  Sierakowski:IdealStructureCrossedProducts,
  Kwasniewski-Meyer:Aperiodicity, Kwasniewski-Meyer:Stone_duality}).
The usual assumptions that guarantee it imply a stronger property,
namely, that the full crossed product has a unique quotient in which
the coefficient algebra detects ideas.  In this section, we study
this generalised intersection property.  We introduce the concept of
aperiodicity for the \(\Cst\)\nb-inclusion \(A\subseteq B\), which
implies the generalised intersection property.  If, in addition,
\(E\colon B\to\tilde{A}\) is a generalised expectation, then we find
a criterion when~\(A\) detects ideals in~\(B/\Null_E\): this happens
if the conditional expectation is supportive.  We show that any
\(\Locmult\)-expectation has this property.  This
gives general simplicity and pure infiniteness criteria for~\(B\)
when~\(E\) is almost faithful.

\subsection{Generalised intersection property and hidden ideal}

Before we begin with the study of aperiodicity, we prove some
elementary results about detection of ideals in quotients of~\(B\).
Let
\[
\Ideals_0(B) \defeq \setgiven{J\in \Ideals(B)}{J\cap A=0}.
\]
So~\(A\) detects ideals in~\(B\) if and only if
\(\Ideals_0(B) = \{0\}\).

\begin{lemma}
  \label{lem:detection_in_quotient}
  Let \(J\in\Ideals(B)\).  The composite map \(A \to B \to B/J\) is
  injective if and only if \(J \in \Ideals_0(B)\).  If this is the
  case, then the inclusion \(A\hookrightarrow B/J\) detects ideals
  if and only if~\(J\) is a maximal element of\/~\(\Ideals_0(B)\).
\end{lemma}

\begin{proof}
  The first claim holds because the kernel of the composite map
  \(A \to B/J\) is \(A\cap J\).  To prove the second claim, we use
  that any ideal in~\(B/J\) is the image of an ideal
  \(K\in\Ideals(B)\) with \(J\subseteq K\).  And \(A\cap K = 0\) is
  the preimage in~\(A\) of the ideal in~\(B/J\) corresponding
  to~\(K\).  So~\(A\) does not detect ideals in~\(B/J\) if and only
  if there is \(K\in\Ideals_0(B)\) with \(J\subsetneq K\).
\end{proof}

\begin{example}
  \label{exa:inclusion_no_hidden}
  Consider the unital inclusion
  \(\C\subseteq \Cst(\Z) \cong \Cont(\T)\).  An ideal
  \(J\in\Ideals(\Cont(\T))\) satisfies \(J\cap\C=0\) if and only if
  it is proper.  Thus \(\Ideals_0(\Cont(\T))\) has many different
  maximal elements, namely, all the maximal ideals
  \(\Cont_0(\T\setminus\{z\})\) for \(z\in\T\).  The resulting
  quotients are all isomorphic to~\(\C\), in which~\(\C\) detects
  ideals.  There is, however, no canonical way to choose one of
  these quotients.
\end{example}

\begin{lemma}
  \label{lem:hidden_ideals}
  We have \(0\in\Ideals_0(B)\).  If \(J_1,J_2\in \Ideals(B)\)
  satisfy \(J_1 \subseteq J_2\) and \(J_2\in\Ideals_0(B)\), then
  \(J_1\in\Ideals_0(B)\).  If \(X\subseteq \Ideals_0(B)\) with the
  partial order~\(\subseteq\) is directed, then the supremum
  of~\(X\) in~\(\Ideals(B)\) belongs to~\(\Ideals_0(B)\).  Any
  element of\/~\(\Ideals_0(B)\) is contained in a maximal element
  of\/~\(\Ideals_0(B)\).
\end{lemma}

\begin{proof}
  The claim about subideals and \(0\in \Ideals_0(B)\) are obvious.
  Let \(X\subseteq \Ideals_0(B)\) be directed.  Its supremum
  in~\(\Ideals(B)\) is the closure~\(K\) of the union
  \(\bigcup_{J\in X} J\), which is equal to the closed linear span
  because~\(X\) is directed.  If \(a\in A\) and \(J\in X\), then the
  distance between~\(a\) and~\(J\) is~\(\norm{a}\) because
  \(A\hookrightarrow B/J\) is injective.  Hence the distance
  between~\(a\) and~\(K\) is~\(\norm{a}\).  So \(K\in\Ideals_0(B)\).
  It follows that any chain of ideals in~\(\Ideals_0(B)\) has a
  supremum in~\(\Ideals_0(B)\).  So the set of elements
  in~\(\Ideals_0(B)\) containing a given \(J\in\Ideals_0(B)\) has a
  maximal element by Zorn's Lemma.  This element remains maximal
  in~\(\Ideals_0(B)\).
\end{proof}

\begin{proposition}
  \label{prop:hidden_ideal}
  The subset\/~\(\Ideals_0(B)\) has a unique maximal element if and
  only if \(J_1 + J_2 \in \Ideals_0(B)\) for all
  \(J_1,J_2 \in \Ideals_0(B)\).  If~\(\Null\) is this unique maximal
  element, then
  \(\Ideals_0(B) = \setgiven{J\in\Ideals(B)}{J\subseteq\Null}\).
\end{proposition}

\begin{proof}
  By Lemma~\ref{lem:hidden_ideals}, \(J_1 + J_2 \in \Ideals_0(B)\)
  holds for all \(J_1,J_2 \in \Ideals_0(B)\) once it holds for all
  maximal elements of~\(\Ideals_0(B)\).  If \(J_1 \neq J_2\) for two
  maximal elements of~\(\Ideals_0(B)\), then
  \(J_1 + J_2 \supsetneq J_1,J_2\) and hence
  \(J_1 + J_2 \notin\Ideals_0(B)\) because otherwise \(J_1,J_2\)
  would not be maximal.  Hence \(J_1 + J_2 \in \Ideals_0(B)\) for
  all \(J_1,J_2 \in \Ideals_0(B)\) if and only if~\(\Ideals_0(B)\)
  has a unique maximal element.  Since~\(\Ideals(B)\) is a complete
  lattice, a subset of~\(\Ideals(B)\) is of the form
  \(\setgiven{J\in\Ideals(B)}{J\subseteq\Null}\) for some
  \(\Null\in\Ideals(B)\) if and only if it has~\(\Null\) as a unique
  maximal element.
\end{proof}

\begin{definition}
  \label{def:visible}
  The inclusion \(A\subseteq B\) has the \emph{generalised
    intersection property} if there is a unique maximal ideal
  \(\Null\in\Ideals(B)\) with \(\Null\cap A =0\).  Then we
  call~\(\Null\) the \emph{hidden ideal}, and we define
  \(B_\ess \defeq B/\Null\).
\end{definition}

Proposition~\ref{prop:hidden_ideal} shows that the hidden ideal
exists if and only if the sum of two ideals in \(\Ideals_0(B)\)
remains in \(\Ideals_0(B)\), if and only if
\(\Ideals_0(B) = \setgiven{J\in\Ideals(B)}{J\subseteq \Null}\) for
some \(\Null\in\Ideals(B)\), and then~\(\Null\) is the hidden ideal.

\begin{definition}
  A \(\Cst\)\nb-subalgebra \(A\subseteq B\) is called
  \emph{\(B\)\nb-minimal} if \(\overline{BIB}=B\) for all
  \(0\neq I\in \Ideals(A)\).
\end{definition}

In the following proposition, and in the whole paper, we assume
\(A\neq 0\).

\begin{proposition}
  \label{pro:simple_criterion}
  A \(\Cst\)\nb-inclusion \(A\subseteq B\) has the generalised
  intersection property if and only if there is a unique
  quotient~\(B/\Null\) of~\(B\) such that~\(A\cap \Null=0\) and the image of \(A\) detects ideals in~\(B/\Null\).  Then~\(\Null\) is the
  hidden ideal and \(B_\ess=B/\Null\).

  If the \(\Cst\)\nb-inclusion \(A\subseteq B\) has the generalised
  intersection property and~\(A\) is \(B\)\nb-minimal,
  then~\(B_\ess\) is the unique simple quotient of~\(B\) for which
  the quotient map is injective on~\(A\).  Moreover, \(B\) is simple
  if and only if~\(A\) detects ideals in~\(B\) and~\(A\) is
  \(B\)\nb-minimal.
\end{proposition}

\begin{proof}
  The first part of the assertion follows from
  Lemma~\ref{lem:detection_in_quotient} and
  Proposition~\ref{prop:hidden_ideal}.  Assume that the
  \(\Cst\)\nb-inclusion \(A\subseteq B\) has the generalised
  intersection property and that~\(A\) is \(B\)\nb-minimal.
  Since~\(A\) detects ideals in \(B_\ess=B/\Null\), every non-zero
  ideal in~\(B_\ess\) is the image of an ideal \(J\in\Ideals(B)\),
  with \(\Null \subseteq J\) and \(J\cap A\neq 0\).  Then
  \(B = B(J\cap A) B \subseteq J\) by minimality.  So the ideal in
  question is~\(B_\ess\).  That is, \(B_\ess\) is simple.  If~\(J\)
  is any ideal in~\(\Ideals_0(B)\) such that \(B/J\) is simple,
  then~\(A\) detects ideals in~\(B/J\) and hence \(B/J=B_\ess\) by
  the first part.

  As a result, if~\(A\) detects ideals in~\(B\) and~\(A\) is
  \(B\)\nb-minimal, then \(B=B_\ess\) is simple.  Conversely,
  if~\(A\) does not detect ideals in~\(B\), then there is an ideal
  \(J\subseteq B\) with \(J\neq0\) and \(J\cap A = 0\).  It cannot
  be \(0\) or~\(B\).  So~\(B\) is not simple.  If~\(A\) is not
  \(B\)\nb-minimal, then there is \(0\neq I\in \Ideals(A)\) with
  \(\overline{BIB}\neq B\).  Then \(\overline{BIB}\in\Ideals(B)\) is
  not~\(B\) and not~\(0\) because it contains~\(I\).  So~\(B\) is
  not simple.
\end{proof}

\begin{proposition}
  \label{pro:abstract_uniqueness}
  Let \(A\subseteq B\) be a \(\Cst\)\nb-inclusion with a generalised
  expectation \(E\colon B\to \tilde{A} \supseteq A\).  The following
  are equivalent:
  \begin{enumerate}
  \item \label{enu:abstract_uniqueness0}%
    \(\Null_E\) is the hidden ideal for the inclusion
    \(A\subseteq B\);
  \item \label{enu:abstract_uniqueness0.5}%
    \(B/\Null_E\) is the unique quotient of~\(B\) for which the
    induced map \(A\to B/\Null_E\) is injective and detects ideals;
  \item \label{enu:abstract_uniqueness1}%
    if \(J\in\Ideals(B)\) satisfies \(J\cap A=0\), then
    \(J\subseteq \ker E\);
  \item \label{enu:abstract_uniqueness1b}%
    if \(J\in\Ideals(B)\) satisfies \(J\cap A=0\), then
    \(J\subseteq \Null_E\);
  \item \label{enu:abstract_uniqueness2}%
    for every \(\Cst\)\nb-algebra~\(C\) and \Star{}homomorphism
    \(\pi\colon B\to C\) that is injective on~\(A\),
    there is a \Star{}homomorphism \(\varphi\colon \pi(B)\to B/\Null_E\) with
    \(\varphi\circ \pi = \Lambda\colon B \to B/\Null_E\);
  \item \label{enu:abstract_uniqueness3}%
    for every \(\Cst\)\nb-algebra~\(C\) and \Star{}homomorphism
    \(\pi\colon B\to C\) that is injective on~\(A\),
    there is a generalised expectation
    \(E_\pi\colon \pi(B)\to \tilde{A}\) with \(E_\pi\circ \pi = E\).
  \end{enumerate}
  In particular, \(A\) detects ideals in~\(B\) if and only if the
  statements above hold and~\(E\) is almost faithful.
\end{proposition}

\begin{proof}
  Statements \ref{enu:abstract_uniqueness1}
  and~\ref{enu:abstract_uniqueness1b} are equivalent by the
  definition of~\(\Null_E\).  Since \(\Null_E \cap A = 0\),
  \ref{enu:abstract_uniqueness1b} is equivalent
  to~\ref{enu:abstract_uniqueness0}.  And
  \ref{enu:abstract_uniqueness0}
  and~\ref{enu:abstract_uniqueness0.5} are equivalent by
  Proposition~\ref{pro:simple_criterion}.  Any ideal~\(J\) in~\(B\)
  is the kernel of a \Star{}homomorphism \(B \to B/J\).  And a
  \Star{}homomorphism~\(\pi\) of~\(B\) is faithful on~\(A\) if and
  only if its kernel \(J\defeq \ker \pi\) satisfies \(J\cap A=0\).
  Hence
  \ref{enu:abstract_uniqueness1b}--\ref{enu:abstract_uniqueness3}
  are equivalent by Lemma~\ref{lem:abstract_uniqueness}.

  If~\(A\) detects ideals in~\(B\), then \(\Null_E=0\) is the hidden
  ideal and~\(E\) is almost faithful.  Conversely, if~\(\Null_E\) is
  the hidden ideal and~\(E\) is almost faithful, then \(\Null_E=0\)
  and~\(A\) detects ideals in~\(B\)
  by~\ref{enu:abstract_uniqueness0.5}.
\end{proof}

\subsection{Aperiodic inclusions}
\label{sec:aperiodic_hidden}

Kishimoto's condition for automorphisms was extended to bimodules
in~\cite{Kwasniewski-Meyer:Aperiodicity} in order to generalise the
known criteria for detection of ideals in reduced crossed products
for group actions to Fell bundles over groups.  Here we rename
Kishimoto's condition, speaking more briefly of \emph{aperiodicity}.

Let \(\Her(A)\) denote the set of \emph{non-zero}, hereditary
\(\Cst\)\nb-subalgebras of~\(A\).  Let~\(A^+\) be the cone of
positive elements in~\(A\).

\begin{definition}[\cite{Kwasniewski-Meyer:Aperiodicity}]
  \label{def:aperiodic_E}
  Let~\(X\) be a normed \(A\)\nb-bimodule.  We say that \(x\in X\)
  satisfies \emph{Kishimoto's condition} if, for all \(D\in \Her(A)\)
  and \(\varepsilon>0\), there is \(a\in D^+\) with \(\norm{a}=1\)
  and \(\norm{a x a}<\varepsilon\).  We call~\(X\) \emph{aperiodic}
  if Kishimoto's condition holds for all \(x\in X\).
\end{definition}

\begin{lemma}
  \label{lem:anti-Kishimoto}
  Consider the \(\Cst\)\nb-algebra \(A\) as an \(A\)\nb-bimodule.
  No non-zero positive element of~\(A\) satisfies Kishimoto's
  condition.
\end{lemma}

\begin{proof}
  Given \(b\in A^+\) with \(\norm{b}=1\),
  \cite{Kwasniewski-Meyer:Aperiodicity}*{Lemma~2.9} provides a
  hereditary subalgebra \(D_0\subseteq A\) such that
  \(\norm{x b x} \ge \norm{x^2} - \norm{x b x - x^2} >
  (1-\varepsilon)\norm{x}^2\) for all \(x\in D_0^+\).
\end{proof}

\begin{lemma}[\cite{Kwasniewski-Meyer:Aperiodicity}*{Lemma~4.2}]
  \label{lem:Kish_vector_subspace}
  The subset of elements in a normed \(A\)\nb-bimodule that satisfy
  Kishimoto's condition is a closed vector subspace.
\end{lemma}

\begin{lemma}
  \label{lem:Kish_hereditary}
  Subbimodules, quotient bimodules, extensions, finite direct sums,
  and inductive limits of aperiodic normed \(A\)\nb-bimodules remain
  aperiodic.  If \(f\colon X\to Y\) is a bounded \(A\)\nb-bimodule
  homomorphism with dense range and~\(X\) is aperiodic, then so
  is~\(Y\).  If \(D\in\Her(A)\), then an aperiodic \(A\)\nb-bimodule
  is also aperiodic as a \(D\)\nb-bimodule.  If \(J\in\Ideals(A)\)
  is an essential ideal and~\(X\) an \(A\)\nb-bimodule, then
  \(J X J\) is aperiodic as a \(J\)\nb-bimodule if and only if~\(X\)
  is aperiodic as an \(A\)\nb-bimodule.
\end{lemma}

\begin{proof}
  The claims about subbimodules and about aperiodicity as a
  \(D\)\nb-bimodule are trivial.  If \(f\colon X\to Y\) is a bounded
  \(A\)\nb-bimodule homomorphism and \(x\in X\), then~\(f(x)\)
  inherits Kishimoto's condition from~\(x\).  Hence the second
  statement follows from Lemma~\ref{lem:Kish_vector_subspace}.  This
  implies the claim about quotient bimodules.  The claim about
  extensions is proved as in the proof of
  \cite{Kwasniewski-Meyer:Aperiodicity}*{Lemma~4.2}.  Namely, let
  \(M_1 \into M_2 \onto M_3\) be an extension of \(A\)\nb-bimodules
  such that \(M_1\) and~\(M_3\) are aperiodic.  Let \(x\in M_2\),
  \(\varepsilon>0\), and \(D\in\Her(A)\).  We may assume without
  loss of generality that \(\norm{x}=1\).  Since~\(M_3\) is
  aperiodic, there is \(a_0\in D^+\) with \(\norm{a_0}=1\) and
  \(\norm{a_0 x a_0+M_1}_{M_3}<\varepsilon/2\).  Hence
  \(a_0 x a_0 = x_1 + x_2\) with \(x_1\in M_1\) and
  \(\norm{x_2}<\varepsilon\).  By
  \cite{Kwasniewski-Meyer:Aperiodicity}*{Lemma~2.9}, there is
  \(D_0\in \Her(A)\) such that \(D_0\subseteq D\) and
  \(\norm{a - a\cdot a_0} \le \varepsilon \norm{a}\) for all
  \(a\in D_0\).  Kishimoto's condition for~\(x_1\) gives
  \(a\in D_0^+ \subseteq D^+\) with \(\norm{a}=1\) and
  \(\norm{a x_1 a}<\varepsilon\).  Then
  \(\norm{a x a}<4\varepsilon\).  Thus~\(M_2\) is aperiodic.

  A direct sum of two aperiodic normed \(A\)\nb-bimodules is also an
  extension, hence inherits aperiodicity.  By induction, this
  remains true for direct sums of finitely many aperiodic normed
  \(A\)\nb-bimodules; here the norm should be one that defines the
  product topology.  The claim about inductive limits follows from
  Lemma~\ref{lem:Kish_vector_subspace}.

  Now let \(J\in\Ideals(A)\) be an essential ideal and~\(X\) an
  \(A\)\nb-bimodule.  The claims in the lemma already proven show
  that~\(J X J\) is aperiodic as a \(J\)\nb-bimodule if~\(X\) is
  aperiodic as an \(A\)\nb-bimodule.  Conversely, assume~\(J X J\)
  to be aperiodic as a \(J\)\nb-bimodule.  The Cohen--Hewitt
  Factorisation Theorem shows that \(J X\), \(X J\) and \(J X J\)
  are closed \(J\)\nb-subbimodules in~\(X\).  There are extensions
  \(J X J \into J X \onto J X/J X J\) and
  \(J X \into X \onto X/ J X\).  The quotients in both are aperiodic
  as \(J\)\nb-bimodules because they satisfy \(x a=0\) or \(a x=0\)
  for all \(a\in J\), respectively.  Hence the claim about
  extensions shows that~\(X\) is also aperiodic as a
  \(J\)\nb-bimodule.  Let \(D\in\Her(A)\), \(x\in X\), and
  \(\varepsilon>0\).  Since~\(J\) is essential, the intersection
  \(D\cap J\) is still non-zero.  It is a hereditary
  \(\Cst\)\nb-subalgebra in~\(J\), and Kishimoto's condition
  for~\(x\) gives \(a\in (D\cap J)^+\) with \(\norm{a}=1\) and
  \(\norm{a x a}<\varepsilon\).  This witnesses that~\(X\) is
  aperiodic as an \(A\)\nb-bimodule.
\end{proof}

\begin{remark}
  An infinite direct sum of aperiodic normed bimodules inherits
  aperiodicity when it is given a norm that defines the product
  topology on each finite sub-sum.  This follows from
  Lemma~\ref{lem:Kish_hereditary} by viewing it as an inductive
  limit of finite direct sums.
\end{remark}

For any \(\Cst\)\nb-inclusion \(A\subseteq B\), both \(A\) and~\(B\)
are naturally normed \(A\)\nb-bimodules.  So is the quotient Banach
space~\(B/A\) with the quotient norm.

\begin{definition}
  A \(\Cst\)\nb-inclusion \(A\subseteq B\) is \emph{aperiodic} if
  the Banach \(A\)\nb-bimodule \(B/A\) is aperiodic.
\end{definition}

\begin{proposition}
  \label{pro:aperiodicity_vs_quotients}
  If \(A\subseteq B\) is aperiodic and \(A\subseteq C \subseteq B\),
  then the inclusion \(A\subseteq C\) is aperiodic.  If
  \(A\subseteq B\) is aperiodic and \(J\in\Ideals(B)\) satisfies
  \(J\cap A = 0\), then the induced inclusion
  \(A \hookrightarrow B/J\) is aperiodic.  Let \(I\in\Ideals(A)\).
  If \(A\subseteq B\) is aperiodic, then \(I\subseteq I B I\) is
  aperiodic; conversely, \(A\subseteq B\) is aperiodic if
  \(I\subseteq I B I\) is aperiodic and~\(I\) is essential.
\end{proposition}

\begin{proof}
  This follows from Lemma~\ref{lem:Kish_hereditary} because \(C/A\)
  and \(I B I/I\) are isometrically isomorphic to
  \(A\)\nb-subbimodules of~\(B/A\) and \((B/J)/A \cong B/(J+A)\) is
  isometrically isomorphic to a quotient bimodule of~\(B/A\).
\end{proof}

\begin{proposition}
  \label{pro:inclusion_aperiodic_visible_ideal}
  Let \(A\subseteq B\) be aperiodic and \(J\in\Ideals(B)\).  Then
  \(J\cap A = 0\) if and only if~\(J\) is an aperiodic
  \(A\)\nb-bimodule.
\end{proposition}

\begin{proof}
  First assume \(J\cap A\neq0\).  Then there is \(b\in J\cap A\)
  with \(b\neq0\).  Then~\(b^* b\) is a non-zero positive element of
  \(J\cap A\).  By Lemma~\ref{lem:anti-Kishimoto}, it does not
  satisfy Kishimoto's condition.  Then~\(J\) is not aperiodic.
  Another proof goes as follows.  The intersection \(J\cap A\) is an
  \(A\)\nb-subbimodule in~\(J\) and an ideal in~\(A\).
  Lemma~\ref{lem:anti-Kishimoto} shows that \(J\cap A\) is not
  aperiodic as a \(J\cap A\)-bimodule.  Then
  Lemma~\ref{lem:Kish_hereditary} implies that~\(J\) is not
  aperiodic as an \(A\)\nb-bimodule.

  Now assume that
  \(J\cap A = 0\).  Then the composite \Star{}homomorphism
  \(A\hookrightarrow B \onto B/J\) is injective, hence isometric.
  So its image is closed.  Thus the map \(A\oplus J \to B\),
  \((a,x)\mapsto a+x\), is a continuous bijection onto a closed
  subspace of~\(B\).  It follows that it is a topological
  isomorphism.  Then the injective map \(J \hookrightarrow B/A\) is
  also a topological isomorphism onto its image.  Since~\(B/A\) is
  aperiodic by assumption, so is~\(J\) by
  Lemma~\ref{lem:Kish_hereditary}.
\end{proof}

\begin{theorem}
  \label{the:aperiodic_hidden_ideal}
  Every aperiodic inclusion \(A\subseteq B\) has the generalised
  intersection property.  The hidden ideal is the largest ideal that
  is aperiodic as an \(A\)\nb-bimodule.
\end{theorem}

\begin{proof}
  Let~\(\Ideals_a(B)\) be the set of all ideals in~\(B\) that are
  aperiodic as an \(A\)\nb-bimodule.  Certainly,
  \(0\in\Ideals_a(B)\).  Let \(J_1,J_2\in\Ideals_a(B)\).  Then there
  is a Banach \(A\)\nb-bimodule extension
  \(J_1 \into J_1 + J_2 \onto J_2/(J_1 \cap J_2)\).  So
  Lemma~\ref{lem:Kish_hereditary} implies
  \(J_1 + J_2\in \Ideals_a(B)\).  Hence
  \(\Ideals_a(B)\subseteq\Ideals(B)\) is closed under finite joins.
  Lemma~\ref{lem:Kish_hereditary} implies that~\(\Ideals_a(B)\) is
  also closed under increasing unions.  Hence it is closed under
  arbitrary joins.  Let~\(\Null\) be the join of all ideals
  in~\(\Ideals_a(B)\).  Then \(J\in\Ideals(B)\) satisfies
  \(J\subseteq \Null\) if and only if \(J\in\Ideals_a(B)\).
  Proposition~\ref{pro:inclusion_aperiodic_visible_ideal} says that
  \(\Ideals_a(B) = \Ideals_0(B)\).  So~\(\Null\) is the hidden
  ideal.
\end{proof}

\begin{remark}
  \label{rem:max_aperiodic_ideal}
  Let \(A\subseteq B\) be a \(\Cst\)\nb-inclusion that need not be
  aperiodic.  The proof of Theorem~\ref{the:aperiodic_hidden_ideal}
  still gives \(\Null\in\Ideals(B)\) such that \(J\in\Ideals(B)\) is
  aperiodic as an \(A\)\nb-bimodule if and only if
  \(J\subseteq\Null\).  That is, \(\Null\) is the largest aperiodic
  ideal in~\(B\).  The proof of
  Proposition~\ref{pro:inclusion_aperiodic_visible_ideal} still
  shows that \(\Null \cap A = 0\).  So we get an induced inclusion
  \(A\hookrightarrow B/\Null\).  Lemma~\ref{lem:Kish_hereditary}
  implies that no non-zero ideal in~\(B/\Null\) is aperiodic as an
  \(A\)\nb-bimodule.  But~\(A\) need not detect ideals
  in~\(B/\Null\).  For instance, in
  Example~\ref{exa:inclusion_no_hidden}, the largest aperiodic ideal
  is~\(0\) because a unital \(\C\)\nb-bimodule is never aperiodic.
\end{remark}

\subsection{Supportive generalised expectations}
\label{sec:aperiodicity_with_expectation}

Now assume the inclusion \(A\subseteq B\) to be aperiodic and let
\(E\colon B\to \tilde{A} \supseteq A\) be a generalised expectation.
Then \(A\subseteq B\) has the generalised intersection property by
Theorem~\ref{the:aperiodic_hidden_ideal}.  The hidden
ideal~\(\Null\) contains~\(\Null_E\) because \(\Null_E \cap A = 0\).
When is~\(\Null\) equal to~\(\Null_E\)?  We cannot expect this for
all generalised expectations.  For instance, for the trivial
generalised expectation in Example~\ref{exa:trivial_gen_expectation}
we always have \(\Null_E=0\) independently of~\(\Null\).  More
importantly, there are examples of aperiodic inclusions coming from
non-Hausdorff groupoids where \(\Null_E\neq \Null\) for the
canonical weak conditional expectation \(E\colon B\to A''\).  We are
going to identify an extra property of generalised conditional
expectations that implies that the hidden ideal is~\(\Null_E\).
Even more, it implies that the positive elements in~\(A\) support
all positive elements in~\(B/\Null_E\).

\begin{definition}
  \label{def:supportive_weak_conditional_expectation}
  A generalised expectation \(E\colon B\to \tilde{A} \supseteq A\)
  is called \emph{supportive} if, for any \(b\in B^+\) with
  \(E(b)\neq0\), there are \(\delta>0\) and a hereditary
  \(\Cst\)\nb-subalgebra \(D\in\Her(A)\) such that \(\norm{xE(b) x}\ge
  \delta\) for all \(x\in D^+\) with \(\norm{x}=1\).
\end{definition}

\begin{remark}
  By definition, \(E\) is supportive if and only if no non-zero
  element of~\(E(B^+)\) satisfies Kishimoto's condition.  Then
  any aperiodic ideal in~\(B\) is contained in~\(\ker E\) and hence
  in~\(\Null_E\).  If \(A\subseteq B\) is aperiodic, then
  Proposition~\ref{pro:inclusion_aperiodic_visible_ideal} implies
  that~\(\Null_E\) is aperiodic because \(\Null_E\cap A = 0\); so
  the unique maximal aperiodic ideal is~\(\Null_E\) if
  \(A\subseteq B\) is aperiodic and~\(E\) is supportive.  This short
  argument justifies our definition of supportive conditional
  expectations.  Theorem~\ref{the:aperiodic_consequences} will prove
  the stronger statement that~\(A\) supports~\(B\).
\end{remark}

\begin{remark}
  Since the property of being supportive depends only on \(E(B^+)\),
  a generalised expectation \(E\colon B\to \tilde{A} \supseteq A\)
  is supportive if and only if the corresponding reduced generalised
  expectation \(B/\Null_E \to \tilde{A}\) is supportive.
\end{remark}

\begin{proposition}
  \label{pro:local_expectation_supportive}
  Any \(\Locmult\)-expectation \(E\colon B\to \Locmult(A)\) and, in
  particular, any genuine conditional expectation \(E\colon B\to A\)
  is supportive.
\end{proposition}

\begin{proof}
  Let \(b\in B^+\) satisfy \(E(b) \neq0\).  Then there is
  \(\varepsilon>0\) with \(\norm{E(b)}>\varepsilon\) and
  \((1-\varepsilon)^2>1/2\).  By the definition of \(\Locmult(A)\)
  as an inductive limit, there are an essential ideal
  \(I\subseteq A\) and \(c\in \Mult(I)\subseteq \Locmult(A)\) with
  \(\norm{E(b)-c}<\varepsilon/4\).  Since \(E(b)\ge 0\), we may
  assume without loss of generality that \(c\ge0\).  Then
  \(\norm{c}> 3\varepsilon/4\).  Hence there is \(a\in I\) with
  \(0 \le a \le 1\), \(\norm{a}=1\), and
  \(\norm{a c^{1/2}}> (3\varepsilon/4)^{1/2}\).  Let
  \(d \defeq (c^{1/2} a^2 c^{1/2})^{1/2} \in I\).
  \cite{Kwasniewski-Meyer:Aperiodicity}*{Lemma~2.9} gives a
  hereditary \(\Cst\)\nb-subalgebra \(D\in\Her(A)\) such that
  \(\norm{x d - x}<\varepsilon \norm{x} \cdot \norm{d}\) and
  \(\norm{x d}>(1-\varepsilon) \norm{x} \cdot \norm{d}\) for all
  \(x\in D\).  Let \(x\in D^+\) satisfy \(\norm{x}=1\).  Then
  \(x c x \ge x c^{1/2} a^2 c^{1/2} x = x d (x d)^*\).  Hence
  \[
    \norm{x c x} \geq \norm{x d}^2
    > (1-\varepsilon)^2 \norm{d}^2
    > (1-\varepsilon)^2\cdot 3\varepsilon/4
    >  3\varepsilon/8.
  \]
  Then
  \[
    \norm{x E(b) x} \ge \norm{x c x} - \norm{x}^2 \norm{E(b)-c}
    > 3\varepsilon/8 - \varepsilon/4
    = \varepsilon/8.
  \]
  Since this holds for all \(x\in D^+\) with \(\norm{x}=1\), \(E\)
  is supportive.
\end{proof}

\begin{lemma}
  \label{lem:supportive_condition}
  A generalised expectation \(E\colon B\to \tilde{A} \supseteq A\)
  is supportive if, for any \(b\in B^+\) with \(E(b)\neq0\), there is
  \(a\in A^+\setminus\{0\}\) with \(a \le E(b)\).
\end{lemma}

\begin{proof}
  Choose \(0<\delta < \norm{a^{1/2}}\) and let \(\varepsilon = 1 -
  \delta/\norm{a^{1/2}}\).  The element \(\norm{a}^{-1/2}a^{1/2} \in
  A^+\) has norm~\(1\), and
  \cite{Kwasniewski-Meyer:Aperiodicity}*{Lemma~2.9} gives a hereditary
  subalgebra \(D\in\Her(A)\) with \(\norm{a^{1/2} x}>
  (1-\varepsilon)\norm{a^{1/2}} \norm{x} = \delta \norm{x}\) for all
  \(x\in D\).  Let \(x\in D\).  Since \(x^* E(b) x \ge x^* a x\), we
  may estimate \(\norm{x^* E(b) x} \ge \norm{x^* a x} = \norm{a^{1/2}
    x}^2 \ge \delta^2 \norm{x}^2\) as desired.
\end{proof}

\begin{definition}[\cite{Kwasniewski:Crossed_products}*{Definition~2.39}]
  We say that \emph{\(A^+\) supports~\(B\)} if, for every \(b\in
  B^+\setminus\{0\}\), there is \(a\in A^+\setminus\{0\}\) with \(a
  \precsim b\) in the Cuntz preorder
  (see~\cite{Cuntz:Dimension_functions}); that is, for every
  \(\varepsilon>0\), there is \(x \in B\) with \(\norm{a-x^* b x}
  <\varepsilon\).
\end{definition}

\begin{definition}[\cite{Kwasniewski-Szymanski:Pure_infinite}*{Lemma~2.1}]
  \label{def:infinite}
  An element \(a\in B^+\setminus \{0\}\) is \emph{infinite} in~\(B\)
  if and only if there is \(b\in B^+\setminus\{0\}\) such that for
  all \(\varepsilon >0\) there are \(x,y\in a B\) with
  \(\norm{x^*x-a}<\varepsilon\), \(\norm{y^*y-b}<\varepsilon\) and
  \(\norm{x^*y}<\varepsilon\).
\end{definition}

\begin{proposition}
  \label{prop:support_vs_pure_infinite}
  If~\(B\) is simple, then~\(B\) is purely infinite if and only if
  \(A^+\subseteq B\) supports~\(B\) and all elements of
  \(A^+\setminus\{0\}\) are infinite in~\(B\).
\end{proposition}

\begin{proof}
  Since~\(B\) is simple, every infinite element in~\(B\) is properly
  infinite by
  \cite{Kirchberg-Rordam:Non-simple_pi}*{Proposition~3.14}.  Hence
  by \cite{Kirchberg-Rordam:Non-simple_pi}*{Theorem~4.16}, \(B\) is
  purely infinite if and only if all elements of
  \(B^+\setminus\{0\}\) are infinite in~\(B\).  If~\(B\) is purely
  infinite, then elements of \(A^+\setminus\{0\}\) are infinite
  in~\(B\), and~\(A^+\) supports~\(B\) by
  \cite{Kirchberg-Rordam:Non-simple_pi}*{Definition~4.1} and the
  simplicity of~\(B\).  Conversely, assume that \(A^+\)
  supports~\(B\) and that all elements of \(A^+\setminus\{0\}\) are
  infinite in~\(B\).  Let \(b\in B^+\setminus\{0\}\).  Then there is
  \(a\in A^+\setminus\{0\}\) with \(a \precsim b\).  Since~\(B\) is
  simple, \(b\in \overline{BaB} = B\).  This implies \(b\precsim a\)
  by \cite{Kirchberg-Rordam:Non-simple_pi}*{Proposition~3.5}.  Hence
  \(a\) and~\(b\) are Cuntz equivalent.  So~\(b\) is infinite.
\end{proof}

\begin{lemma}
  \label{lem:generalised_support}
  If \(\Null \in \Ideals_0(B)\) is such that for every \(b\in B^+\)
  with \(b\notin \Null\), there is \(a\in A^+\setminus\{0\}\) with
  \(a \precsim b\), then~\(\Null\) is the hidden ideal for
  \(A\subseteq B\) and~\(A^+\) supports~\(B/\Null\).  In particular,
  if~\(A^+\) supports~\(B\), then~\(A\) detects ideals in~\(B\).
\end{lemma}

\begin{proof}
  Let \(J\in\Ideals(B)\) satisfy \(J\not\subseteq \Null\).  Then
  there is \(b\in J^+\setminus \Null\).  By assumption, there is
  \(a\in A^+\setminus\{0\}\) with \(a \precsim b\).  This implies
  \(a\in B b B \subseteq J\), and so \(J\cap A \neq\{0\}\).
  Therefore, if \(J\cap A=\{0\}\) for \(J\in\Ideals(B)\), then
  \(J\subseteq \Null\).  The converse also holds because
  \(\Null \cap A=0\).  So~\(\Null\) is the hidden ideal.  And
  \(a \precsim b\) in~\(B\) implies \(a \precsim q(b)\), where
  \(q\colon B\to B/\Null\) is the quotient map.  Hence~\(A^+\)
  supports~\(B/\Null\).
\end{proof}

\begin{theorem}
  \label{the:aperiodic_consequences}
  Let \(A\subseteq B\) be an aperiodic \(\Cst\)\nb-inclusion.
  Let~\(E\) be an \(\Locmult\)-expectation or,
  more generally, a supportive generalised expectation
  \(E\colon B\to \tilde{A} \supseteq A\).  Then
  \begin{enumerate}
  \item \label{enu:aperiodic_consequences1}%
    for every \(b\in B^+\) with \(b\notin \Null_E\), there is
    \(a\in A^+\setminus\{0\}\) with \(a \precsim b\);
  \item \label{enu:aperiodic_consequences2}%
    \(A^+\) supports \(B/\Null_E\);
  \item \label{enu:aperiodic_consequences3}%
    \(A\) detects ideals in~\(B/\Null_E\);
  \item \label{enu:aperiodic_consequences4}%
    \(\Null_E\) is the hidden ideal for the inclusion
    \(A\subseteq B\), and so all the equivalent statements in
    Proposition~\textup{\ref{pro:abstract_uniqueness}} hold;
  \item \label{enu:aperiodic_consequences5}%
    \(B\) is simple if and only if \(A\) is \(B\)-minimal and~\(E\)
    is almost faithul;
  \item \label{enu:aperiodic_consequences6}%
    if~\(B\) is simple, then~\(B\) is purely infinite if and only if
    all elements of \(A^+\setminus\{0\}\) are infinite in~\(B\).
  \end{enumerate}
\end{theorem}

\begin{proof}
  \ref{enu:aperiodic_consequences1} implies
  \ref{enu:aperiodic_consequences2}--\ref{enu:aperiodic_consequences4}
  by Lemma~\ref{lem:generalised_support}.
  \ref{enu:aperiodic_consequences4} implies
  \ref{enu:aperiodic_consequences5} by
  Proposition~\ref{pro:simple_criterion},
  and~\ref{enu:aperiodic_consequences2}
  implies~\ref{enu:aperiodic_consequences6} by
  Proposition~\ref{prop:support_vs_pure_infinite}.  So everything
  follows once we show~\ref{enu:aperiodic_consequences1}.  Moreover,
  \(\Locmult\)-expectations are supportive by
  Proposition~\ref{pro:local_expectation_supportive}.  So we may
  assume~\(E\) to be a supportive generalised expectation.
	
  Let \(b\in B^+\) with
  \(b\notin \Null_E\).  Then~\(b^{1/2}\notin\Null_E\).
  Proposition~\ref{pro:GNS_kernels} gives \(x\in B^+\) with
  \(E(x^* b x)\neq0\).  Since \(x^* b x \precsim b\), we may
  replace~\(b\) by~\(x^* b x\) and assume without loss of generality
  that \(E(b)\neq0\).
  Since~\(E\) is supportive, there are \(\delta>0\) and
  \(D\in\Her(A)\) such that \(\norm{h E(b) h} \ge \delta\) for all
  \(h\in D^+\) with \(\norm{h}=1\).  Choose
  \(\varepsilon\defeq \delta/4\).  Since~\(B/A\) is aperiodic, there
  is \(h\in D^+\) with \(\norm{h}=1\) and
  \(\norm{h b h}_{B/A} < \varepsilon\).  That is, there is \(c\in A\)
  with \(\norm{h b h - c}< \varepsilon\).  Since~\(h b h\) is
  self-adjoint, even positive, we may replace \(c\) by its real part
  \((c^*+c)/2\).  This makes~\(c\) self-adjoint and does not increase
  \(\norm{h b h - c}\).  Next, decompose~\(c\) into its positive and
  negative parts, \(c= c_+ - c_-\) with \(c_\pm\ge0\) and
  \(c_+\cdot c_- = c_- \cdot c_+ = 0\).  We claim that
  \(\norm{h b h - c_+}< 2\varepsilon\).  First,
  \(\norm{h b h - c}<\varepsilon\) implies
  \(\varepsilon \ge h b h - c \ge -c = c_- - c_+\) because
  \(h b h \ge 0\).  Since \(c_\pm\) are orthogonal, this implies
  \(\varepsilon \ge \norm{c_-} = \norm{c - c_+}\).  So
  \(\norm{h b h - c_+}\le \norm{h b h - c} + \norm{c-c_+}<
  2\varepsilon\).  Now we estimate
  \[
  \norm{c_+}
  = \norm{E(c_+)}
  \ge  \norm{h E(b) h} - \norm{E(h b h - c_+)}
  > \delta - 2\varepsilon = 2\varepsilon.
  \]
  Hence \((c_+-\varepsilon)_+ \neq0\).  Let
  \(a\defeq (c_+ - \varepsilon)_+ \in A^+ \setminus\{0\}\).  Since
  \(\norm{h b h - c_+} < 2\varepsilon\),
  \cite{Kirchberg-Rordam:Infinite_absorbing}*{Lemma~2.2} gives a
  contraction \(y\in B\) with \(a = y^* h b h y\).  Thus
  \(a\precsim b\).
\end{proof}

By Proposition~\ref{pro:aperiodicity_vs_quotients}, the reduced inclusion
\(A\hookrightarrow B/\Null_E\) is aperiodic if \(A\subseteq B\) is
aperiodic.  The converse is false:

\begin{example}
  Embed \(A=\C\) diagonally into \(B=\C\oplus \C\) and define
  \(E\colon B\to A\), \(E(x,y) \defeq x\).  Then the reduced inclusion
  is an isomorphism \(A \cong B/\Null_E\) and hence aperiodic.  But the
  inclusion \(A\subseteq B\) is not aperiodic.
\end{example}

\section{Aperiodicity for inverse semigroup actions}
\label{sec:aperiodic_isg}

In this section, we characterise when the inclusion
\(A\subseteq A\rtimes S\) for an inverse semigroup action is
aperiodic.  In this case, \(A\rtimes_\ess S\) is the unique quotient
of \(A\rtimes S\) in which~\(A\) embeds and detects ideals.  Even
more, \(A^+\) supports \(A\rtimes_\ess S\).  Using
previous results in~\cite{Kwasniewski-Meyer:Aperiodicity}, we relate
aperiodicity of an inverse semigroup action to topological freeness
of the dual groupoid and pure outerness of the action.  Following
Archbold and Spielberg~\cite{Archbold-Spielberg:Topologically_free},
we also show directly that a variant of topological freeness implies
detection of ideals.

\subsection{Aperiodic inverse semigroup actions}
\label{sec:aperiodic_isg_characterisation}

\begin{definition}
  \label{def:aperiodic_action}
  An inverse semigroup action~\(\Hilm\) is called \emph{aperiodic}
  if the Hilbert \(A\)\nb-bimodules \(\Hilm_t \cdot I_{1,t}^\bot\)
  are aperiodic for all \(t\in S\), where~\(I_{1,t}^\bot\) is
  defined in~\eqref{eq:I1t_orthogonal}.
\end{definition}

Recall the inverse semigroup \(\Slice(A,B)\) for a regular inclusion
defined in Proposition~\ref{prop:regular_vs_inverse_semigroups}.

\begin{lemma}
  \label{lem:aperiodicity_vs_slices}
  Let \(A\subseteq B\) be a regular \(\Cst\)\nb-inclusion and let
  \(\mathcal{S}\subseteq \Slice(A,B)\) be a subset with closed
  linear span~\(B\); for instance, \(\mathcal{S}= \Slice(A,B)\).
  Let \((N\cap A)^\bot\) for \(N\in \Slice(A,B)\) be the annihilator
  of the ideal~\(N\cap A\) in~\(A\).  The inclusion \(A\subseteq B\)
  is aperiodic if and only if the image of \(N\cdot (N\cap A)^\bot\)
  in~\(B/A\) is an aperiodic \(A\)\nb-bimodule for all
  \(N\in\mathcal{S}\).
\end{lemma}

\begin{proof}
  Since \(\sum N\) is dense in~\(B\), the images of the
  \(A\)\nb-submodules \(N \subseteq B\) in~\(B/A\) are linearly
  dense.  By Lemma~\ref{lem:Kish_vector_subspace}, \(B/A\) is
  aperiodic if and only if these images, equipped with the quotient
  norm from~\(B/A\), are all aperiodic.  Fix \(N\in\mathcal{S}\) and
  put \(I\defeq N\cap A\).  Let \(J \defeq I + I^\bot\).  This is an
  essential ideal in~\(A\) and
  \(N\cdot J = N\cdot I \oplus N\cdot I^\bot = N\cap A \oplus N\cdot
  I^\bot\).  By Lemma~\ref{lem:Kish_hereditary}, the image of~\(N\)
  in~\(B/A\) is an aperiodic \(A\)\nb-bimodule if and only
  if the image of \(J\cdot N\cdot J\) in~\(B/A\) is an aperiodic
  \(J\)\nb-bimodule, if and only if the image of
  \(J\cdot N\cdot J\) in~\(B/A\) is an aperiodic
  \(A\)\nb-bimodule.  And the same proof works with \(N\cdot J\)
  instead of \(J\cdot N\cdot J\).  Finally, the subspace
  \(N\cdot J\) has the same image in~\(B/A\) as
  \(N\cdot (N\cap A)^\bot\).
\end{proof}

\begin{proposition}
  \label{pro:aperiodic_isg}
  Let~\(B\) be an \(S\)\nb-graded \(\Cst\)\nb-algebra with a grading
  \(\Hilm=(\Hilm_t)_{t\in S}\).  Let \(A\defeq \Hilm_1\subseteq B\)
  and turn~\(\Hilm\) into an \(S\)\nb-action on~\(A\).  If this
  action is aperiodic, then the inclusion \(A\subseteq B\) is
  aperiodic.  The converse holds if the grading is topological as in
  Definition~\textup{\ref{def:topological_grading}}.
\end{proposition}

\begin{proof}
  Suppose first that~\(\Hilm\) is aperiodic.  Any \(S\)\nb-grading
  satisfies \(\Hilm_t\subseteq \Hilm_u\) for \(t \le u\) in~\(S\).
  Hence the ideal~\(I_{1,t}\) used in
  Definition~\ref{def:aperiodic_action} is contained in
  \(\Hilm_t\cap A\) for all \(t\in S\).  So
  \(\Hilm_t \cdot (\Hilm_t \cap A)^\bot\) is a subbimodule of
  \(\Hilm_t \cdot I_{1,t}^\bot\) and hence inherits aperiodicity
  from the latter by Lemma~\ref{lem:Kish_hereditary}.  Thus
  \(A\subseteq B\) is aperiodic by
  Lemma~\ref{lem:aperiodicity_vs_slices}.

  Conversely, assume that \(A\subseteq B\) is aperiodic.  Let
  \(t\in S\).  Using that the grading is topological, we are going
  to prove that the seminorm on~\(\Hilm_t I_{1,t}^\bot\) induced by
  the quotient norm on~\(B/A\) is equal to the usual norm in~\(B\).
  Hence~\(\Hilm_t I_{1,t}^\bot\) inherits aperiodicity from~\(B/A\)
  by Lemma~\ref{lem:Kish_hereditary}.  It is clear that
  \(\norm{x}_{B/A} \le \norm{x}_B\) for all \(x\in B\).  It remains
  to prove the opposite inequality for
  \(x\in \Hilm_t I_{1,t}^\bot\).  Since
  \(\norm{x}_B^2 = \norm{x^* x}_A\), it follows that
  \(\norm{x}_B = \norm{x}_{A\rtimes_\ess S}\).  By
  Proposition~\ref{pro:topological_grading_through_expectation},
  \(\norm{x}_{B/A} \ge \norm{x}_{(A\rtimes_\ess S)/A}\).  So we may
  assume without loss of generality that \(B=A\rtimes_\ess S\).  By
  definition of \(EL\), \(EL(a^* x)=0\) for all \(a\in A\).  Hence
  \(x \odot 1\) and \(a \odot 1\) in
  \((A\rtimes_\alg S)\odot \Locmult(A)\) are orthogonal.  The
  following norms are computed in the Hilbert module
  completion~\(\Hilm[F]\) of
  \((A \rtimes_\alg S) \odot \Locmult(A)\) or in~\(\Locmult(A)\):
  \[
    \norm{(x+a)\odot 1}^2
    \ge \norm{x\odot 1}^2
    = \norm{\braket{x\odot 1}{x\odot 1}}
    = \norm{EL(x^* x)}
    = \norm{x^* x}
    = \norm{x}^2.
  \]
  The first inequality follows because \(x \odot 1\) and
  \(a \odot 1\) are orthogonal, and \(EL(x^* x) = x^* x\) because
  \(x^* x \in A\).  Hence the norm of the operator of left
  multiplication by \(x+a\) on~\(\Hilm[F]\) is at least~\(\norm{x}\)
  for all \(a\in A\).  This implies
  \(\norm{x}_{(A\rtimes_\ess S)/A} \ge \norm{x}_B\) as desired.
\end{proof}

\begin{example}
  \label{exa:aperiodic_inclusion_not_aperiodic_action}
  We are going to build a regular inclusion \(A\subseteq B\) with
  two different wide gradings, such that one grading gives an aperiodic
  action and the other not.  By Proposition~\ref{pro:aperiodic_isg},
  the inclusion \(A\subseteq B\) is aperiodic although it can be
  equipped with a wide grading which is not aperiodic.  So the
  equivalence in the second part of
  Proposition~\ref{pro:aperiodic_isg} fails for non-topological
  gradings, even if they are wide.

  Let~\(B\) be the \(\Cst\)\nb-algebra of all bounded functions
  \(f\colon [0,1] \to \C\) such that~\(f\) is continuous at
  irrational numbers and, at rational numbers in~\([0,1]\), still
  satisfies \(f(t) \defeq \lim_{x\nearrow t} f(x)\) and that the
  limit from the right \(\lim_{x\searrow t} f(x)\) exists.  Let
  \(A\defeq \Cont([0,1]) \subseteq B\).  A character on~\(B\) maps
  \(f\colon [0,1]\to\C\) to the value \(f(t)\) for some
  \(t\in[0,1]\) or to the right limit
  \(f(t^+) \defeq \lim_{x\searrow t} f(x)\) for some \(t\in\Q\).
  The set \(\R\setminus\Q\) of irrational numbers is dense in the
  spectrum of~\(B\).  This implies that~\(A\) detects ideals
  in~\(B\).  The inclusion \(A\subseteq B\) is unital, and it is
  regular by Lemma~\ref{lem:commutative_regular} because~\(B\) is
  commutative.  If \(u\in B\) is unitary, then
  \(u\cdot A \subseteq B\) belongs to \(\Slice(A,B)\) (see the proof
  of Lemma~\ref{lem:commutative_regular}).  These
  elements generate a subgroup of \(\Slice(A,B)\) that is isomorphic
  to the quotient group
  \(\Gamma \defeq \mathrm{U}(B)/\mathrm{U}(A)\).

  Let \(\Gamma_0\subseteq \Gamma\) be the subgroup of all unitaries
  that are discontinuous at only finitely many points of~\([0,1]\).
  These separate the characters of~\(B\) and hence generate~\(B\) as
  a \(\Cst\)\nb-algebra by the Stone--Weierstraß Theorem.  So~\(B\)
  is graded by the group~\(\Gamma_0\).  And then it is also graded
  by the larger group~\(\Gamma\).  If \(u\in\Gamma_0\), then
  \(f\in A\) satisfies \(f\cdot u \in A\) if and only if \(f(t)=0\)
  at all \(t\in[0,1]\) where~\(u\) is not continuous.  So
  \(A\cap u\cdot A = \Cont_0([0,1]\setminus F)\) for a finite set
  \(F\subseteq\Q\).  In order for this to be useful, the
  subgroup~\(\Gamma_0\) of \(\Slice(A,B)\) must be enlarged to the
  wide inverse subsemigroup that it generates.  Or we may as well
  take
  \(\bar\Gamma_0 \defeq \setgiven{u\cdot J}{u\in \Gamma_0,\
    J\idealin A}\).  Since \([0,1]\setminus F\) is dense
  in~\([0,1]\), the ideals \(I_{t,1}^\bot\) are zero for all
  \(t\in\bar\Gamma_0\).  Hence the \(\bar\Gamma_0\)\nb-action
  on~\(A\) defined by the \(\bar\Gamma_0\)\nb-grading of~\(B\) is
  aperiodic.

  This action is topologically principal as well because
  its isotropy is trivial in \([0,1] \setminus \Q\).  At the same
  time, this action is trivial: any slice \(u\cdot J\)
  in~\(\bar\Gamma_0\) is isomorphic to~\(J\) as an
  \(A\)\nb-bimodule, and the multiplication in the Fell bundle
  over~\(\bar\Gamma_0\) is also the usual multiplication in~\(A\).
  The \(\Gamma_0\)\nb-grading on~\(B\) is not aperiodic, however,
  because for any \(t\in\Gamma_0\setminus\{1\}\), there is no
  element in~\(\Gamma_0\) below \(1\) and~\(t\), so that
  \(I_{1,t}=0\).  We thank Jonathan Taylor for pointing this out to
  us.

  Let
  \(\bar\Gamma \defeq \setgiven{u\cdot J}{u\in \Gamma,\ J\idealin
    A}\).  Then~\(B\) carries a wide \(\bar\Gamma\)\nb-grading.
  There is a unitary \(u\in B\) that is discontinuous at all
  rational points in~\([0,1]\).  Then \(u\cdot f\) for \(f\in A\) is
  only continuous if \(f=0\).  So \(u\cdot A \cap A = 0\).  Then
  \(I_{A,u\cdot A}=0\) and aperiodicity of the
  \(\bar\Gamma\)\nb-grading on~\(B\) would ask for \(u\cdot A\) to
  be aperiodic as an \(A\)\nb-bimodule.  This contradicts
  Lemma~\ref{lem:anti-Kishimoto} because \(u\cdot A\) carries the
  standard \(A\)\nb-bimodule structure.  Hence the
  \(\bar\Gamma\)\nb-action on~\(A\) defined by the
  \(\bar\Gamma\)\nb-grading on~\(B\) is not aperiodic.
\end{example}

\subsection{Properties of aperiodic inverse semigroup actions}
\label{sec:properties_aperiodic_isg}

The following theorem specialises
Theorem~\ref{the:aperiodic_consequences} to aperiodic inverse
semigroup actions:

\begin{theorem}
  \label{the:aperiodic_action_consequences}
  Let~\(\Hilm\) be an aperiodic action of a unital inverse
  semigroup~\(S\) on a \(\Cst\)\nb-algebra~\(A\).  Let
  \(EL\colon A\rtimes S \to \Locmult(A)\) be the canonical
  \(\Locmult\)-expectation.
  \begin{enumerate}
  \item \label{en:aperiodic_uniqueness1}%
    The ideal \(\Null_{EL}\) is the hidden ideal of the
    \(\Cst\)\nb-inclusion \(A\subseteq A\rtimes S\).
  \item \label{en:aperiodic_uniqueness2}%
    \(A^+\) supports~\(A\rtimes_\ess S\).
  \item \label{en:aperiodic_uniqueness3}%
    \(A\) detects ideals in~\(A\rtimes_\ess S\), and
    \(A\rtimes_\ess S\) is the unique quotient of \(A\rtimes S\)
    with this property.
  \item \label{en:aperiodic_uniqueness11}%
    For any representation~\(\pi\) of~\((\Hilm_t)_{t\in S}\) in a
    \(\Cst\)\nb-algebra~\(B\) that is faithful on~\(A\), the image
    of the induced \Star{}homomorphism \(A\rtimes S\to B\) is an
    exotic crossed product, that is, \(\pi(B)\) is topologically
    graded by~\((\Hilm_t)_{t\in S}\).
  \end{enumerate}
\end{theorem}

\begin{proof}
  Proposition~\ref{pro:aperiodic_isg} implies that the inclusion
  \(A\subseteq A\rtimes S\) is aperiodic.  Then
  \ref{en:aperiodic_uniqueness1}--\ref{en:aperiodic_uniqueness3}
  follow from
  Theorem~\ref{the:aperiodic_consequences}.  
  In the situation of~\ref{en:aperiodic_uniqueness11}, the theorem
  implies \(\ker \pi \subseteq \Null_{EL}\).  And
  then~\ref{en:aperiodic_uniqueness11} follows easily.
\end{proof}

\begin{theorem}
  \label{the:simple_crossed_minimal}
  Let~\(B\) be an \(S\)\nb-graded \(\Cst\)\nb-algebra and turn the
  grading into an action~\(\Hilm\) of~\(S\) on~\(A\).  Then~\(B\) is
  simple if and only if~\(A\) detects ideals in~\(B\) and the
  action~\(\Hilm\) is minimal.  In particular, if the
  \(S\)\nb-action on~\(A\) is aperiodic and minimal, then
  \(A\rtimes_\ess S\) is simple.
\end{theorem}

\begin{proof}
  It is shown in
  \cite{Kwasniewski-Meyer:Stone_duality}*{Section~6.3} that an ideal
  in~\(A\) is of the form \(J\cap A\) for \(J\in\Ideals(B)\) if and
  only if it is invariant.  Therefore, the minimality assumption
  here is equivalent to the \(B\)\nb-minimality assumption in
  Proposition~\ref{pro:simple_criterion}, and the first claim
  follows.  The second claim follows because~\(A\) detects ideals
  in~\(A\rtimes_\ess S\) for any aperiodic action
  (Theorem~\ref{the:aperiodic_action_consequences}).
\end{proof}

\begin{corollary}
  \label{cor:simple_crossed_pi}
  Let~\(\Hilm\) be an aperiodic, minimal action of a unital inverse
  semigroup~\(S\) on a \(\Cst\)\nb-algebra~\(A\).  Then
  \(A\rtimes_\ess S\) is simple and purely infinite if and only if
  every element in \(A^+\setminus\{0\}\) is infinite in
  \(A\rtimes_\ess S\).
\end{corollary}

\begin{proof}
  This follows from the previous theorem and
  Proposition~\ref{prop:support_vs_pure_infinite}.
\end{proof}

\begin{proposition}
  \label{prop:visible_crossed_products}
  Let~\(\Hilm\) be an aperiodic action of a unital inverse
  semigroup~\(S\) on a \(\Cst\)\nb-algebra~\(A\).  Let
  \(\norm{\cdot}_{\min}\) be the infimum of all
  \(\Cst\)\nb-seminorms on \(A\rtimes_\alg S\) that restrict to the
  given \(\Cst\)\nb-norm on~\(A\).  This is a \(\Cst\)\nb-seminorm
  on \(A\rtimes_\alg S\).  The Hausdorff completion of
  \(A\rtimes_\alg S\) in this seminorm is canonically isomorphic to
  \(A\rtimes_\ess S\).
\end{proposition}

\begin{proof}
  Any \(\Cst\)\nb-seminorm on~\(A\rtimes_\alg S\) is of the
  form~\(\norm{\pi(x)}\) for a representation~\(x\).  And~\(\pi\) is
  injective on~\(A\) if and only if \(\norm{\pi(x)} = \norm{x}\) for
  all \(x\in A\).  Hence
  Theorem~\ref{the:aperiodic_action_consequences}.\ref{en:aperiodic_uniqueness11}
  says that the \(\Cst\)\nb-norm of~\(A\rtimes_\ess S\), restricted
  to~\(A\rtimes_\alg S\), is the minimal \(\Cst\)\nb-seminorm
  on~\(A\rtimes_\alg S\) that restricts to the given
  \(\Cst\)\nb-norm on~\(A\).
\end{proof}

\subsection{Aperiodicity versus topological freeeness and pure outerness}
\label{sec:aperiodic_top_free}

Aperiodicity of single Hilbert bimodules is shown
in~\cite{Kwasniewski-Meyer:Aperiodicity} to be equivalent to several
other properties, under some mild assumptions.  This allows us to
compare aperiodicity of inverse semigroup actions to the
non-triviality conditions in
Definition~\ref{def:non-triviality_groupoids}.

\begin{definition}
  A Hilbert \(A\)\nb-bimodule~\(\Hilm[H]\)
  over a \(\Cst\)\nb-algebra~\(A\)
  is \emph{purely outer} if there is no non-zero ideal
  \(J\in\Ideals(A)\)
  with \(\Hilm[H]\cdot J \cong J\) as a Hilbert bimodule
  (see~\cite{Kwasniewski-Meyer:Aperiodicity}).
  An action~\(\Hilm\) of an inverse semigroup on a
  \(\Cst\)\nb-algebra~\(A\) is \emph{purely outer} if the
  Hilbert \(A\)\nb-bimodules \(\Hilm_t\cdot I_{1,t}^\bot\) are
  purely outer for all \(t\in S\).
\end{definition}

\begin{definition}[\cite{Kwasniewski-Meyer:Aperiodicity}]
  A Hilbert \(A\)\nb-bimodule~\(\Hilm[H]\) over a
  \(\Cst\)\nb-algebra~\(A\) is \emph{topologically non-trivial} if
  the set
  \(\setgiven*{[\pi]\in \dual{A}}{\dual{\Hilm[H]}[\pi]=[\pi]}\) has
  empty interior.  Here \(\dual{\Hilm[H]}[\pi]=[\pi]\) means that
  the irreducible representations \(\Hilm[H] \otimes_A \pi\)
  and~\(\pi\) are unitarily equivalent.
\end{definition}

\begin{theorem}[\cite{Kwasniewski-Meyer:Aperiodicity}*{Theorem~8.1}]
  \label{the:non-trivial_Hilbi}
  Let~\(A\) be a \(\Cst\)\nb-algebra and let~\(\Hilm[H]\) be a
  Hilbert \(A\)\nb-bimodule.  Consider the following conditions:
  \begin{enumerate}
  \item \label{the:non-trivial_Hilbi_1}%
    \(\Hilm[H]\) is aperiodic;
  \item \label{the:non-trivial_Hilbi_2}%
    the partial homeomorphism \(\dual{\Hilm[H]}\) of~\(\dual{A}\) is
    topologically non-trivial;
  \item \label{the:non-trivial_Hilbi_3}%
    \(\Hilm[H]\) is purely outer.
  \end{enumerate}
  Then \ref{the:non-trivial_Hilbi_1}
  or~\ref{the:non-trivial_Hilbi_2}
  implies~\ref{the:non-trivial_Hilbi_3}.  If~\(A\) contains a
  separable essential ideal, then \ref{the:non-trivial_Hilbi_1}
  and~\ref{the:non-trivial_Hilbi_2} are equivalent.  If~\(A\)
  contains a simple, essential ideal, then
  \ref{the:non-trivial_Hilbi_1} and~\ref{the:non-trivial_Hilbi_3}
  are equivalent.  If~\(A\) contains an essential ideal of Type~I,
  then \ref{the:non-trivial_Hilbi_1}--\ref{the:non-trivial_Hilbi_3}
  are equivalent.
\end{theorem}


\begin{lemma}[\cite{Kwasniewski-Meyer:Aperiodicity}*{Proposition~9.7}]
  \label{lem:topological_freeness_vs_nontriviality}
  Assume that~\(A\) contains an essential ideal which is separable
  or whose spectrum is Hausdorff.  Let \((\Hilm[H]_i)_{i\in I}\) be
  a countable family of Hilbert \(A\)\nb-bimodules.  The following
  are equivalent:
  \begin{enumerate}
  \item \label{en:topological_freeness_vs_nontriviality1}%
    \(\dual{\Hilm[H]}_i\) is topologically non-trivial for every
    \(i\in I\);
  \item \label{en:topological_freeness_vs_nontriviality2}%
    the union
    \(\bigcup_{i\in I} \setgiven*{[\pi]\in
      \dual{A}}{\dual{\Hilm[H]_i}[\pi]=[\pi]}\) has empty interior.
  \end{enumerate}
\end{lemma}

\begin{proof}
  We only need to show that~\ref{en:topological_freeness_vs_nontriviality1}
  implies~\ref{en:topological_freeness_vs_nontriviality2}.  By
  \cite{Kwasniewski-Meyer:Aperiodicity}*{Corollary~6.9} we may assume
  that~\(A\) is separable or that~\(\dual{A}\) is Hausdorff (compare
  Corollary~\ref{cor:non-triviality_conditions_equiv}).
  If~\(\dual{A}\) is Hausdorff, then the sets
  \(\setgiven*{[\pi]\in \dual{A}}{\dual{\Hilm[H]_i}[\pi]=[\pi]}\) for
  \(i\in I\) are closed in~\(\dual{A}\).
  Hence~\ref{en:topological_freeness_vs_nontriviality1}
  implies~\ref{en:topological_freeness_vs_nontriviality2}
  because~\(\dual{A}\) is a Baire space.

  Now assume~\(A\) to be separable.  We will reduce our
  considerations to the case when the Hilbert
  bimodules~\(\Hilm[H]_i\) come from automorphisms of~\(A\).  We may
  assume that \(A\) and \((\Hilm[H]_i)_{i\in I}\) are embedded into
  a \(\Cst\)\nb-algebra~\(B\) in such a way that all the structure
  of these objects is inherited from~\(B\).  In addition, we may
  assume that the \(\Cst\)\nb-algebra~\(B\) is generated by \(A\)
  and \((\Hilm[H]_i)_{i\in I}\).  Let~\(\mathbb{F}\) be the free
  group on the set~\(I\).  We define an \(\mathbb{F}\)\nb-grading
  \((\Hilm_t)_{t\in\mathbb{F}}\) on~\(B\) with unit fibre~\(A\).
  Let
  \[
    \Hilm_t\defeq \Hilm[H]_t\quad\text{ if }t\in I,
    \quad \text{ and } \quad
    \Hilm_t\defeq \Hilm[H]_t^*\quad\text{ if }t\in I^{-1}.
  \]
  If \(t_1\cdots t_n\in \mathbb{F}\) is a reduced word, that is,
  \(t_k\in I\cup I^{-1}\) for \(1\le k \le n\) and
  \(t_k\neq t_{k+1}^{-1}\) for \(1\le k <n\), then
  \[
    \Hilm_{t_1\cdots t_n} \defeq  \Hilm_{t_1}\Hilm_{t_2}\dotsm \Hilm_{t_n}.
  \]
  One readily sees that \((\Hilm_t)_{t\in \mathbb{F}}\) is indeed an
  \(\mathbb{F}\)\nb-grading on~\(B\) and hence a Fell bundle
  over~\(\mathbb{F}\).  Let \(\gamma\colon \mathbb{F}\to \Aut(C)\)
  be the Morita globalisation of \((\Hilm_t)_{t\in \mathbb{F}}\) in
  \cite{Kwasniewski-Meyer:Aperiodicity}*{Proposition~7.1}.  So~\(C\)
  is a separable \(\Cst\)\nb-algebra and for each \(i\in I\), the
  Hilbert bimodules \((\Hilm_{tit^{-1}})_{t\in \mathbb{F}}\)
  cover~\(C_{\gamma_i}\) up to Morita equivalence.  For each
  \(t\in \mathbb{F}\), either
  \(\Hilm_{tit^{-1}}=\Hilm_t\Hilm[H]_i\Hilm_{t^{-1}}\) or
  \(\Hilm_{tit^{-1}}=\Hilm_u\Hilm[H]_i\Hilm_{u^{-1}}\) with
  \(u=t i^{\pm1}\).  Hence~\(\Hilm_{tit^{-1}}\) is topologically
  non-trivial if~\(\Hilm[H]_i\) is.  Thus by
  \cite{Kwasniewski-Meyer:Aperiodicity}*{Proposition~6.8.(2)}
  \(\gamma_i\) is topologically non-trivial, that is, the sets
  \(\setgiven*{[\pi]\in
    \dual{C}}{\dual{C_{\gamma_i}}[\pi]=[\pi]}=\setgiven*{[\pi]\in
    \dual{C}}{[\pi\circ \gamma_i]=[\pi]}\) for \(i\in I\) have empty
  interiors.  As in the proof of
  \cite{Olesen-Pedersen:Applications_Connes_2}*{Proposition~4.4},
  one shows that the union
  \(\bigcup_{i\in I} \setgiven*{[\pi]\in \dual{C}}{[\pi\circ
    \gamma_i]=[\pi]}\) has empty interior.  Since
  \(\gamma\colon \mathbb{F}\to \Aut (C)\) is a Morita globalisation
  of \((\Hilm_t)_{t\in \mathbb{F}}\), the partial action
  \((\dual{\Hilm}_t)_{t\in \mathbb{F}}\) may be identified with the
  restriction of~\(\dual{\gamma}\) to an open subset.  Hence the
  union
  \(\bigcup_{i\in I} \setgiven*{[\pi]\in
    \dual{A}}{\dual{\Hilm[H]_i}[\pi]=[\pi]}\) has empty interior.
\end{proof}

\begin{theorem}
  \label{the:aperiodic_top_non-trivial}
  Let~\(A\) be a \(\Cst\)\nb-algebra, \(S\) a unital inverse
  semigroup, and~\(\Hilm\) an \(S\)\nb-action on~\(A\) by Hilbert
  bimodules.  Consider the following conditions:
  \begin{enumerate}
  \item \label{en:aperiodic_top_non-trivial1}%
    the action~\(\Hilm\) is aperiodic;
  \item \label{en:aperiodic_top_non-trivial15}%
    for each \(t\in S\) and \(0 \neq K\in\Ideals(I_{1,t}^\bot) \)
    there is \([\pi]\in \dual{K}\) with
    \(\dual{\Hilm}_t[\pi] \neq [\pi]\);
  \item \label{en:aperiodic_top_non-trivial3b}%
    the dual groupoid \(\dual{A}\rtimes S\) is topologically free;
  \item \label{en:aperiodic_top_non-trivial3c}%
    the dual groupoid \(\dual{A}\rtimes S\) is AS topologically
    free;
  \item \label{en:aperiodic_top_non-trivial4b}%
    the dual groupoid \(\dual{A}\rtimes S\) is topologically
    principal;
  \item \label{en:aperiodic_top_non-trivial5}%
    the action~\(\Hilm\) is purely outer.
  \end{enumerate}
  If~\(A\) contains an essential ideal which is separable or of
  Type~I, then
  \ref{en:aperiodic_top_non-trivial1}--\ref{en:aperiodic_top_non-trivial3c}
  are equivalent.  If, in addition, \(S\) is countable, then
  \ref{en:aperiodic_top_non-trivial1}--\ref{en:aperiodic_top_non-trivial4b}
  are equivalent.  In general, each of the conditions
  \ref{en:aperiodic_top_non-trivial1}--\ref{en:aperiodic_top_non-trivial4b}
  implies~\ref{en:aperiodic_top_non-trivial5}.  Conversely,
  \ref{en:aperiodic_top_non-trivial5}
  implies~\ref{en:aperiodic_top_non-trivial1} if~\(A\) contains an
  essential ideal which is of Type~I or simple.
\end{theorem}

\begin{proof}
  By definition, the action~\(\Hilm\) has one of the properties in
  Definition~\ref{def:non-triviality_groupoids} if and only if the
  dual groupoid~\(\dual{A}\rtimes S\) has that property.  The
  groupoid \(\dual{A} \rtimes S\) is the union of the
  bisections~\(\dual{\Hilm_t}\).  And
  \(\dual{\Hilm_t} \cap \dual{A} = \dual{I_{1,t}}\).  The
  restriction of~\(\dual{\Hilm_t}\) to~\(\dual{I_{1,t}^\bot}\) is
  the complement of the closure of \(\dual{\Hilm_t} \cap \dual{A}\)
  in~\(\dual{\Hilm_t}\).  So
  \begin{equation}
    \label{eq:Fix_dual_action}
    \setgiven*{[\pi]\in
      \dual{I_{1,t}^\bot}}{\dual{\Hilm_t}[\pi]=[\pi]} =
    \underline{\Fix}(\dual{\Hilm_t})
  \end{equation}
  in the notation of Lemma~\ref{lem:groupoid_vs_action}.
  Condition~\ref{en:aperiodic_top_non-trivial15} for \(t\in S\)
  makes precise what it means for the partial
  homeomorphism
  of~\(\dual{A}\) induced by \(\Hilm_t\cdot I_{1,t}^\bot\) to be
  topologically non-trivial.  So
  \ref{en:aperiodic_top_non-trivial15}
  and~\ref{en:aperiodic_top_non-trivial3b} are equivalent by
  Lemma~\ref{lem:groupoid_vs_action}.  If~\(A\) contains an
  essential ideal which is separable or of Type~I,
  then~\ref{en:aperiodic_top_non-trivial1} is equivalent
  to~\ref{en:aperiodic_top_non-trivial15} by
  Theorem~\ref{the:non-trivial_Hilbi}.  By
  Lemma~\ref{lem:topological_freeness_vs_nontriviality},
  condition~\ref{en:aperiodic_top_non-trivial15} holds if and only
  if countable unions of the subsets
  \(\underline{\Fix}(\dual{\Hilm_t})\) have empty interior.
  Accordingly, \ref{en:aperiodic_top_non-trivial15}
  and~\ref{en:aperiodic_top_non-trivial3c} are equivalent by
  Lemma~\ref{lem:groupoid_vs_action}.\ref{en:groupoid_vs_action2}.
  Also by Lemma~\ref{lem:groupoid_vs_action},
  conditions~\ref{en:aperiodic_top_non-trivial15} and
  \ref{en:aperiodic_top_non-trivial4b} are equivalent if~\(S\) is
  countable.  The implications concerning
  condition~\ref{en:aperiodic_top_non-trivial5} follow from
  Theorem~\ref{the:non-trivial_Hilbi}.
\end{proof}

Following Archbold and
Spielberg~\cite{Archbold-Spielberg:Topologically_free}, we prove
that~\(A\) detects ideals in \(A\rtimes_\ess S\) for AS
topologically free actions, without going through aperiodicity and
thus without separability assumptions on~\(A\):

\begin{theorem}
  \label{the:top_free_uniqueness}
  Let~\(A\) be a \(\Cst\)\nb-algebra with an AS topologically free
  action~\(\Hilm\) of a unital inverse semigroup~\(S\).  Then
  \(A\rtimes_\ess S\) is the unique quotient of \(A\rtimes S\) in
  which~\(A\) embeds and detects ideals.
\end{theorem}

\begin{proof}
  By Proposition~\ref{pro:abstract_uniqueness} it suffices to show
  that, for any representation \(\pi\colon A\rtimes S\to \Bound(H)\)
  which is injective on~\(A\), there is a map
  \(E_\pi\colon \pi(A\rtimes S)\to \Locmult(A)\) such that
  \(E_\pi\circ \pi=EL\circ \pi\).  This is equivalent to
  \(\norm{EL(a)}\le \norm{\pi(a)}\) for all \(a\in A\rtimes S\).  It
  suffices to check this on the dense \Star{}subalgebra
  \(A\rtimes_\alg S\).  So take \(a \in A\rtimes_\alg S\) and write
  it as \(a=\sum_{t\in F} a_t\delta_t\) for a finite subset
  \(F\subseteq S\) and \(a_t\in \Hilm_t\) for \(t\in F\).  Define
  \(J_t \defeq I_{1,t} \oplus I_{1,t}^\bot\) and
  \(J \defeq \bigcap_{t\in S} J_t\) as in the proof of
  Proposition~\ref{pro:EL_exists}.  These are essential ideals
  in~\(A\), and \(\norm{EL(a)}\) is equal to the supremum of
  \(\norm{EL(a x)}\) for \(x\in J\) with \(\norm{x}\le1\).  So
  \(\norm{EL(a)}\le \norm{\pi(a)}\) follows if
  \(\norm{EL(a x)}\le \norm{\pi(a x)}\) holds for all \(x\in J\).
  So we may assume without loss of generality that
  \(a_t \in \Hilm_t\cdot J\) for all \(t\in F\).

  Then we may further decompose \(a_t = EL(a_t\delta_t) + a_t'\)
  with \(EL(a_t\delta_t) \in I_{1,t}\) and
  \(a_t'\in \Hilm_t^\bot\defeq \Hilm_t\cdot I_{1,t}^\bot\).  Thus
  \(a=EL(a)+\sum_{t\in F} a_t'\delta_t\) with \(EL(a) \in A\).  To
  simplify, we may assume without loss of generality that
  \(\norm{EL(a)}=1\).  Let \(0<\varepsilon < 1\).  The set
  \(U\defeq \setgiven{[\sigma] \in \dual{A}}{\norm{\sigma(EL(a))} >
    1 - \varepsilon}\) is open and non-empty because
  \(\norm{EL(a)}=1\).  For each \(t\in S\),
  Equation~\eqref{eq:Fix_dual_action} and
  Lemma~\ref{lem:groupoid_vs_action}.\ref{en:groupoid_vs_action2}
  show that there is \([\sigma] \in U\) with
  \(\dual{\Hilm_t^\bot} [\sigma] \neq [\sigma]\) for all \(t\in F\).
  There is a representation
  \(\nu\colon \pi(A\rtimes S)\to \Bound(H_\nu)\) that extends the
  irreducible representation~\(\sigma\), that is, there is a closed
  subspace \(H_\sigma \subseteq H_\nu\) on which \(\nu\circ \pi\) is
  unitarily equivalent to~\(\sigma\).  For each \(t\in F\),
  let~\(P_t\) be the orthogonal projection from~\(H_\nu\) onto the
  closed subspace \(\nu(\pi(\Hilm_t^\bot))H_\sigma\).
  Let~\(P_\sigma\) be the orthogonal projection onto~\(H_\sigma\).
  The subspaces~\(P_t H_\nu\) are invariant for \(\nu(\pi(A))\), and
  \(\nu\circ\pi\colon A \to \Bound(P_tH_\nu)\) is either zero or an
  irreducible representation whose equivalence class is
  \(\dual{\Hilm_t^\bot} [\sigma]\) (see
  \cite{Kwasniewski:Topological_freeness}*{Lemma~1.3} or
  \cite{Abadie-Abadie:Ideals}*{Proposition~3.1}).  Therefore,
  \(\dual{\Hilm_t^\bot} [\sigma] \neq [\sigma]\) implies
  \(P_\sigma P_t=0\).  Thus
  \(P_\sigma \nu(\pi(a_t')) P_\sigma= P_\sigma P_t\nu(\pi(a_t'))
  P_\sigma=0\) for all \(t\in F\).  Hence
  \[
  1 - \varepsilon
  < \norm{\sigma(E(a))}
  = \norm{P_\sigma \nu(\pi(E(a)))P_\sigma}
  = \norm{P_\sigma \nu(\pi(a))P_\sigma}
  \le \norm{\pi(a)}.
  \]
  Since \(\varepsilon\in (0,1)\) is arbitrary, this implies
  \(\norm{EL(a)} = 1 \le \norm{\pi(a)}\).  Then~\(EL\) factors
  through~\(\pi\), as desired.
\end{proof}

\begin{corollary}
  \label{cor:simple_minimal}
  Suppose that~\(\Hilm\) is aperiodic or that \(\dual{A}\rtimes S\)
  is AS topologically free.  Then \(A\rtimes_\ess S\) is simple if
  and only if~\(\Hilm\) is minimal, if and only
  if~\(\dual{A}\rtimes S\) is minimal.
\end{corollary}

\begin{proof}
  Theorem~\ref{the:simple_crossed_minimal} applies in both cases.
\end{proof}

\begin{remark}
  More recently, it is shown
  in~\cite{Kwasniewski-Meyer:Aperiodicity_pseudo_expectations} that
  topologically free actions are aperiodic.  This significantly
  improves Theorem~\ref{the:top_free_uniqueness} and
  Corollary~\ref{cor:simple_minimal}.  First, aperiodicity is weaker
  than (AS) topological freeness and, secondly, it also gives
  results about pure infiniteness (see
  Corollary~\ref{cor:simple_crossed_pi}).
\end{remark}

The aperiodicity assumption in
Theorem~\ref{the:aperiodic_action_consequences} is not necessary
for~\(A\) to detect ideals in~\(A\rtimes_\ess S\).  An easy
counterexample is an irrational rotation algebra, viewed as the
(essential) crossed product for a twisted action of the
group~\(\Z^2\) on \(A=\C\).  An important case where aperiodicity is
necessary for detection of ideals is when~\(S\) is very special,
namely, \(\Z\) or~\(\Z_n\) with square-free \(n\in\N_{\ge1}\):

\begin{theorem}
  \label{the:Cartan_for_cyclic_groups}
  Let the compact group~\(\Gamma\) be either~\(\T\) or~\(\Z_n\) with
  square-free \(n\in\N_{\ge1}\).  Let~\(B\) be a \(\Cst\)\nb-algebra
  with a continuous action \(\beta\colon \Gamma\to \Aut(B)\) and let
  \(A=B^\beta\) be the fixed point algebra.  Assume that~\(A\)
  contains an essential ideal which is separable, simple, or of
  Type~I.  Then the following are equivalent:
  \begin{enumerate}
  \item \label{enu:Cartans1}%
    \(A\subseteq B\) is aperiodic;
  \item \label{enu:Cartans2}%
    \(A\) detects ideals in~\(B\);
  \item \label{enu:Cartans3}%
    \(A^+\) supports~\(B\).
  \end{enumerate}
\end{theorem}

\begin{proof}
  Let \(G=\dual{\Gamma}\) be the dual group, that is, \(G=\Z\) or
  \(G=\Gamma=\Z_n\).  Let \(\B = (B_g)_{g\in G}\) be the
  Fell bundle formed by the spectral subspaces of the
  action~\(\beta\).  Then \(B_0=A^\beta=A\) and
  \(B=\Cst_\red(\B)=\Cst(\B)\) by the Gauge-Equivariant Uniqueness
  Theorem.  By Proposition~\ref{pro:aperiodic_isg}, the inclusion
  \(A\subseteq B\) is aperiodic if and only if the Fell bundle is
  aperiodic.  Then
  \cite{Kwasniewski-Meyer:Aperiodicity}*{Theorem~9.12} implies that
  \ref{enu:Cartans1}--\ref{enu:Cartans3} are equivalent.
\end{proof}

\begin{remark}
  For an action of a discrete group on a separable
  \(\Cst\)\nb-algebra for which \(A\subseteq A\rtimes_\red G\)
  detects ideals, Kennedy and
  Schafhauser~\cite{Kennedy-Schafhauser:noncomm_boundaries}
  introduce a cohomological obstruction whose vanishing implies
  aperiodicity (the condition they call proper outerness is
  aperiodicity).
\end{remark}

\section{Fell bundles over groupoids}
\label{sec:groupoids}

Fell bundles over groups are studied, for instance,
in~\cite{Exel:Partial_dynamical}.  We are going to describe their
analogues over étale groupoids through inverse semigroup actions.
Then we carry over our definitions and results for inverse semigroup
actions to the realm of Fell bundles over étale groupoids.
\emph{Throughout this section, \(X\) is a locally compact Hausdorff
  space, \(H\) is an étale groupoid with unit space~\(X\), and
  \(S\subseteq \Bis(H)\) is a unital, wide inverse subsemigroup of
  bisections of~\(H\).}

\subsection{Groupoid Fell bundles and inverse semigroups actions}
\label{sec:groupoid_Fell}

\begin{definition}%
  [\cite{BussExel:Fell.Bundle.and.Twisted.Groupoids}*{Section~2}]
  \label{def:Fell_bundle_groupoid}
  A \emph{Fell bundle} over~\(H\) is an upper semicontinuous bundle
  \(\A=(A_\gamma)_{\gamma\in H}\) of Banach spaces with a continuous
  involution \(^*\colon \A\to \A\) and a continuous multiplication
  \[
    {\cdot}\colon \setgiven{(a,b)\in \A\times \A} {a\in
      A_{\gamma_1},\ b \in A_{\gamma_2},\ \gamma_1,\gamma_2\in
      H,\ \s(\gamma_1) = \rg(\gamma_2)} \to \A,
  \]
  which is associative whenever defined.  In
  addition, the fibres~\(A_x\) for \(x\in X\) must be
  \(\Cst\)\nb-algebras and the fibres~\(A_\gamma\) for
  \(\gamma\in H\) must be Hilbert
  \(A_{\rg(\gamma)}\)-\(A_{\s(\gamma)}\)-bimodules for the left and
  right inner products \(\BRAKET{x}{y} \defeq x y^*\) and
  \(\braket{x}{y} \defeq x^* y\) for \(x,y\in A_\gamma\).  Then the
  multiplication map yields isometric Hilbert bimodule maps
  \[
    \mu_{\gamma_1,\gamma_2}\colon
    A_{\gamma_1}\otimes_{A_{\s(\gamma_1)}} A_{\gamma_2}\to
    A_{\gamma_1\gamma_2}
  \]
  for all \(\gamma_1,\gamma_2\in H\) with
  \(\s(\gamma_1) = \rg(\gamma_2)\).  If these maps are surjective,
  the Fell bundle~\(\A\) is called \emph{saturated}.  This holds if
  and only if~\(\mu_{\gamma,\gamma^{-1}}\) is surjective for all
  \(\gamma\in H\).
\end{definition}

\begin{example}
  \label{exa:Fell_bundle_from_action}
  An action~\(\alpha\) of~\(H\) on a \(\Cst\)\nb-algebra as in
  \cite{Sieben-Quigg:ActionsOfGroupoidsAndISGs}*{Section~3} yields a
  saturated Fell bundle over~\(H\).  Namely, let
  \(A_\gamma\defeq A_{\rg(\gamma)}\) for \(\gamma\in H\) and define
  the multiplication maps and involutions by
  \(A_\eta\times A_\gamma \to A_{\eta\gamma}\),
  \((a,b) \mapsto a\cdot \alpha_\eta(b)\), and
  \(A_\gamma \to A_{\gamma^{-1}}\),
  \(a\mapsto \alpha_\gamma^{-1}(a^*)\).
\end{example}

We are going to turn a Fell bundle \(\A=(A_\gamma)_{\gamma\in H}\)
over~\(H\) into an inverse semigroup action.  Let~\(A\) be the
\(\Cont_0(X)\)\nb-algebra corresponding to the bundle of
\(\Cst\)\nb-algebras \((A_x)_{x\in X}\).  Let \(S\subseteq \Bis(H)\)
be any unital, wide inverse subsemigroup of bisections of~\(H\) and
let \(U\in S\).  This subset is always Hausdorff and locally compact
because the source and range maps restrict to homeomorphisms
from~\(U\) onto the open subsets \(\s(U)\) and~\(\rg(U)\) in~\(X\).
Let~\(A_U\) be the space of \(\Cont_0\)\nb-sections of the
restriction of~\((A_\gamma)_{\gamma\in H}\) to~\(U\).  The spaces
\(A_{\rg(U)}\) and~\(A_{\s(U)}\) are closed ideals in \(A=A_X\).
And the formulas
\begin{align*}
  (a\cdot \xi \cdot b)(\gamma)
  &\defeq a(\rg(\gamma))\xi(\gamma) b(\s(\gamma)),\\
  \braket{\xi}{\eta}(\s(\gamma))
  &\defeq \xi(\gamma)^*\eta(\gamma),\\
  \BRAKET{\xi}{\eta}(\rg(\gamma))
  &\defeq \xi(\gamma)\eta(\gamma)^*
\end{align*}
for \(a\in A_{\rg(U)}\), \(\xi,\eta\in A_U\), \(b\in A_{\s(U)}\),
and \(\gamma\in U\) define on~\(A_U\) the structure of a Hilbert
\(A_{\rg(U)}\)-\(A_{\s(U)}\)-bimodule.  For \(U, V\in S\), there is
a unique isometric Hilbert bimodule map
\(\mu_{U,V}\colon A_U\otimes_A A_V\to A_{UV}\) with
\[
\mu_{U,V}(\xi\otimes \eta) (\gamma_1\cdot \gamma_2)=
\xi(\gamma_1)
\eta(\gamma_2)
\]
for all \(\gamma_1\in U\), \(\gamma_2\in V\).  The involutions are
the maps \(A_U \to A_U^*\), \(f^*(\gamma) \defeq f(\gamma^{-1})^*\).
This defines a Fell bundle over~\(S\) (see
\cite{BussExel:Fell.Bundle.and.Twisted.Groupoids}*{Example~2.9}).
If~\(\A\) is saturated, then the maps~\(\mu_{U,V}\) are surjective,
that is, the Fell bundle over~\(S\) is saturated.  Then the Fell
bundle above defines an action of~\(S\) on~\(A\) by Hilbert
bimodules.

If~\(\A\) is not saturated, then we have to modify~\(S\) to make the
Fell bundle saturated.  Instead of using the general construction
in~\cite{BussExel:InverseSemigroupExpansions} mentioned in
Remark~\ref{rem:make_saturated}, we prefer a more concrete
construction that depends on the Fell bundle~\(\A\).
Let~\(\tilde{S}_U\) for \(U\in S\) be the set of all Hilbert
subbimodules of~\(A_U\) and let \(\tilde{S}\) be the disjoint union
of the sets~\(\tilde{S}_U\).  By the Rieffel correspondence, any
element of~\(\tilde{S}_U\) is of the form
\(\Hilm[F] = A_U\cdot I = J\cdot A_U\), where \(I\) and~\(J\) are
the source and range ideals of~\(\Hilm[F]\).  The involutions
\(A_U \to A_U^*\) and multiplication maps
\(\mu_{U,V}\colon A_U\otimes_A A_V\to A_{UV}\) map Hilbert
subbimodules again to Hilbert subbimodules.  So they define maps
\(\tilde{S}_U \to \tilde{S}_{U^*}\) and
\(\tilde{S}_U\times \tilde{S}_V \to \tilde{S}_{UV}\).  These define
an involution and a multiplication on~\(\tilde{S}\).

\begin{lemma}
  \label{lem:make_saturated_groupoid_Fell}
  The multiplication above makes~\(\tilde{S}\) a unital inverse
  semigroup.  And \((\Hilm[F])_{U\in S, \Hilm[F]\in\tilde{S}_U}\)
  is a saturated Fell bundle over~\(\tilde{S}\).  There is also a
  canonical inverse semigroup homomorphism \(\tilde{S}\to S\) with
  fibres~\(\tilde{S}_U\).  The Fell bundles over \(S\)
  and~\(\tilde{S}\) defined above have the same algebraic, full and
  reduced section \(\Cst\)\nb-algebras.
\end{lemma}

\begin{proof}
  The equation \(t t^* t = t\) holds for all \(t\in \tilde{S}\)
  because
  \(\Hilm[F] \otimes_A \Hilm[F]^* \otimes_A \Hilm[F] \cong
  \Hilm[F]\) holds for all Hilbert \(A\)\nb-bimodules~\(\Hilm[F]\).
  To show that~\(\tilde{S}\) is an inverse semigroup, it suffices to
  prove that the idempotent elements in~\(\tilde{S}\) form a
  commutative subsemigroup.  If \(\Hilm[F]\in\tilde{S}_U\) is
  idempotent, then \(U^2=U\) and so~\(A_U\) is an ideal in~\(A\).
  Then~\(\Hilm[F]\) is an ideal in~\(A\) as well.  The product of
  two ideals is their intersection.  Since this operation on ideals
  is commutative and \(E(S)\) is commutative, the idempotents
  in~\(\tilde{S}\) form a commutative semigroup.  Therefore,
  \(\tilde{S}\) is an inverse semigroup.  The element
  \(A\in \tilde{S}_1\) is an identity element in~\(\tilde{S}\).

  We have defined the multiplication in~\(\tilde{S}\) so that the
  involutions and multiplication maps in our original Fell bundle
  over~\(S\) restricted to the elements of~\(\tilde{S}\) give
  \((\Hilm[F])_{U\in S, \Hilm[F]\in\tilde{S}_U}\) the structure of a
  saturated Fell bundle over~\(\tilde{S}\).  We only discuss the
  inclusion maps (which are redundant in the saturated case by the
  results in~\cite{Buss-Meyer:Actions_groupoids}).  Let
  \(\Hilm[F]_1\subseteq A_{U_1}\) and
  \(\Hilm[F]_2\subseteq A_{U_2}\) be elements of~\(\tilde{S}\).
  Then \((U_1,\Hilm[F]_1) \le (U_2,\Hilm[F]_2)\) holds
  in~\(\tilde{S}\) if and only if \(U_1\le U_2\) in~\(S\) and the
  multiplication map
  \(A_{U_2}\otimes_A A_{\s(U_1)} \hookrightarrow A_{U_1}\) maps
  \(\Hilm[F]_2\otimes_A \s(\Hilm[F]_1)\) onto~\(\Hilm[F]_1\); this
  follows from our description of idempotent elements
  in~\(\tilde{S}\) above.  In this situation, there is a
  canonial inclusion map
  \(\Hilm[F]_1 \cong \Hilm[F]_2\otimes_A \s(\Hilm[F]_1)
  \hookrightarrow \Hilm[F]_2\).  It agrees with the restriction of
  the map \(A_{U_1} \hookrightarrow A_{U_2}\).  And
  \((U_1,\Hilm[F]_1) \le (U_2,\Hilm[F]_2)\) holds if and only if the
  map \(A_{U_1} \hookrightarrow A_{U_2}\) maps~\(\Hilm[F]_1\)
  into~\(\Hilm[F]_2\).

  A section of the Fell bundle over~\(\tilde{S}\) is a finite formal
  linear combination of elements of~\(\Hilm[F]\) for
  \((U,\Hilm[F])\in\tilde{S}\).  The relations are defined
  in~\cite{BussExel:Fell.Bundle.and.Twisted.Groupoids} using only
  the inclusion maps \(\Hilm[F]_1\hookrightarrow \Hilm[F]_2\) for
  \(\Hilm[F]_1\subseteq A_{U_1}\), \(\Hilm[F]_2\subseteq A_{U_2}\)
  with \((U_1,\Hilm[F]_1) \le (U_2,\Hilm[F]_2)\).  More relations
  are used in~\cite{Buss-Exel-Meyer:Reduced}, giving a variant of
  the algebraic section algebra; but this is worked out only in the
  saturated case.  The relations
  in~\cite{BussExel:Fell.Bundle.and.Twisted.Groupoids} already
  suffice to show that any element of~\(\Hilm[F]\) for
  \((U,\Hilm[F])\in \tilde{S}\) is identified in the algebraic
  section \(\Cst\)\nb-algebra with the corresponding element of
  \((U,A_U)\).  And these are subjected to the same relations and
  \Star{}algebra structure as in the section \(\Cst\)\nb-algebra of
  the Fell bundle over~\(S\).  So the section
  \Star{}algebras of the Fell bundles as defined
  in~\cite{BussExel:Fell.Bundle.and.Twisted.Groupoids} are
  canonically isomorphic.  The same would be true for the definition
  in~\cite{Buss-Exel-Meyer:Reduced} if that were carried over to the
  non-saturated case.  The full section \(\Cst\)\nb-algebras are
  defined as the maximal \(\Cst\)\nb-completions of the
  section \Star{}algebras.  Hence these are also canonically
  isomorphic.  The reduced section \(\Cst\)\nb-algebra of a
  non-saturated Fell bundle is defined in~\cite{Exel:noncomm.cartan}
  using all representations of the full section \(\Cst\)\nb-algebra
  that are obtained by inducing irreducible representations of~\(A\)
  in a certain way.  When we identify the full section
  \(\Cst\)\nb-algebras as above, we get the same induced
  representations for both of them.  Hence the reduced section
  \(\Cst\)\nb-algebras are also isomorphic.  For saturated Fell
  bundles, the definition of the reduced section \(\Cst\)\nb-algebra
  through a weak conditional expectation is equivalent to Exel's
  definition by Lemma~\ref{lem:atomic_weak_expectation}.
\end{proof}

Let~\(\A\) be a Fell bundle over~\(H\).  Then~\(H\) acts naturally
on the space~\(\dual{A}\) of all irreducible representations of the
unit fibre \(A=A_X\).  First, any irreducible representation
of~\(A\) factors through the evaluation map \(A \to A_x\) for some
\(x\in X\), and this defines a continuous map
\(\psi\colon \dual{A} \to X\).  Such a map is also equivalent to a
\(\Cont_0(X)\)-\(\Cst\)-algebra structure on~\(A\) by the
Dauns--Hofmann Theorem (see~\cite{Nilsen:Bundles}).  The
map~\(\psi\) is the anchor map of the \(H\)\nb-action
on~\(\dual{A}\).  Secondly, if \(\gamma\in H\), then the Hilbert
\(A_{\rg(\gamma)},A_{\s(\gamma)}\)-bimodule \(A_\gamma\) induces a
partial homeomorphism~\(\dual{A_\gamma}\) from
\(\dual{A_{\s(\gamma)}}\) to \(\dual{A_{\rg(\gamma)}}\).  If the
Fell bundle over~\(H\) is not saturated, then the domain and
codomain~\(\dual{A_\gamma}\) may be smaller than
\(\dual{A_{\s(\gamma)}}\) and \(\dual{A_{\rg(\gamma)}}\),
respectively.  These partial homeomorphisms still form a continuous
``partial'' action of~\(H\) on~\(\dual{A}\), and this is enough to
form a transformation groupoid \(\dual{A}\rtimes H\), just as for a
partial group action on a space.  The proof of
Proposition~\ref{prop:wide_semigroup_isomorphism} extends to this
situation and shows that
\(\dual{A}\rtimes H \cong \dual{A}\rtimes S = \dual{A} \rtimes
\tilde{S}\).

\begin{remark}
  If~\(H\) is Hausdorff, then the unit space of~\(H\) is closed
  in~\(H\).  This implies that the unit space~\(\dual{A}\) is
  closed in~\(\dual{A}\rtimes H\).  That is, all inverse semigroup
  actions associated to actions of~\(H\) are closed (see
  \cite{Buss-Exel-Meyer:Reduced}*{Example~6.7}).
\end{remark}

The following theorem describes when a saturated Fell bundle over
\(\Bis(H)\) comes from a saturated Fell bundle over~\(H\).  The
extra ingredient is the map \(\psi\colon \dual{A} \to X\).

\begin{theorem}[\cite{Buss-Meyer:Actions_groupoids}*{Theorem~6.1}]
  Let~\(H\) be an étale groupoid with locally compact, Hausdorff
  object space~\(X\).  A saturated Fell bundle over~\(\Bis(H)\)
  comes from a saturated Fell bundle over~\(H\) if and only if the
  map \(U\mapsto A_U\) from open subsets in~\(X\) to ideals in~\(A\)
  commutes with suprema, if and only if there is a continuous map
  \(\pi\colon \dual{A} \to X\) such that
  \(\dual{A_U} = \pi^{-1}(U)\) for all open subsets
  \(U\subseteq X\).
\end{theorem}

A similar result holds for non-saturated Fell bundles over~\(H\) and
\(\Bis(H)\).  And if we replace \(\Bis(H)\) by
\(S\subseteq \Bis(H)\), then a Fell bundle over~\(S\) comes from a
Fell bundle over~\(H\) if and only if there is a continuous,
\(S\)\nb-equivariant map \(\pi\colon \dual{A} \to X\) such that
\(\dual{A_e} = \pi^{-1}(e)\) for all \(e\in E(S)\),
identified with open subsets of~\(X\).

\subsection{The full groupoid crossed product}
\label{sec:full_crossed_groupoid}

If \(U\subseteq H\) is a bisection, then~\(U\) is Hausdorff and
locally compact.  Let \(\Contc(U,\A)\subseteq A_U\) be the space of
continuous sections of~\(\A|_U\) with compact support.
Extend functions in \(\Contc(U,\A)\) by~\(0\) to~\(H\).  These
extensions need not be continuous any more.  Let
\(\mathfrak{S}(H,\A)\) be the linear span of \(\Contc(U,\A)\)
for all bisections \(U\subseteq H\).  We call sections in
\(\mathfrak{S}(H,\A)\) \emph{quasi-continuous}.  The space 
\(\mathfrak{S}(H,\A)\) carries
a convolution product and an involution given by
\[
  (f*g)(\gamma) \defeq \sum_{\rg(\eta) = \rg(\gamma)}
  f(\eta)\cdot g(\eta^{-1}\cdot \gamma),\qquad
  (f^*)(\gamma) \defeq f(\gamma^{-1})^*
\]
for all \(f,g\in \mathfrak{S}(H,\A)\), \(\gamma\in H\).  The
\emph{full section \(\Cst\)\nb-algebra} \(\Cst(H,\A)\) of the Fell
bundle~\(\A\) over~\(H\) is defined as the maximal \(\Cst\)\nb-completion of
the \Star{}algebra \(\mathfrak{S}(H,\A)\).

\begin{proposition}
  \label{pro:groupoid_vs_isg_full}
  Let~\(H\) be an étale groupoid with locally compact and Hausdorff
  unit space~\(X\) and let~\(\A\) be a Fell bundle over~\(H\).
  Turn~\(\A\) into a saturated Fell bundle over an inverse
  semigroup~\(\tilde{S}\) as in
  Lemma~\textup{\ref{lem:make_saturated_groupoid_Fell}} above.  Then
  \(\Cst(H,\A) \cong A\rtimes \tilde{S}\).
\end{proposition}

\begin{proof}
  If the Fell bundle~\(\A\) is saturated, then the isomorphism
  \(\Cst(H,\A) \cong A\rtimes S\) is
  \cite{Buss-Meyer:Actions_groupoids}*{Corollary~5.6}.  Under
  separability hypotheses,
  \cite{BussExel:Fell.Bundle.and.Twisted.Groupoids}*{Theorem~2.13}
  says that \(\Cst(H,\A)\) is isomorphic to the full section
  \(\Cst\)\nb-algebra of the Fell bundle~\((A_U)_{U\in S}\).  This
  is isomorphic to \(A\rtimes \tilde{S}\) by
  Lemma~\ref{lem:make_saturated_groupoid_Fell}.  We briefly sketch
  the proof to point out that the extra saturatedness or
  separability assumptions in these proofs are not needed.  What
  makes the proof tricky is that the algebraic \Star{}subalgebras
  used to define \(\Cst(H,\A)\) and \(A\rtimes \tilde{S}\) are \emph{not}
  the same.  The proof shows that they have the same
  \Star{}representations.

  If \(U\in\Bis(H)\), then any compact subset of~\(U\) is covered by
  finitely many bisections in~\(S\).  Using a partition of unity,
  one shows that functions in \(\Contc(t,\A)\) for \(t\in S\)
  already span \(\mathfrak{S}(H,\A)\).  This gives a surjective
  linear map
  \(\bigoplus_{t\in \tilde{S}} \Contc(t,\A) \onto \mathfrak{S}(H,\A)\).  The
  proof of Lemma~\ref{lem:make_saturated_groupoid_Fell} shows that
  \(A\rtimes_\alg \tilde{S}\subseteq A\rtimes \tilde{S}\) is spanned
  by the subspaces~\(A_t\) for \(t\in S\), giving a surjective
  linear map
  \(\bigoplus_{t\in S} A_t \onto A\rtimes_\alg \tilde{S}\).  We
  claim that the map
  \(\bigoplus_{t\in S} \Contc(t,\A) \to \bigoplus_{t\in S} A_t\)
  defined by the inclusion maps \(\Contc(t,\A)\to A_t\) for
  \(t\in S\) descends to a well defined map
  \(\mathfrak{S}(H,\A) \to A\rtimes_\alg \tilde{S}\).  This follows
  from \cite{Buss-Meyer:Actions_groupoids}*{Proposition~B.2}, which
  describes the kernel of the map
  \(\bigoplus_{t\in S} \Contc(t,\A) \to \mathfrak{S}(H,\A)\) in
  terms of the inclusion maps
  \(\Contc(t,\A)\hookrightarrow \Contc(u,\A)\) for \(t\le u\).  The
  resulting map \(\mathfrak{S}(H,\A) \to A\rtimes_\alg \tilde{S}\)
  is an injective \Star{}algebra homomorphism, but not surjective.

  If \(U\subseteq X\), then any \Star{}representation of
  \(\Contc(U,\A)\) is already bounded in the \(\Cst\)\nb-norm
  on~\(A_U\) because \(\Contc(U,\A)\) is a union of
  \(\Cst\)\nb-subalgebras of~\(A_U\).  Hence it extends uniquely to
  a \Star{}representation of~\(A_U\).  Then it follows that the
  restriction of a \Star{}representation of \(\mathfrak{S}(H,\A)\)
  to \(\Contc(t,\A)\) for \(t\in S\) is bounded in the norm
  of~\(A_t\) and hence extends uniquely to a bounded linear map
  on~\(A_t\).  These maps form a representation of the Fell
  bundle~\((A_t)_{t\in S}\) over~\(S\).  Thus \(\mathfrak{S}(H,\A)\)
  and \(A\rtimes_\alg \tilde{S}\) have the same representations and
  hence the same maximal \(\Cst\)\nb-completions.
\end{proof}

\begin{remark}
  \label{rem:Fell_line_bundle}
  A Fell line bundle over~\(H\) is a continuous Fell bundle with
  \(A_\gamma\cong\C\) as a vector space for all \(\gamma\in H\).
  Such Fell bundles correspond to ``twists'' of~\(H\) (the proof of
  \cite{Deaconu-Kumjian-Ramazan:Fell_groupoid_morphism}*{Theorem~5.6}
  still works for non-Hausdorff groupoids).  The corresponding Fell
  bundles over inverse semigroups are studied
  in~\cite{BussExel:Fell.Bundle.and.Twisted.Groupoids}, where they
  are called \emph{semi-Abelian}.  The section \(\Cst\)\nb-algebra
  of a Fell line bundle over~\(H\) is a \emph{twisted} groupoid
  \(\Cst\)\nb-algebra of~\(H\).  The usual groupoid
  \(\Cst\)\nb-algebra \(\Cst(H)\) corresponds to the ``trivial''
  Fell line bundle, where all the multiplication maps are the usual
  multiplication map on~\(\C\).
\end{remark}

\subsection{The reduced section \texorpdfstring{$\Cst$}{C*}-algebra}
\label{sec:reduced_crossed_groupoid}

Next we define the reduced section \(\Cst\)\nb-algebra of the Fell
bundle~\(\A\) over~\(H\).  Let \(x\in X\).  Then
\[
  \Contc(H_x,\A) = \bigoplus_{\s(\gamma)=x} A_\gamma
\]
is a pre-Hilbert \(A_x\)\nb-module for the obvious right
multiplication and the standard inner product
\(\braket{f}{g} \defeq \sum_{\s(\gamma)=x} f(\gamma)^* g(\gamma)\).
Let \(\ell^2(H_x,\A)\) denote its Hilbert \(A_x\)\nb-module
completion.  If \(f\in \mathfrak{S}(H,\A)\),
\(g\in \Contc(H_x,\A)\), then define
\(\lambda_x(f)(g) \in\Contc(H_x,\A)\) by
\[
  \lambda_x(f)(g)(\gamma)\defeq
  \sum_{\rg(\eta) = \rg(\gamma)} f(\eta) g(\eta^{-1}\gamma).
\]
The operator~\(\lambda_x(f)\) extends uniquely to an adjointable
operator on \(\ell^2(H_x,\A)\) with adjoint~\(\lambda_x(f^*)\), and
this defines a non-degenerate \Star{}representation of
\(\mathfrak{S}(H,\A)\) on \(\ell^2(H_x,\A)\).  It extends uniquely
to a non-degenerate \Star{}representation
\[
  \lambda_x\colon \Cst(H,\A) \to \Bound\bigl(\ell^2(H_x,\A)\bigr).
\]

\begin{definition}
  \label{def:reduced_groupoid}
  The \emph{reduced norm} on \(\Cst(H,\A)\) or
  \(\mathfrak{S}(H,\A)\) is defined by
  \[
    \norm{f}_\red \defeq \sup_{x\in X} {}\norm{\lambda_x(f)}.
  \]
  The \emph{reduced section \(\Cst\)\nb-algebra}~\(\Cst_\red(H,\A)\)
  of the Fell bundle~\(\A\) over~\(H\) is defined as the completion
  of \(\mathfrak{S}(H,\A)\) in the reduced norm.  Equivalently, it
  is the quotient of~\(\Cst(H,\A)\) by the ideal
  \(\bigcap_{x\in X} \ker (\lambda_x)\), which is the null space of
  the reduced norm on~\(\Cst(H,\A)\).
\end{definition}

\begin{proposition}
  \label{pro:groupoid_vs_isg_red}
  The isomorphism \(\Cst(H,\A) \cong A\rtimes \tilde{S}\) in
  Proposition~\textup{\ref{pro:groupoid_vs_isg_full}} descends to an
  isomorphism \(\Cst_\red(H,\A) \cong A\rtimes_\red \tilde{S}\).
\end{proposition}

\begin{proof}
  For saturated Fell bundles over~\(H\), which correspond to
  saturated Fell bundles over~\(S\) and thus actions of~\(S\) by
  Hilbert bimodules, this is contained in
  \cite{Buss-Exel-Meyer:Reduced}*{Theorem~8.11}.  The same idea
  works in the non-saturated case.  Each representation~\(\pi\)
  of~\(A\) may be induced to a representation \(i(\pi)\) of
  \(A\rtimes \tilde{S}\).  The representation
  \(\bigoplus_{\pi\in \dual{A}} i(\pi)\) of \(A\rtimes \tilde{S}\)
  descends to a faithful representation of
  \(A\rtimes_\red \tilde{S}\).  This is how
  Lemma~\ref{lem:atomic_weak_expectation} is proven.  Any
  irreducible representation of~\(A\) factors through one of the
  fibres~\(A_x\) for \(x\in X\).  The representation of
  \(A\rtimes \tilde{S}\) induced by \(\pi\in\dual{A_x}\) corresponds
  to the representation \(\lambda_x\otimes 1\) of \(\Cst(H,\A)\) on
  \(\ell^2(H_x,\A)\otimes_{A_x} \Hils_\pi\).  And
  \(\norm{\lambda_x(f)}\) is the supremum of
  \(\norm{\lambda_x\otimes 1_{\Hils_\pi}(f)}\) over all
  \(\pi\in\dual{A_x}\).  So the reduced norm that defines
  \(\Cst_\red(H,\A)\) corresponds to the supremum of
  \(\norm{i(\pi)(f)}\) over all \(\pi\in \dual{A}\), which gives
  \(A\rtimes_\red \tilde{S}\).
\end{proof}

Let \(\mathfrak{B}(H,\A)\) denote the Banach space of bounded Borel
sections of the Banach space bundle~\(\A\), and similarly for
\(\mathfrak{B}(X,\A)\).  So
\(\mathfrak{S}(H,\A) \subseteq \mathfrak{B}(H,\A) \subseteq
\prod_{\gamma\in H} A_\gamma\) as vector spaces.

\begin{proposition}
  \label{prop:embed_into_sections}
  The embedding \(\mathfrak{S}(H,\A) \to \mathfrak{B}(H,\A)\)
  extends uniquely to an injective and contractive linear map
  \(j\colon \Cst_\red(H,\A) \to \mathfrak{B}(H,\A)\).  The map
  \[
    E_\red\colon \Cst_\red(H,\A) \to \mathfrak{B}(X,\A),
    \qquad
    f\mapsto j(f)|_X,
  \]
  is a faithful generalised expectation.
\end{proposition}

\begin{proof} 
  Let \(\gamma\in H\).  Then
  \(A_\gamma \subseteq \ell^2(H_{\s(\gamma)},\A)\) is a direct summand.
  Let \(T_\gamma\colon A_\gamma \to \ell^2(H_{\s(\gamma)},\A)\) be
  the isometric inclusion.  Then~\(T_\gamma^*\) is the orthogonal
  projection onto~\(A_\gamma\).  If \(f\in \Cst_\red(H,\A)\), then
  define
  \[
    j(f)(\gamma) \defeq T_\gamma^* \lambda_x(f) T_{\s(\gamma)}\colon
    A_{\s(\gamma)} \to A_\gamma.
  \]
  If \(f\in \Contc(U,\A)\) for some bisection \(U\subseteq H\), then
  \(j(f)(\gamma)(a) = f(\gamma)\cdot a\) for all
  \(a\in A_{\s(\gamma)}\) and \(\gamma \in H\).  Thus
  \(j(f)(\gamma)\) is the compact operator corresponding under the
  isomorphism \(\Comp(A_{s(\gamma)}, A_\gamma)\cong A_\gamma\) to
  the element \(f(\gamma)\in A_\gamma\).  We simply write
  \(j(f)(\gamma) = f(\gamma)\) for all \(f\in\mathfrak{S}(H,\A)\).
  Since \(\norm{T_\gamma} = 1\), we may estimate
  \(\norm{j(f)(\gamma)} \le \norm{\lambda_x(f)} \le \norm{f}_\red\)
  for all \(f\in\Cst_\red(H,\A)\).  The section
  \(j(f) \in \prod_{x\in X} A_x\) is Borel for all
  \(f\in\Contc(U,\A)\) and hence for \(f\in\mathfrak{S}(H,\A)\).
  Since~\(j\) is bounded, \(\mathfrak{S}(H,\A)\) is dense in
  \(\Cst_\red(H,\A)\), and uniform limits of Borel functions are
  Borel, it follows that~\(j\) is the unique contractive linear map
  \(\Cst_\red(H,\A)\to \mathfrak{B}(H,\A)\) extending~\(j\) on
  \(\mathfrak{S}(H,\A)\).

  If \(x\in X\), then \(f\mapsto j(f)(x) = T_x^* \lambda_x(f) T_x\)
  is a completely positive, contractive linear map
  \(E_x\colon \Cst_\red(H,\A) \to A_x\).  Hence
  \(E_\red(f) \defeq j(f)|_X\) is completely positive and
  contractive as a map to \(\prod_{x\in X} A_x\).  Then it is a
  completely positive contraction
  \(E_\red\colon \Cst_\red(H,\A)\to \mathfrak{B}(X,\A)\) as well.
  Being the identity map on~\(A\), it is a generalised conditional
  expectation.  We may further compose~\(E_x\) with the faithul
  representation
  \(A_x \to \prod_{\pi\in\dual{A_x}} \Bound(\Hils_\pi)\).  Since
  \(\dual{A} = \bigsqcup_{x\in X} \dual{A_x}\), this gives a
  generalised expectation
  \[
    \tilde{E}_\red\colon \Cst_\red(H,\A)
    \to \prod_{\pi\in\dual{A}} \Bound(\Hils_\pi),
    \qquad
    \tilde{E}_\red(f)(\pi) \defeq \pi(j(f)|_X).
  \]
  Let \(E_\red\colon A\rtimes_\red \tilde{S} \to A''\) be the
  canonical weak conditional expectation (it should not be confused
  with \(E_\red\colon \Cst_\red(H,\A)\to \mathfrak{B}(X,\A)\) as the
  domain and codomain are different).  The generalised expectation
  \[
    \varrho\circ E_\red \colon A\rtimes_\red \tilde{S} \to
    \prod_{\pi\in\dual{A}} \Bound(\Hils_\pi),
  \]
  is faithful by Lemma~\ref{lem:atomic_weak_expectation} and
  Theorem~\ref{the:crossed_expectation_faithful}.  The isomorphism
  \(\Cst_\red(H,\A) \cong A\rtimes_\red \tilde{S}\) in
  Proposition~\ref{pro:groupoid_vs_isg_red} intertwines the
  generalised expectations \(\tilde{E}_\red\) and
  \(\varrho\circ E_\red\) because it does so on functions in
  \(f\in \Contc(t,\A)\) for \(t\in S\).  Hence the generalised
  expectation~\(\tilde{E}_\red\) is faithful.  Then so is
  \(E_\red\colon \Cst_\red(H,\A)\to \mathfrak{B}(X,\A)\).

  If \(f\in \mathfrak{S}(H,\A)\), then we compute
  \[
    E_\red(f^* * f)(x)
    = j(f^* * f)(x)
    = \sum_{\s(\gamma)=x} f^*(\gamma^{-1}) f(\gamma)
    = \sum_{\s(\gamma)=x} j(f)(\gamma)^* j(f)(\gamma).
  \]
  The norm of the left hand side is bounded
  by~\(\norm{f}^2_{\Cst_\red(H,\A)}\).  The norm of the right hand
  side is the square of the norm of~\(j(f)\) in \(\ell^2(H_x,\A)\).
  Hence \(j(f)|_{H_x} \in \ell^2(H_x,\A)\) and
  \(\norm{j(f)|_{H_x}}_{\ell^2(H_x,\A)} \le
  \norm{f}_{\Cst_\red(H,\A)}\) for all \(f\in\Cst_\red(H,\A)\).  By
  continuity, we get
  \begin{equation}
    \label{eq:j_vs_E}
    E_\red(f^* * f)(x)
    = \sum_{\s(\gamma)=x} j(f)(\gamma)^* j(f)(\gamma)
  \end{equation}
  for all \(f\in \Cst_\red(H,\A)\).  So \(E_\red(f^* * f) = 0\) is
  equivalent to \(j(f) = 0\).  Since~\(E_\red\) is faithful, this is
  equivalent to \(f=0\).  Hence~\(j\) is injective.
\end{proof}

\begin{remark}
  \label{rem:E-tilde_scalar-valued}
  If~\(\A\) is a Fell line bundle, then its restriction to the unit
  space is the trivial bundle \(X\times\C\).  Hence identifying
  sections~\(j(f)|_X\) with scalar-valued functions on~\(X\),
  we may view~\( E_\red\) as a generalised expectation into the
  \(\Cst\)\nb-algebra~\(\mathfrak{B}(X)\) of Borel functions
  on~\(X\).  This expectation has been used already by Khoskham and
  Skandalis~\cite{Khoshkam-Skandalis:Regular}.
\end{remark}

\subsection{The essential groupoid crossed product}
\label{sec:essential_groupoid}

\begin{definition}
  Let \(\Cst_\ess(H,\A)\) be the quotient of~\(\Cst(H,\A)\) that
  corresponds to the quotient \(A\rtimes_\ess \tilde{S}\) of
  \(A\rtimes \tilde{S}\) under the isomorphism in
  Proposition~\ref{pro:groupoid_vs_isg_full}.
\end{definition}

So \(\Cst_\ess(H,\A)\) is the quotient of~\(\Cst(H,\A)\) by the
ideal~\(\Null_{EL}\) for the canonical \(\Locmult\)-expectation
\[
  EL\colon \Cst(H,\A) \congto A\rtimes \tilde{S} \to \Locmult(A).
\]
Since~\(EL\) is symmetric by
Theorem~\ref{the:crossed_expectation_faithful}, \(b\in \Cst(H,\A)\)
belongs to~\(\Null_{EL}\) if and only if \(EL(b^* * b)=0\).  Since
\(A\rtimes_\ess \tilde{S}\) is a quotient of
\(A\rtimes_\red \tilde{S}\),
Proposition~\ref{pro:groupoid_vs_isg_red} allows to identify
\(\Cst_\ess(H,\A)\) with the quotient of \(\Cst_\red(H,\A)\) by the
image of~\(\Null_{EL}\) in \(\Cst_\red(H,\A)\).  We denote this
image by~\(J_\sing\) as in Section~\ref{sec:compare} and call its
elements \emph{singular}.  We are going to describe~\(J_\sing\) in
terms of the groupoid Fell bundle.

By Proposition~\ref{pro:Locmult_as_operator_families}, the essential
multiplier algebra \(\Locmult(A)\) is embeds into the quotient of
\(\prod_{\pi\in\dual{A}} \Bound(\Hils_\pi)\) by the null space of
the essential supremum norm~\(\norm{\cdot}_\ess\), which takes the
minimum of the supremum norms over comeagre subsets of~\(\dual{A}\).
In the proof of Proposition~\ref{prop:embed_into_sections}, we have
noticed that the canonical generalised expectation
\(E_\red\colon \Cst_\red(H,\A) \to \prod_{x\in X} A_x\) -- composed
with the standard faithful representations -- gives a faithful
generalised expectation
\[
  \tilde{E}_\red\colon \Cst_\red(H,\A)
  \to \prod_{x\in X} A_x \to \prod_{\pi\in\dual{A}} \Bound(\Hils_\pi),
\]
which corresponds to
\(\varrho\circ E_\red \colon A\rtimes_\red \tilde{S} \to
\prod_{\pi\in\dual{A}} \Bound(\Hils_\pi)\).  So by
Proposition~\ref{pro:Locmult_as_operator_families}, the norm of the
image of \(b\in \Cst_\red(H,\A)\) in \(\Cst_\ess(H,\A)\) is the
\emph{essential} supremum of \(\norm*{\tilde{E}_\red(b)(\pi)}\).
Therefore, \(b\in J_\sing\) if and only if the set of
\(\pi\in\dual{A}\) with \(\tilde{E}_\red(b^* * b)(\pi)\neq 0\) is
meagre.  In general, this is the best we can say.  Under extra
assumptions, we are going to rewrite this in terms of the set of
\(x\in X\) with \(E_\red(b^* * b)(x)\neq 0\), or the set of
\(\gamma\in H\) with \(j(b)(\gamma)\neq 0\).  The starting point is
the following analogue of Proposition~\ref{pro:essential_norm}:

\begin{lemma}
  \label{lem:groupoid_Cstar_element_a-e-continuous}
  Let \(f\in \Cst_\red(H,\A)\).  There is a comeagre subset
  \(C\subseteq H\) such that the section~\(j(f)\) of~\(\A\) is
  continuous in all points of~\(C\); that is, if \(\gamma\in C\),
  then there is an open neighbourhood~\(U\) of~\(\gamma\) and a
  continuous section~\(h\) of~\(\A|_U\) with
  \(\lim_{\eta\to\gamma} {}\norm{f(\eta) - h(\eta)} = 0\).  Thus the
  section~\(E_\red(f)\) of~\(\A|_X\) is continuous in \(C\cap X\),
  which is comeagre in~\(X\).
\end{lemma}

\begin{proof}
  First let \(f\in \Contc(U,A)\) for some bisection~\(U\).
  Then~\(j(f)\) is a continuous section on~\(U\) and vanishes on the
  interior of \(H\setminus U\).  Thus it is continuous in all points
  of the dense open subset \(H\setminus \partial U\).  If
  \(f\in \mathfrak{S}(H,\A)\), then~\(f\) is a finite linear
  combination of functions as above.  Hence~\(j(f)\) is continuous
  in all points of a finite intersection of dense open subsets,
  which is again dense open.  Finally, if \(f\in\Cst_\red(H,\A)\),
  then there is a sequence~\((f_n)_{n\in\N}\) in
  \(\mathfrak{S}(H,\A)\) with \(\lim {}\norm{f_n-f} = 0\).  Hence
  \(j(f_n)\) converges uniformly towards~\(j(f)\).  For each
  \(n\in\N\), there is a dense open subset \(Y_n\subseteq H\)
  where~\(f_n\) is continuous.  Let
  \(C\defeq \bigcap_{n\in\N} Y_n\).  The subset~\(C\) is comeagre as
  a countable intersection of dense open subsets.  Its intersection
  with the closed subspace \(X\subseteq H\) is comeagre in~\(X\).

  We claim that~\(j(f)\) is a continuous section in all
  \(\gamma\in C\).  Thus \(E_\red(f) = j(f)|_X\) is continuous in
  \(C\cap X\).  A continuous section~\(h\) with
  \(\lim_{\eta\to\gamma} {}\norm{f(\eta) - h(\eta)} = 0\) is built
  as follows.  Let \(f_n\) and~\(Y_n\) for \(n\in\N\) be as above.
  Define \(f_{-1}\defeq0\).  We may arrange that
  \(\norm{f_n - f_{n-1}}_\infty < 2^{-n}\) for all \(n\in\N\).
  Let~\(U\) be a bisection containing~\(\gamma\).  Since~\(Y_n\) is
  open, there is a function
  \(w_n\in \Contc(U\cap Y_n \cap Y_{n-1})\) with \(w_n(\gamma)=1\)
  and \(\norm{w_n}_\infty \le1\).  Then \((f_n-f_{n-1})\cdot w_n\) is
  a continuous section of~\(\A|_U\) with
  \(\norm{(f_n-f_{n-1})\cdot w_n}\le 2^{-n}\).  Hence
  \(h\defeq \sum_{n=0}^\infty (f_n-f_{n-1})\cdot w_n\) is a well
  defined continuous section of~\(\A|_U\).  The continuity of the
  functions~\(w_n\) implies that
  \(\lim_{\eta\to\gamma} {}\norm{f(\eta) - h(\eta)} = 0\).
\end{proof}

In order to compare our description of the essential crossed product
to the one in~\cite{Exel-Pitts:Weak_Cartan}, we need more
information about the comeagre subset~\(C\) in
Lemma~\ref{lem:groupoid_Cstar_element_a-e-continuous}.

\begin{definition}
  \label{def:dangerous}
  Call \(x\in X\) \emph{dangerous} if there is a net~\((\gamma_n)\)
  in~\(H\) that converges towards two different points
  \(\gamma\neq\gamma'\in H\) with \(\s(\gamma) = \s(\gamma') = x\).
\end{definition}

\begin{lemma}
  \label{lem:dangerous}
  Let \(f\in \Cst_\red(H,\A)\) and \(\gamma\in H\).  If
  \(\s(\gamma)\) is not dangerous, then~\(j(f)\) is continuous
  at~\(\gamma\).  If \(x\in X\) is dangerous, then the isotropy
  group~\(H(x)\) at~\(x\) is non-trivial.  If~\(H\) is covered by
  countably many bisections, then the subset~\(D\) of dangerous
  points is meagre, and so is \(\s^{-1}(D)\subseteq H\).
\end{lemma}

\begin{proof}
  First let \(f\in \Contc(U,\A)\) for some \(U\in \Bis(H)\) and assume
  that \(j(f)\) is discontinuous at \(\gamma\in H\).  Then
  \(\gamma\in H\setminus U\) and there is a net~\((\gamma_n)\)
  converging to~\(\gamma\) with \(f(\gamma_n) \not\to 0\).
  So~\(\gamma_n\) must lie in the support of~\(f\), which is a
  compact subset of~\(U\).  Passing to a subnet, we may arrange
  that~\(\gamma_n\) converges to some \(\gamma'\in U\).  This must
  be different from~\(\gamma\).  Since~\(X\) is Hausdorff,
  \(\s(\gamma) = \lim \s(\gamma_n) = \s(\gamma')\).
  So~\(\s(\gamma)\) is dangerous.  In other words, \(j(f)\) for
  \(f\in \Contc(U,\A)\) is continuous at all \(\gamma\in H\) with
  \(\s(\gamma)\notin D\).  This remains so for uniform limits of
  finite linear combinations of such \(j(f)\), giving the first
  claim for all \(f\in \Cst_\red(H,\A)\).  The same argument shows
  that \(\rg(\gamma) = \rg(\gamma')\).  So \(\gamma^{-1} \gamma'\)
  is a non-trivial element in the isotropy group of~\(x\) and
  \(H(x)\) is non-trivial if~\(x\) is dangerous.

  Now assume that~\(H\) is covered by countably many bisections
  \(S\subseteq \Bis(H)\).  The arrows \(\gamma,\gamma'\) witnessing that some
  \(x\in X\) is dangerous must belong to some bisections
  \(U,V\in S\).  Since a countable union of meagre subsets is
  meagre, it suffices to fix \(U,V\in S\) and prove the meagreness
  of the set of all \(x\in X\) for which there are \(\gamma\in U\),
  \(\gamma'\in V\) with \(\s(\gamma) = \s(\gamma') = x\) and a net
  \((\gamma_n)\) in~\(H\) converging both to \(\gamma\)
  and~\(\gamma'\).  Since \(U\) and~\(V\) are open, we may restrict
  our net \((\gamma_n)\) to a subnet that belongs to \(U\cap V\).
  But \(\gamma,\gamma' \notin U\cap V\).  Since~\(\s\) restricts to
  homeomorphisms \(U\cong \s(U)\) and \(V\cong \s(V)\), it follows
  that \(x\in \partial(\s(U\cap V))\).  This subset is closed and
  nowhere dense, hence meagre.  The subset
  \(\s^{-1}(\partial(\s(U\cap V)))\) is also closed and nowhere
  dense because~\(\s\) is continuous and open.  Hence
  \(\s^{-1}(D)\subseteq H\) is meagre as well.
\end{proof}

\begin{example}
  If~\(H\) is not covered by countably many bisections, then all
  \(x\in X\) may be dangerous.  To produce such an example, let
  \(X=[0,1]\) and let the free group~\(F\) on the set~\([0,1]\) act
  identically on~\(X\).  For each \(t\in [0,1]\) we identify the
  arrow \(s\to s\) given by the generator \(t\in F\) with the
  identity arrow on~\(s\) if \(s\in [0,t)\).  This extends to a
  congruence relation~\(\sim\) on the transformation groupoid
  \([0,1]\rtimes F\).  By construction, each \(t\in [0,1]\) is
  dangerous in \([0,1]\rtimes F/ {\sim}\).
\end{example}

So far, we have assumed~\(\A\) to be an upper semicontinuous field.
Then Lemma~\ref{lem:groupoid_Cstar_element_a-e-continuous} implies
that, for every \(f\in \Cst_\red(H,\A)\), the function
\[
  \nu_f\colon H\to[0,\infty),\qquad
  \gamma\mapsto \norm*{j(f)(\gamma)}_{A_\gamma},
\]
is upper semicontinuous in a comeagre subset of~\(X\).  This is
useless, however.  The applications of
Proposition~\ref{pro:essential_norm} in Section~\ref{sec:compare}
need \emph{lower} semicontinuity instead.  Therefore, we
assume~\(\A\) to be a continuous field of Banach spaces from now on.
For brevity, we call~\(\A\) a \emph{continuous Fell bundle}.

\begin{remark}
  \label{rem:continuous_field}
  A Fell bundle~\(\A\) over~\(H\) is continuous if and only if its
  restriction~\(\A|_X\) is a continuous field of
  \(\Cst\)\nb-algebras over~\(X\).  One implication is trivial, and
  the continuity of~\(\A|_X\) implies continuity of~\(\A\) because
  \(\norm{a(\gamma)}_{A_\gamma}^2 = \norm{a^* a}_{A_{\s(\gamma)}}\)
  for all \(a\in A_\gamma\).  Recall also that a
  \(\Cont_0(X)\)-\(\Cst\)-algebra structure on the
  \(\Cst\)\nb-algebra~\(A_X\) is equivalent to a continuous map
  \(\psi\colon \dual{A_X} \to X\).  The field~\(\A|_X\) is a
  continuous field of \(\Cst\)\nb-algebras if and only if~\(\psi\)
  is open (see \cite{Nilsen:Bundles}*{Theorem~3.3}). A Fell  bundle~\(\A\) is a
  line bundle if and only if~\(\A|_X\) is the trivial
  \(\Cst\)\nb-algebra bundle over~\(X\) with fibre~\(\C\).  This
  implies that~\(\A\) is continuous.  In fact, Fell line bundles are
  locally trivial.
\end{remark}

We can now describe the ideal
\(J_\sing\defeq \ker \left(\Cst_\red(H,\A) \to \Cst_\ess(H,\A)\right)\):

\begin{proposition}
  \label{pro:upper_semicontinuous_dichotomy}
  Let~\(H\) be an étale groupoid with locally compact and Hausdorff
  unit space~\(X\).  Let~\(\A\) be a continuous Fell bundle
  over~\(H\); so the map \(\psi\colon \dual{A}\to X\) with
  \(\psi^{-1}(x) = \dual{A_x} \subseteq \dual{A}\) for \(x\in X\)
  is open.  Assume~\(H\) to be covered by countably many bisections.
  Let \(f\in\Cst_\red(H,\A)\) and \(\varepsilon\ge0\).  Define
  \begin{align*}
    s_{\dual{A}}^{\varepsilon}(f)
    &\defeq \setgiven*{\pi\in\dual{A}}{\norm{\pi(\tilde{E}_\red(f^* *f))}> \varepsilon},\\
    s_X^\varepsilon(f)
    &\defeq \setgiven*{x\in X}{\norm{E_\red(f^* *f)(x)}>\varepsilon},\\
    s_H^\varepsilon(f)
    &\defeq \setgiven*{\gamma\in H}{\norm{j(f)(\gamma)}>\varepsilon}.
  \end{align*}
  Let \(D\subseteq X\) be the set of dangerous points.  The
  following are equivalent:
  \begin{enumerate}
  \item \label{en:upper_semicontinuous_dichotomy_1}%
    \(f\in J_\sing\);
  \item \label{en:upper_semicontinuous_dichotomy_2a}%
    \(s_{\dual{A}}^0(f)\subseteq \psi^{-1}(D)\);
  \item \label{en:upper_semicontinuous_dichotomy_2}%
    \(s_{\dual{A}}^0(f)\subseteq \dual{A}\) is meagre;
  \item \label{en:upper_semicontinuous_dichotomy_3}%
    \(s_{\dual{A}}^0(f)\subseteq \dual{A}\) has empty interior;
   \item \label{en:upper_semicontinuous_dichotomy_3a}%
     \(s_{\dual{A}}^\varepsilon(f)\subseteq \dual{A}\) has empty interior for all
     \(\varepsilon>0\);
   \item \label{en:upper_semicontinuous_dichotomy_4a}%
     \(s_X^0(f)\subseteq D\);
   \item \label{en:upper_semicontinuous_dichotomy_4}%
     \(s_X^0(f)\subseteq X\) is meagre;
   \item \label{en:upper_semicontinuous_dichotomy_6}%
     \(s_X^0(f)\subseteq X\) has empty interior;
   \item \label{en:upper_semicontinuous_dichotomy_7}%
     \(s_X^\varepsilon(f)\subseteq X\) has empty interior for all
     \(\varepsilon>0\);
   \item \label{en:upper_semicontinuous_dichotomy_8a}%
     \(s_H^0(f)\subseteq \s^{-1}(D)\);
   \item \label{en:upper_semicontinuous_dichotomy_8}%
     \(s_H^0(f)\subseteq H\) is meagre;
   \item \label{en:upper_semicontinuous_dichotomy_10}%
     \(s_H^0(f)\subseteq H\) has empty interior;
   \item \label{en:upper_semicontinuous_dichotomy_11}%
     \(s_H^\varepsilon(f)\subseteq H\) has empty interior for all
     \(\varepsilon>0\).
   \end{enumerate}	
   In general, without any restriction on~\(\A\) and~\(H\),
   \ref{en:upper_semicontinuous_dichotomy_1} is equivalent to each
   of the conditions
   \ref{en:upper_semicontinuous_dichotomy_2}--\ref{en:upper_semicontinuous_dichotomy_3a}.
   And if~\(\A\) is a Fell line bundle, then they are further
   equivalent to
   \ref{en:upper_semicontinuous_dichotomy_4}--\ref{en:upper_semicontinuous_dichotomy_7}.
\end{proposition}

\begin{proof}
  In general, without any restriction on \(\A\) and~\(H\),
  \ref{en:upper_semicontinuous_dichotomy_1} is equivalent to
  \ref{en:upper_semicontinuous_dichotomy_2}--\ref{en:upper_semicontinuous_dichotomy_3a}
  by Corollary \ref{cor:singular_ideal_description}.  If~\(\A\) is a
  line bundle, then \(\psi\colon \dual{A}\to X\) is a homeomorphism
  and
  \(\psi(s_{\dual{A}}^\varepsilon(f))=s_{X}^\varepsilon(f)\)
  for all \(\varepsilon \ge 0\).  Thus
  \ref{en:upper_semicontinuous_dichotomy_2}--\ref{en:upper_semicontinuous_dichotomy_3a}
  are equivalent to
  \ref{en:upper_semicontinuous_dichotomy_4}--\ref{en:upper_semicontinuous_dichotomy_7}
  in this case.

  From now on, we assume~\(\A\) to be a continuous Fell bundle
  and~\(H\) to be covered by countably many bisections.  We first
  show that \ref{en:upper_semicontinuous_dichotomy_8a}%
  --\ref{en:upper_semicontinuous_dichotomy_11} are equivalent.  The
  subset \(\s^{-1}(D)\subseteq H\) is meagre by
  Lemma~\ref{lem:dangerous}.
  So~\ref{en:upper_semicontinuous_dichotomy_8a}
  implies~\ref{en:upper_semicontinuous_dichotomy_8}.  Since~\(H\) is
  a union of open sets that are Baire, \(H\) is a Baire space. Hence
  a meagre subset of \(H\) must have empty interior.
  So~\ref{en:upper_semicontinuous_dichotomy_8}
  implies~\ref{en:upper_semicontinuous_dichotomy_10}.  And
  \ref{en:upper_semicontinuous_dichotomy_10} implies
  \ref{en:upper_semicontinuous_dichotomy_11} because
  \(s_H^0(f)\supseteq s_H^\varepsilon(f)\) for all
  \(\varepsilon>0\).  If \(s_H^0(f)\) is not contained in
  \(\s^{-1}(D)\), then there is \(\gamma\in H\) with
  \(\s(\gamma)\notin D\) and \(j(f)(\gamma)\neq0\).  Let
  \(\varepsilon = \norm{j(f)(\gamma)}/2\).  Since
  \(\s(\gamma)\notin D\) and the bundle~\(\A\) is lower
  semicontinuous, \(\norm{j(f)}\) is lower semicontinuous
  in~\(\gamma\) by
  Lemma~\ref{lem:groupoid_Cstar_element_a-e-continuous}.  This gives
  an open neighbourhood~\(U\) of~\(\gamma\) with
  \(\norm{j(f)(\eta)}>\varepsilon\) for all \(\eta\in U\).
  So~\(s_H^\varepsilon(f)\) has non-empty interior.  This shows
  that~\ref{en:upper_semicontinuous_dichotomy_11}
  implies~\ref{en:upper_semicontinuous_dichotomy_8a} and finishes
  the proof that \ref{en:upper_semicontinuous_dichotomy_8a}%
  --\ref{en:upper_semicontinuous_dichotomy_11} are equivalent.

  The spaces \(X\) and~\(\dual{A}\) are Baire spaces as well (see
  \cite{Dixmier:Cstar-algebras}*{Proposition~3.4.13}).  The subset
  \(D\subseteq X\) is meagre.  Since~\(\psi\) is open and
  continuous, this implies that \(\psi^{-1}(D) \subseteq \dual{A}\)
  is also meagre.  By Lemma~\ref{lem:dangerous}, the function
  \(\norm{E_\red(f^* * f)(x)}\) is continuous in \(X\setminus D\).
  We claim that the function \(\norm{\tilde{E}_\red(f^* *f)(\pi)}\)
  is lower semicontinuous in \(\psi^{-1}(X\setminus D)\). Indeed, if
  \(x\in X\setminus D\), then there is a continuous section~\(h\)
  of~\(\A|_X\) with
  \(\lim_{y\to x} {}\norm{E_\red(f^* *f)(y) - h(y)} = 0\).  Since
  \(h\in A\), the function
  \(\norm{\tilde{E}_\red(h)(\pi)} = \norm{\pi(h)}\) is lower
  semicontinuous on~\(\dual{A}\) (see
  \cite{Dixmier:Cstar-algebras}*{Proposition~3.3.2}).  Since
  \[
    \abs*{\norm{\tilde{E}_\red(h)(\pi)} - \norm{\tilde{E}_\red(f^* *f)(\pi)}}
    \le \norm{E_\red(f^* *f)(y) - h(y)}
  \]
  for \(\pi \in \dual{A_y}\subseteq \dual{A}\), it follows that
  \(\norm{\tilde{E}_\red(f^* *f)(\pi)}\) is lower semicontinuous in
  \(\psi^{-1}(x)\).

  Using the facts gathered in the previous paragraph, we may carry
  over the proof of the equivalence
  \ref{en:upper_semicontinuous_dichotomy_8a}%
  --\ref{en:upper_semicontinuous_dichotomy_11} to prove that
  \ref{en:upper_semicontinuous_dichotomy_4a}%
  --\ref{en:upper_semicontinuous_dichotomy_7} are equivalent and
  that \ref{en:upper_semicontinuous_dichotomy_2a}%
  --\ref{en:upper_semicontinuous_dichotomy_3a} are equivalent.
  Finally, \(E_\red(f^* * f)(x)=0\) is equivalent to
  \(\norm{\pi(\tilde{E}_\red(f^* *f))}=0\) for all
  \(\pi\in \dual{A_x}\), and to \(j(f)|_{H_x}=0\)
  by~\eqref{eq:j_vs_E}.  Hence
  \ref{en:upper_semicontinuous_dichotomy_4a} is equivalent to
  \ref{en:upper_semicontinuous_dichotomy_2a} and
  \ref{en:upper_semicontinuous_dichotomy_8a}.
\end{proof}

\begin{remark}
  By definition, \(\Cst_\red(H,\A) = \Cst_\ess(H,\A)\) if and only
  if the only element \(f\in \Cst_\red(H,\A)\) that belongs
  to~\(J_\sing\) is the zero element.  Thus
  Proposition~\ref{pro:upper_semicontinuous_dichotomy} implies many
  equivalent characterisations for
  \(\Cst_\red(H,\A) = \Cst_\ess(H,\A)\).
\end{remark}

\begin{remark}
  An element~\(f\) of \(\Cst_\red(H)\) is called \emph{singular} in
  \cite{Clark-Exel-Pardo-Sims-Starling:Simplicity_non-Hausdorff} if
  \(s_H^0(f)\defeq\setgiven{\gamma\in H}{j(f)(\gamma)\neq 0}\) has
  empty interior.
  Proposition~\ref{pro:upper_semicontinuous_dichotomy} shows that
  \(f\in \Cst_\red(H)\) is singular in the notation
  of~\cite{Clark-Exel-Pardo-Sims-Starling:Simplicity_non-Hausdorff}
  if and only if it belongs to~\(J_\sing\), provided~\(H\) is
  covered by countably many bisections.  This follows if~\(H\) is
  countable at infinity.
\end{remark}

Let~\(\A\) be a Fell line bundle.  Remark~\ref{rem:continuous_field}
shows that~\(\A\) is a continuous field of Banach spaces over~\(H\).
Exel and Pitts show that
\[
  \Gamma\defeq \setgiven{f \in \Cst_\red(H,\A)}
  {E_\red(f^*f)(x)=0 \text{ for all }x\in X \text{ with }H(x)=\{x\}}
\]
is an ideal in \(\Cst_\red(H,\A)\), and they define the essential
groupoid \(\Cst\)\nb-algebra as the quotient
\(\Cst_\red(H,\A)/\Gamma\).  It is clear that
\(\Gamma \cap \Cont_0(X)=0\) if and only if the set of \(x\in H\)
with \(H(x)=\{x\}\) has empty interior, that is, \(H\) is
topologically principal.  Therefore, the map from \(\Cont_0(X)\) to
\(\Cst_\red(H,\A)/\Gamma\) is injective if and only if~\(H\) is
topologically principal.  In contrast, we have defined
\(\Cst_\ess(H,\A)\) so that the map
\(\Cont_0(X) \to \Cst_\ess(H,\A)\) is always injective.  Hence the
essential twisted groupoid \(\Cst\)\nb-algebras defined here and
in~\cite{Exel-Pitts:Weak_Cartan} differ when~\(H\) is not
topologically principal.  If, however, \(H\) is topologically
principal, then the two definitions are equivalent:

\begin{proposition}
  \label{prop:singular_ideal_topological_principal}
  Let~\(H\) be topologically principal and let~\(\A\) be a Fell
  bundle over~\(H\).  Assume either that~\(\A\) is a Fell line
  bundle or that~\(\A\) is continuous and~\(H\) is covered by
  countably many bisections.  Then
  \[
    J_\sing = \setgiven{f\in \Cst_\red(H,\A)}{
      E_\red(f^* *f)(x)=0\text{ for all } x \in X\text{ with }H(x)=\{x\}}.
  \]
\end{proposition}

\begin{proof}
  Assume \(E_\red(f^*f)(x)\neq 0\) for some \(x\in X\) with
  \(H(x)=\{x\}\).  Then \(E_\red(f^*f)\) is continuous at this
  point by Lemma~\ref{lem:dangerous}.  Hence there is an open
  neighbourhood of~\(x\) on which \(E_\red(f^*f)(y)\neq 0\).  Thus
  \(f\not\in J_\sing\) by
  Proposition~\ref{pro:upper_semicontinuous_dichotomy}.  Conversely,
  assume \(E_\red(f^*f)(x)=0\) for all \(x\in X\) with
  \(H(x)=\{x\}\).  Since~\(H\) is topologically principal, the set
  of \(x\in X\) with \(H(x)\neq\{x\}\) has empty interior.  Hence
  the set of \(x\in X\) with \(E_\red(f^*f)(x)\neq0\) has empty
  interior.  Then \(f\in J_\sing\) by
  Proposition~\ref{pro:upper_semicontinuous_dichotomy}.
\end{proof}

If~\(\A\) is a Fell line bundle, then the conditions above may also
be related to supportive conditional expectations and
the criterion in Lemma~\ref{lem:supportive_condition}:

\begin{theorem}
  \label{thm:supportive_groupoid_characterisation}
  Let~\(H\) be a topologically free étale groupoid with locally
  compact Hausdorff object space~\(X\).  Let~\(\L\) be a Fell line
  bundle over~\(H\).
  The following are equivalent:
  \begin{enumerate}
  \item \label{enu:supportive_groupoid_characterisation1}%
    \(\Cst_\red(H,\L)=\Cst_\ess(H,\L)\);
  \item \label{enu:supportive_groupoid_characterisation1.25}%
    \(\Cont_0(X)\) supports \(\Cst_\red(H,\L)\);
  \item \label{enu:supportive_groupoid_characterisation4}%
    \(\Cont_0(X)\) detects ideals in \(\Cst_\red(H,\L)\);
  \item \label{enu:supportive_groupoid_characterisation6}%
    for any \(f\in \Cst_\red(H,\L)^+\setminus \{0\}\), there is
    \(a\in \Cont_0(X)^+\setminus\{0\}\) with \(a \le j(f)|_X\);
  \item \label{enu:supportive_groupoid_characterisation7}%
    the canonical weak expectation
    \(\Cst_\red(H,\L) \to \Cont_0(X)''\) is supportive.
  \end{enumerate}
\end{theorem}

\begin{proof}
  The dual groupoid of our action is~\(H\), and \(\Cont_0(X)\) is of
  Type~I.  So Theorem~\ref{the:aperiodic_top_non-trivial} shows that
  the inclusion
  \(\Cont_0(X) \hookrightarrow \Cst(H,\L) \cong \Cont_0(X)\rtimes
  S\) is aperiodic.  Thus
  Theorem~\ref{the:aperiodic_action_consequences} shows that the
  essential crossed product \(\Cst_\ess(H,\L)\) is the unique
  quotient of \(\Cst(H,\L)\) in which \(\Cont_0(X)\) embeds and
  detects ideals, and also that \(\Cont_0(X)\) supports
  \(\Cst_\ess(H,\L)\).  So
  \ref{enu:supportive_groupoid_characterisation1}
  and~\ref{enu:supportive_groupoid_characterisation4} are equivalent
  and they imply~\ref{enu:supportive_groupoid_characterisation1.25}.
  And~\ref{enu:supportive_groupoid_characterisation1.25}
  implies~\ref{enu:supportive_groupoid_characterisation4} by
  Lemma~\ref{lem:generalised_support}.  Hence the conditions
  \ref{enu:supportive_groupoid_characterisation1}--%
  \ref{enu:supportive_groupoid_characterisation4} are equivalent.
  Assume~\ref{enu:supportive_groupoid_characterisation1} and let
  \(f\in \Cst_\red(H,\L)^+ \setminus\{0\}\).  Since \(J_\sing=0\),
  Proposition~\ref{pro:upper_semicontinuous_dichotomy} implies that
  there are \(\varepsilon>0\) and an open subset \(U\subseteq X\)
  with \(j(f)|_X(x) = E_\red(f)(x) > \varepsilon\) for all
  \(x\in U\).  There is \(a\in \Cont_0(U)\) with
  \(0 \le a \le \varepsilon\).  It witnesses that~\(f\) satisfies
  the condition in~\ref{enu:supportive_groupoid_characterisation6}.
  So~\ref{enu:supportive_groupoid_characterisation1}
  implies~\ref{enu:supportive_groupoid_characterisation6}.
  And~\ref{enu:supportive_groupoid_characterisation6}
  implies~\ref{enu:supportive_groupoid_characterisation7} by
  Lemma~\ref{lem:supportive_condition}.
  Condition~\ref{enu:supportive_groupoid_characterisation7}
  implies~\ref{enu:supportive_groupoid_characterisation1.25}
  and~\ref{enu:supportive_groupoid_characterisation4} by
  Theorem~\ref{the:aperiodic_consequences}.  This shows that all
  conditions are equivalent.
\end{proof}

\begin{remark}
  The existence of non-trivial singular elements in
  \(\Cst_\red(H,\L)\) for a Fell line bundle~\(\L\) depends on the
  line bundle~\(\L\).  Exel constructed in
  \cite{Exel:Non-Hausdorff}*{Section~2} a non-Hausdorff,
  topologically principal, \'etale groupoid~\(H\) such that
  \(\Cont_0(X)\) does not detect ideals in~\(\Cst_\red(H)\).  A line
  bundle~\(\L\) over the same groupoid such that \(\Cont_0(X)\) does
  detect ideals in~\(\Cst_\red(H,\L)\) is built
  in~\cite{Exel-Pitts:Weak_Cartan}*{Section~23}.
\end{remark}

\subsection{Ideal structure and pure infiniteness for groupoid Fell
  bundles}
\label{sec:ideals_groupoid}

The following theorem carries our results for inverse semigroup
crossed products over to the groupoid case:

\begin{theorem}
  \label{the:aperiodic_vs_top_free_groupoid}
  Let~\(\A\) be a Fell bundle over an étale groupoid~\(H\) with
  locally compact Hausdorff object space~\(X\).  Assume that
  \(A\defeq \Cont_0(\A|_X)\) contains an essential ideal that is
  separable or of Type~I.  The inclusion \(A \subseteq \Cst(H,\A)\)
  is aperiodic if and only if the dual groupoid
  \(\dual{A}\rtimes H\) is topologically free.  This follows if
  \(\Prim(A)\rtimes H\) is topologically free.  If~\(\A\) is a
  continuous Fell bundle, then it follows also if~\(H\) is
  topologically free.
\end{theorem}

\begin{proof}
  The first statement follows from
  Theorem~\ref{the:aperiodic_top_non-trivial} because
  \(\dual{A}\rtimes H\) is naturally isomorphic to the dual groupoid
  \(\dual{A}\rtimes \tilde{S}\) for the inverse semigroup action
  defined in Lemma~\ref{lem:make_saturated_groupoid_Fell}.  If
  \(X\onto Y\) is an \(H\)\nb-equivariant, continuous and open map
  between two \(H\)\nb-spaces, then \(X\rtimes H\) inherits
  topological freeness from \(Y\rtimes H\).  Hence
  \(\dual{A}\rtimes H\) is topologically free if
  \(\Prim(A)\rtimes H\) is topologically free.  And
  \(\Prim(A)\rtimes H\) is topologically free if \(H = X\rtimes H\)
  is topologically free and the base map
  \(\psi\colon \Prim(A)\to X\) is open.  The base map is open if and
  only if~\(A\) is a continuous \(\Cont_0(X)\)-algebra if and only
  if~\(\A\) is a continuous Fell bundle (see
  Remark~\ref{rem:continuous_field}).
\end{proof}

\begin{theorem}
  \label{the:groupoid_crossed_properties}
  Let~\(\A\) be a Fell bundle over an étale groupoid~\(H\) with
  locally compact Hausdorff object space~\(X\).  Assume that the
  inclusion \(A\defeq \Cont_0(\A|_X) \subseteq \Cst(H,\A)\) is
  aperiodic.  Then
  \begin{enumerate}
  \item \label{the:groupoid_crossed_properties_1}%
    \(A\) supports \(\Cst_\ess(H,\A)\);
  \item \label{the:groupoid_crossed_properties_2}%
    \(A\) detects ideals in~\(\Cst_\ess(H,\A)\),
    and~\(\Cst_\ess(H,\A)\) is the only quotient of \(\Cst(H,\A)\)
    with this property;
  \item \label{the:groupoid_crossed_properties_3}%
    if \(J\in \Ideals(\Cst(H,\A))\) and \(J\cap A=0\), then
    \(J\subseteq \ker(\Cst(H,\A) \to \Cst_\ess(H,\A))\);
  \item \label{the:groupoid_crossed_properties_4}%
    \(\Cst_\ess(H,\A)\) is simple if and only if the dual groupoid
    \(\dual{A}\rtimes H\) is minimal;
  \item \label{the:groupoid_crossed_properties_5}%
    if \(\dual{A}\rtimes H\) is minimal, then \(\Cst_\ess(H,\A)\)
    is simple and purely infinite if and only if every element of
    \(A^+\setminus\{0\}\) is infinite in \(\Cst_\ess(H,\A)\).
  \end{enumerate}
\end{theorem}

\begin{proof}
  We may apply Theorems \ref{the:aperiodic_action_consequences}
  and~\ref{the:simple_crossed_minimal} and
  Corollary~\ref{cor:simple_crossed_pi}, replacing \(A\rtimes S\)
  and \(A\rtimes_\ess S\) by \(\Cst(H,\A)\) and \(\Cst_\ess(H,\A)\).
  This implies all the statements.
\end{proof}

The simplicity criterion above strengthens the criterion of
Renault~\cite{Renault:Ideal_structure} by removing the Hausdorffness
assumption, replacing \(\Prim(A)\rtimes H\) by \(\dual{A}\rtimes H\)
in the topological freeness assumption, and weakening the
separability assumptions.  In the group case,
Theorem~\ref{the:groupoid_crossed_properties} implies Kishimoto's
theorem that purely outer group actions on simple
\(\Cst\)\nb-algebras are simple, whereas Renault's criterion says
nothing about this situation.  Theorems
\ref{the:groupoid_crossed_properties}
and~\ref{the:aperiodic_vs_top_free_groupoid} combined with the
criteria for \(\Cst_\red(H,\A) = \Cst_\ess(H,\A)\) also imply the
results
in~\cite{Clark-Exel-Pardo-Sims-Starling:Simplicity_non-Hausdorff}
about the simplicity of~\(\Cst_\red(H)\).

\begin{theorem}
  \label{theorem:pure_infiniteness_groupoid}
  Let~\(H\) be a minimal, topologically free, étale groupoid with
  locally compact Hausdorff object space~\(X\).  Let~\(\L\) be a
  Fell line bundle over~\(H\).  Then \(\Cst_\ess(H,\L)\) is simple.
  And \(\Cst_\ess(H,\L)\) is purely infinite if and only if every
  element of \(\Cont_0(X)\) is infinite in \(\Cst_\ess(H,\L)\).

  In particular, \(\Cst_\ess(H,\L)\) is purely infinite if, for
  every non-empty open subset \(U\subseteq X\) there are \(n\in\N\), a
  non-empty open subset \(V\subseteq U\), and bisections
  \(t_1,\dotsc,t_n\in \Bis(H)\) on which~\(\L\) is trivial such that
  \[
    \rg(t_i)\cap \rg(t_j)=\emptyset\quad\text{for } 1\le i<j\le n,
    \qquad
    V=\bigcup_{i=1}^n \s(t_i),\qquad
    \overline{\bigcup_{i=1}^n \rg(t_i)}\subsetneq V.
  \]
\end{theorem}

\begin{proof}
  Theorem~\ref{the:aperiodic_vs_top_free_groupoid} implies that the
  inclusion \(\Cont_0(X) \subseteq \Cst(H,\L)\) is aperiodic because
  \(\dual{A} = X\).  Then the inclusion
  \(\Cont_0(X) \subseteq \Cst_\ess(H,\L)\) is aperiodic as well.
  Now Theorem~\ref{the:groupoid_crossed_properties} implies the
  first part of the statement.  It remains to check that the
  criterion in the second paragraph implies that any
  \(f\in \Cont_0(X)^+\setminus\{0\}\) is infinite in
  \(B\defeq \Cst_\ess(H,\L)\) (see Definition~\ref{def:infinite}).

  Let~\(S\) be the set of bisections in \(\Bis(H)\) that
  trivialise~\(\L\).  This is a unital, wide inverse subsemigroup,
  and we may identify \(\Cst_\ess(H,\L)\) with the essential crossed
  product \(\Cont_0(X)\rtimes_\ess S\) by the inverse semigroup
  action \(\Hilm=(\Hilm_t)_{t\in S}\) associated to the Fell line
  bundle~\(\L\).  Let \(U\defeq\setgiven{x\in X}{f(x)\neq0}\).
  Choose \(\emptyset\neq V\subseteq U\) and \(t_1,\dotsc,t_n\in S\)
  as in the statement of the theorem.  Let
  \(a=b\in \Cont_0(X)^+\setminus\{0\}\) be any non-zero positive
  function that is supported in the open subset
  \(V\setminus \overline{\bigcup_{i=1}^n \rg(t_i)}\).  Then
  \(a\precsim f\).  Since \(\Cst_{\ess}(H,\L)\) is simple, \(f\) is
  infinite if~\(a\) is infinite (see the proof of
  Proposition~\ref{prop:support_vs_pure_infinite}).  Let
  \(w_1,\dotsc, w_n\in \Cont_0(X)\) be a partition of unity
  subordinate to the open covering \(V= \bigcup_{i=1}^n \s(t_i)\).
  Let \(a_i \defeq a \cdot w_i^{1/2}\) for \(i=1,\dotsc,n\).  These
  functions vanish outside~\(V\), and~\(a_i\) is supported
  in~\(\s(t_i)\).  Since~\(\L|_{t_i}\) is a trivial line bundle,
  \(a_i\) gives an element of~\(\Hilm_{t_i}\), which we denote by
  \(a_i \delta_{t_i}\).  It belongs to \(a\cdot \Hilm_{t_i}\) by
  construction.  The product
  \((a_i \delta_{t_i})^* a_j \delta_{t_j}\) is defined using the
  Fell bundle structure.  It vanishes for \(i\neq j\) because
  \(\rg(t_i)\cap \rg(t_j)=\emptyset\).  Similarly,
  \((a_i \delta_{t_i})^* a = 0\).  And
  \(\sum_{i=1}^n (a_i \delta_{t_i})^* a_i \delta_{t_i} =
  \sum_{i=1}^n a \cdot w_i = a\).  Hence
  \(x\defeq \sum_{i=1}^n a_i\delta_{t_i}\) and \(y\defeq \sqrt{a}\)
  are elements of \(a \cdot \Cst_\ess(H,\L)\) such that
  \(x^* x =a\), \(y^*y=a\neq0\) and \(x^*y=0\).  Thus~\(a\) is
  infinite in \(\Cst_\ess(H,\L)\).  In fact, the proof shows
  that~\(a\) is properly infinite.
\end{proof}

\begin{remark}
  The condition in the second paragraph of
  Theorem~\ref{theorem:pure_infiniteness_groupoid} is satisfied, for
  instance, for the transformation groupoids of strongly boundary
  group actions (see~\cite{Laca-Spielberg:Purely_infinite}) and,
  more generally, for filling actions
  (see~\cite{Jolissaint-Robertson:Simple_purely_infinite}).  Hence
  the pure infiniteness results in
  \cites{Laca-Spielberg:Purely_infinite,
    Jolissaint-Robertson:Simple_purely_infinite} are covered by
  Theorem~\ref{theorem:pure_infiniteness_groupoid}.

  For an \'etale, locally compact groupoid~\(H\), the condition of
  being \emph{locally contracting}
  in~\cite{Anantharaman-Delaroche:Purely_infinite} is the same as
  the condition in Theorem~\ref{theorem:pure_infiniteness_groupoid}
  with \(n=1\), without the trivialisation of the Fell line bundle.
  Thus \(\Cst_\ess(H)\) is purely infinite and simple if~\(H\) is a
  minimal, topologically free, étale, locally contracting groupoid
  with locally compact Hausdorff object space.
\end{remark}

Now we specialise further to the trivial Fell bundle, which gives
the groupoid \(\Cst\)\nb-algebra \(\Cst(H)\) without any twist.
Then \(\Cst(H) \cong \Cont_0(X) \rtimes S\) for any unital, wide
inverse subsemigroup \(S\subseteq \Bis(H)\), acting in the canonical
way on \(\Cont_0(X)\).  We are going to show that
\(\Cont_0(X) \subseteq \Cst(H)\) has the generalised intersection
property with essential quotient~\(\Cst_\ess(H)\) if \emph{and only
  if}~\(H\) is topologically free.  One direction already follows
from our general results.  For the other direction, we construct the
orbit representations of~\(\Cst(H)\).  We construct them as
covariant representations of the \(S\)\nb-action on~\(\Cont_0(X)\).
Let \(x\in X\) and let \([x]\defeq \rg(\s^{-1}(x))\subseteq X\) be the
orbit of~\(x\).  The \emph{orbit representation of~\([x]\)} takes
place on the Hilbert space~\(\ell^2 ([x])\).  Here \(\Cont_0(X)\)
acts by pointwise multiplication.  If \(U\in S\) is a bisection
of~\(H\), then \(f\in\Cont_0(U)\) acts on \(\ell^2([x])\) by
\(\pi_U(f) \xi(\rg(\gamma)) \defeq f(\gamma)\cdot \xi(\s(\gamma))\)
for all \(\gamma\in U\) with \(\s(\gamma)\in[x]\), and
\(\pi_U(f) \xi(y)=0\) for \(y\in [x]\setminus \rg(U)\).  Simple
computations show that \(\pi_U(f)^*=\pi_{U^*}(f^*)\),
\(\pi_U(f) \pi_V(g) = \pi_{UV}(f*g)\) for \(U,V\in S\),
\(f\in\Cont_0(U)\), \(g\in\Cont_0(V)\), and \(\pi_U(f) = \pi_V(f)\)
for \(U\subseteq V\) and \(f\in\Cont_0(U) \subseteq \Cont_0(V)\).
Hence the maps~\(\pi_U\) for \(U\in S\) form a representation of the
action of~\(S\) on~\(\Cont_0(X)\).  Then they induce a non-degenerate
representation~\(\pi_{[x]}\) of \(\Cont_0(X)\rtimes S\) on
\(\Bound(\ell^2[x])\).  Let
\(\pi \defeq \bigoplus_{x\in X} \pi_{[x]}\) be the direct sum of all
these representations.

\begin{lemma}
  \label{lem:orbit_representation}
  The representation~\(\pi\) is faithful on \(\Cont_0(X)\).  If it
  factors through \(\Cst_\ess(H,\A)\), then~\(H\) is topologically
  free.
\end{lemma}

\begin{proof}
  The algebra \(\Cont_0(X)\) acts on \(\ell^2([x])\) by pointwise
  multiplication.  These representations are faithful when we sum
  over all \(x\in X\).  Assume that~\(H\) is not topologically free.
  That is, there is a non-empty bisection
  \(U\subseteq H\setminus \overline{X}\) with \(\rg|_U = \s|_U\).
  Since \(U\cap \overline{X}=0\), it follows that \(EL(f)=0\) for
  all \(f\in \Cont_0(U)\).  Choose any non-zero \(f\in \Contc(U)\).
  Define \(f_0\in \Contc(\s(U)) \subseteq \Cont_0(X)\) by
  \(f_0(\s(\gamma))\defeq f(\gamma)\) for all \(\gamma\in U\).
  Since \(\rg|_U = \s|_U\), both \(f\in \Cont_0(U)\) and
  \(f_0\in\Cont_0(X)\) act
  on \(\ell^2([x])\) by pointwise multiplication with the same
  function~\(f_0\), for all \(x\in X\).  Hence
  \(f-f_0\in \ker \pi\).  But \(EL(f-f_0)= -f_0 \neq0\).  Hence
  \(f-f_0\notin \Null_{EL}\).
\end{proof}

\begin{theorem}
  \label{thm:detection_in_etale_groupoids}
  Let~\(H\) be an \'etale groupoid with locally compact, Hausdorff
  unit space~\(X\).  The inclusion
  \(\Cont_0(X) \subseteq \Cont_0(X)\rtimes S = \Cst(H)\) has the
  generalised intersection property with hidden ideal~\(\Null_{EL}\)
  if and only if~\(H\) is topologically free.
\end{theorem}

\begin{proof}
  The dual groupoid for the \(S\)\nb-action on~\(\Cont_0(X)\) is
  simply~\(H\), and \(\Cont_0(X)\) is of Type~I.  So
  Theorem~\ref{the:aperiodic_top_non-trivial} (or
  Theorem~\ref{the:aperiodic_vs_top_free_groupoid}) shows that the
  action of~\(S\) on \(\Cont_0(X)\) is aperiodic if and only
  if~\(H\) is topologically free.  Aperiodicity implies that
  \(\Cont_0(X) \subseteq \Cont_0(X)\rtimes S = \Cst(H)\) has the
  generalised intersection property with hidden
  ideal~\(\Null_{EL}\).  Conversely, if~\(H\) is not topologically
  free, then Lemma~\ref{lem:orbit_representation} exhibits an ideal
  \(\ker \pi\) with \(\ker \pi \cap A = 0\), but
  \(\ker \pi \not\subseteq \Null_{EL}\).
\end{proof}

The special case of Theorem~\ref{thm:detection_in_etale_groupoids}
for transformation groups goes back to Kawamura--Tomiyama and
Archbold--Spielberg (see
\cite{Kawamura-Tomiyama:Properties_dynamical}*{Theorem~4.1} and
\cite{Archbold-Spielberg:Topologically_free}*{Theorem~2}).  The
special case of Hausdorff groupoids with \(\Cst_\red(H)\) instead of
\(\Cst_\ess(H)\) is similar to
\cite{Brown-Clark-Farthing-Sims:Simplicity}*{Proposition~5.5};
replacing ``topologically principal'' by ``topologically free''
allows us to remove the second countability assumption, and
replacing the reduced by the essential crossed product allows to
remove the Hausdorffness assumption.

\end{document}